\documentclass[reqno]{amsart}

\usepackage{tikz-cd}
\usepackage{enumitem}
\usepackage{fullpage,dsfont}
\usepackage{leftindex}
\usepackage{tabularx}
\usepackage{mathtools,leftindex,tensor}
\usepackage{mathabx}   
\usepackage{graphicx}  

\usepackage{hyperref} 

\usepackage{parskip}
\makeatletter 
\def\thm@space@setup{%
 \thm@preskip=\parskip \thm@postskip=0pt
}
\def\th@remark{%
  \thm@headfont{\itshape}%
  \normalfont 
  \thm@preskip\parskip \thm@postskip=0pt
}
\makeatother

\usepackage[nobysame,alphabetic,initials,msc-links]{amsrefs}

\usepackage[final]{pdfpages}

\usepackage{multirow}

\usepackage{stmaryrd}

\usepackage{array}

\DefineSimpleKey{bib}{how}
\DefineSimpleKey{bib}{mrclass}
\DefineSimpleKey{bib}{mrnumber}
\DefineSimpleKey{bib}{fjournal}
\DefineSimpleKey{bib}{mrreviewer}

\renewcommand{\PrintDOI}[1]{%
  \href{http://dx.doi.org/#1}{{\tt DOI:#1}}%
}
\renewcommand{\eprint}[1]{#1}
\BibSpec{book}{%
    +{}  {\PrintPrimary}                {transition}
    +{.} { \PrintDate}                  {date}
    +{.} { \textit}                     {title}
    +{.} { }                            {part}
    +{:} { \textit}                     {subtitle}
    +{,} { \PrintEdition}               {edition}
    +{}  { \PrintEditorsB}              {editor}
    +{,} { \PrintTranslatorsC}          {translator}
    +{,} { \PrintContributions}         {contribution}
    +{,} { }                            {series}
    +{,} { \voltext}                    {volume}
    +{,} { }                            {publisher}
    +{,} { }                            {organization}
    +{,} { }                            {address}
    +{,} { }                            {status}
    +{,} { \PrintDOI}                   {doi}
    +{,} { \PrintISBNs}                 {isbn}
    +{}  { \parenthesize}               {language}
    +{}  { \PrintTranslation}           {translation}
    +{;} { \PrintReprint}               {reprint}
    +{.} { }                            {note}
    +{.} {}                             {transition}
    +{}  {\SentenceSpace \PrintReviews} {review}
}
\BibSpec{article}{%
    +{}  {\PrintAuthors}                {author}
    +{,} { \textit}                     {title}
    +{.} { }                            {part}
    +{:} { \textit}                     {subtitle}
    +{,} { \PrintContributions}         {contribution}
    +{.} { \PrintPartials}              {partial}
    +{,} { }                            {journal}
    +{}  { \textbf}                     {volume}
    +{}  { \PrintDatePV}                {date}
    +{,} { \issuetext}                  {number}
    +{,} { \eprintpages}                {pages}
    +{,} { }                            {status}
    +{,} { \PrintDOI}                   {doi}
    +{,} { \eprint}        {eprint}
    +{}  { \parenthesize}               {language}
    +{}  { \PrintTranslation}           {translation}
    +{;} { \PrintReprint}               {reprint}
    +{.} { }                            {note}
    +{.} {}                             {transition}
    +{}  {\SentenceSpace \PrintReviews} {review}
}
\BibSpec{collection.article}{%
    +{}  {\PrintAuthors}                {author}
    +{,} { \textit}                     {title}
    +{.} { }                            {part}
    +{:} { \textit}                     {subtitle}
    +{,} { \PrintContributions}         {contribution}
    +{,} { \PrintConference}            {conference}
    +{}  {\PrintBook}                   {book}
    +{,} { }                            {booktitle}
    +{,} { \PrintDateB}                 {date}
    +{,} { pp.~}                        {pages}
    +{,} { }                            {publisher}
    +{,} { }                            {organization}
    +{,} { }                            {address}
    +{,} { }                            {status}
    +{,} { \PrintDOI}                   {doi}
    +{,} { \eprint}        {eprint}
    +{}  { \parenthesize}               {language}
    +{}  { \PrintTranslation}           {translation}
    +{;} { \PrintReprint}               {reprint}
    +{.} { }                            {note}
    +{.} {}                             {transition}
    +{}  {\SentenceSpace \PrintReviews} {review}
}
\BibSpec{misc}{%
  +{}{\PrintAuthors}  {author}
  +{,}{ \textit}      {title}
  +{.}{ }             {how}
  +{}{ \parenthesize} {date}
  +{,} { available at \eprint}        {eprint}
  +{,}{ available at \url}{url}
  +{,}{ }             {note}
  +{.}{}              {transition}
}
\usepackage{amssymb, amsfonts, amsxtra, amsmath}
\usepackage{mathrsfs}
\usepackage{mathdots}
\usepackage{wasysym}
\usepackage[all]{xy}
\usepackage{bbm}
\usepackage{calc}
\usepackage{accents}

\usepackage[bbgreekl]{mathbbol}			

\usepackage{stackengine} 

\numberwithin{equation}{section}

\DeclareSymbolFontAlphabet{\mathbb}{AMSb}	
\DeclareSymbolFontAlphabet{\mathbbl}{bbold}	

\newtheorem{Theorem}{Theorem}[section]
\newtheorem*{Theorem*}{Theorem}
\newtheorem{Def}[Theorem]{Definition}
\newtheorem*{Def*}{Def}
\newtheorem{Lem}[Theorem]{Lemma}
\newtheorem{Prop}[Theorem]{Proposition}
\newtheorem{Cor}[Theorem]{Corollary}

\newtheorem{Rem}[Theorem]{Remark}

\newtheorem{Exa}[Theorem]{Example}

\newcommand\XGX{\X\times_\G \bar{\X}}
\newcommand\XrX{\X\rtimes\G \ltimes\bar{\X}}

\mathchardef\mhyph="2D

\DeclareMathOperator{\id}{\mathrm{id}}
\DeclareMathOperator{\Rep}{\mathrm{Rep}}
\DeclareMathOperator{\Irr}{\mathrm{Irr}}

\newcommand{\op}{\mathrm{op}}
\newcommand{\msU}{\mathscr{U}}

\newcommand{\Mor}{\operatorname{Mor}}
\renewcommand{\H}{\mathbb{H}}
\newcommand{\mcG}{\mathcal{G}}
\newcommand{\mcH}{\mathcal{H}}
\newcommand{\mcK}{\mathcal{K}}
\newcommand{\mcO}{\mathcal{O}}

\newcommand{\Ww}{\mathds{W}}

\newcommand{\C}{\mathbb{C}}
\newcommand{\G}{\mathbb{G}}
\newcommand{\R}{\mathbb{R}}
\newcommand{\X}{\mathbb{X}}

\newcommand{\mH}{\mathcal{H}}

\newcommand{\Corr}{\mathrm{Corr}}
\newcommand{\ovot}{\bar{\otimes}}

\allowdisplaybreaks

\begin{document}

\title{Quantum hypergroups arising from ergodic coactions}
\author{Joeri De Ro}
\address{Institute of Mathematics of the Polish Academy of Sciences}
\email{jdero@impan.pl}

\begin{abstract} 
    Given a compact quantum group $\G$ and an ergodic action $ L^\infty(\X)\stackrel{\alpha}\curvearrowleft \G$ with algebraic core $\mcO(\X)$, we show that the unital $*$-algebra $\mcO(\XGX):= \mcO(\X)\square\overline{\mcO(\X)}$ carries the structure of an algebraic compact quantum hypergroup. This $*$-algebra admits two (generally distinct) $C^*$-algebra completions (`reduced' and `universal'), both carrying the structure of a $C^*$-algebraic compact quantum hypergroup. This provides a large class of new examples of (analytical) compact quantum hypergroups. We provide characterizations of coamenability for these compact quantum hypergroups, making use of the theory of equivariant correspondences.
\end{abstract}

\maketitle

\section{Introduction and preliminaries}

Compact quantum hypergroups are objects that simultaneously generalize classical (compact) hypergroups (see the book \cite{HB95} and many references therein) and compact quantum groups \cites{Wor87a, Wor87b, Wor98}. Comparing with the definition of a compact quantum group, the main difference is that the coassociative map $\Delta$ of the quantum hypergroup is no longer required to be multiplicative. In the analytical setting (i.e. $C^*$-algebraic or $W^*$-algebraic framework), it is most natural to ask that $\Delta$ is a unital completely positive (= ucp) map instead. Compact quantum hypergroups have been systematically studied from the purely algebraic point of view in \cites{DVD11a, DVD11b} and from the $C^*$-algebraic point of view in \cite{CV99}.

The first non-trivial examples of (analytical) compact quantum hypergroups that were constructed and investigated are the double coset spaces $\H\backslash \G/\H$, arising from an inclusion $\H\le \G$ of compact quantum groups \cites{CV92, Vai95a, Vai95b, PV99, CV99, FS00}. This was later generalized and studied in the case where $\H$ is a compact quasi-subgroup \cite{KS20} of $\G$ \cites{FS09, Zh20,DFW21}. In \cite{Ka01}, a construction of $C^*$-algebraic compact quantum hypergroups is given via certain conditional expectations, which in particular leads to the so-called Delsart hypergroups.
Note also that in the recent paper \cite{BGP26}, it was shown that an inclusion of unital simple $C^*$-algebras together with a conditional expectation of finite Watani-index leads to a $C^*$-algebraic quantum hypergroup-like structure.

In the present paper, we provide a large class of new examples of compact quantum hypergroups. More specifically, we will construct a compact quantum hypergroup from an arbitrary ergodic action $ L^\infty(\X)\stackrel{\alpha}\curvearrowleft \G$, where $\G$ is a compact quantum group. This construction is inspired by the technique where a new compact quantum group is obtained by `reflection around a Galois object' \cites{DC09a, DC11,DC17}. If $\H$ is a compact (quasi-)subgroup of $\G$, our construction for the action $ L^\infty(\H\backslash \G)\stackrel{\Delta_\G}\curvearrowleft \G$ recovers the known examples of compact quantum hypergroups arising from compact (quasi-)subgroups that were discussed above. However, our construction also deals with the strictly more general class of coideal von Neumann algebras.

Here is the concrete plan for this paper. In \emph{Section \ref{matrixcoefficients}}, we consider an ergodic (right) action $L^\infty(\X)\stackrel{\alpha}\curvearrowleft \G$ with algebraic core $\mcO(\X)$.
Write $L^\infty(\bar{\X}):= \overline{L^\infty(\X)}$ for the conjugate von Neumann algebra, which carries a natural `mirrored' (left) action $\G\stackrel{\bar{\alpha}}\curvearrowright L^\infty(\bar{\X})$. The main object of interest for this paper is the unital $*$-algebra
$$\mcO(\XGX):= \mcO(\X)\stackrel{\G}\square\mcO(\bar{\X})=\{z\in \mcO(\X)\odot \mcO(\bar{\X}): (\alpha\odot \id)(z)= (\id \odot \bar{\alpha})(z)\},$$
and various completed versions ($C^*$-algebraic and von Neumann algebraic) of it. We show that $\mcO(\XGX)$ admits a linearly generating set of `matrix coefficients' $Z_\pi(\mu, \nu)$ indexed by $\pi \in \Irr(\G)$ and $\mu, \nu\in \mcG_\pi$, where $\mcG_\pi$ is a canonical finite-dimensional Hilbert space associated to the action $ L^\infty(\X)\stackrel{\alpha}\curvearrowleft \G$. Using these matrix coefficients, we can then define the linear maps
\begin{align*}
    &\Delta_{\XGX}: \mcO(\XGX)\to \mcO(\XGX)\odot \mcO(\XGX)\\
    &\epsilon_{\XGX}: \mcO(\XGX)\to \C, \quad S_{\XGX}: \mcO(\XGX)\to \mcO(\XGX)
\end{align*}
via the expected formulas
\begin{align*}
    \Delta_{\XGX}(Z_\pi(\mu, \nu))= \sum_{j=1}^{m_\pi} &Z_\pi(\mu, f_j^\pi)\otimes Z_\pi(f_j^\pi, \nu), \quad \epsilon_{\XGX}(Z_\pi(\mu, \nu))= \langle \mu, \nu\rangle, \quad S_{\XGX}(Z_\pi(\mu, \nu))= Z_\pi(\nu, \mu)^*,
\end{align*}
where $\{f_j^\pi\}_{j=1}^{m_\pi}$ is an orthonormal basis for $\mcG_\pi$ and $\pi \in \Irr(\G)$. These maps endow $\mcO(\XGX)$ with the structure of an \emph{algebraic compact quantum hypergroup} in the sense of \cite{DVD11a} (see Theorem \ref{hypergroup}). This construction generalizes (and is strongly inspired by) the well-known case where the action $C(\X)\curvearrowleft \G$ is free, in which case we obtain a genuine compact quantum group as opposed to a compact quantum hypergroup. 

As is well-known, every algebraic compact quantum group admits both a universal $C^*$-algebraic and a reduced $C^*$-algebraic version. An analogous statement is also true for the algebraic compact quantum hypergroup $\mcO(\XGX)$, but the main techniques from the theory of compact quantum groups no longer apply. The primary analytical difficulty is extending $\Delta_{\XGX}: \mcO(\XGX)\to \mcO(\XGX)\odot \mcO(\XGX)$ to a ucp map on the relevant completions. Let us briefly explain the difficulties and strategies, worked out in detail in \emph{Section \ref{universalversion}} (universal version) and \emph{Section \ref{reducedversion}} (reduced version). 

\textbf{Universal version.} In general, a universal $C^*$-envelope for $\mcO(\XGX)$ may not exist \cite{DCDR25}*{Remark 2.10}. Rather, one embeds $\mcO(\XGX)$ in the double crossed product $*$-algebra $\mcO(\XrX)$ \cites{AS21,DCDR25}, which is a $*$-algebra that admits a universal $C^*$-envelope $C^u_0(\XrX)$ in which $\mcO(\XrX)$ embeds. Through the embedding $\mcO(\XGX)\hookrightarrow C_0^u(\XrX)$, the $*$-algebra $\mcO(\XGX)$ obtains a `universal' $C^*$-norm. The corresponding $C^*$-algebra completion will be denoted by $C^u(\XGX)$. We then show that $C^u(\XGX)$ admits the structure of a $C^*$-algebraic compact quantum hypergroup in the sense of \cite{CV99}. Our main tool to do this is the theory of equivariant correspondences developed in \cites{DCDR24, DCDR25}, which enters the picture through the isomorphism
$\Rep_*(C^u_0(\XrX))\cong \Corr^\G(L^\infty(\X), L^\infty(\X))$
of $W^*$-categories \cite{DCDR25}*{Proposition 2.14}. 

\textbf{Reduced version.} Through the inclusion $\mcO(\XGX)\subseteq L^\infty(\X)\ovot L^\infty(\bar{\X})\subseteq B(L^2(\X)\otimes L^2(\bar{\X}))$, the $*$-algebra $\mcO(\XGX)$ obtains a natural `reduced' $C^*$-norm.
Starting from an algebraic compact quantum group, one extends the algebraic coproduct to its reduced $C^*$-or $W^*$-algebraic completion by first introducing the Kac-Takesaki operator and then showing that it implements the coproduct. A strategy like this no longer works in the setting of quantum hypergroups, since the coassociative map is not multiplicative in general. Rather, we will make use of the Galois map associated to the action $L^\infty(\X)\stackrel{\alpha}\curvearrowleft \G$ \cites{DC09a,DC11} to construct a normal ucp coassociative map
\begin{equation}\label{comulti}\Delta_{\XGX}^r: L^\infty(\XGX)\to L^\infty(\XGX)\ovot L^\infty(\XGX),\end{equation}
where $L^\infty(\XGX):= L^\infty(\X)\stackrel{\G}\square L^\infty(\bar{\X})= \{z \in L^\infty(\X)\ovot L^\infty(\bar{\X}): (\alpha\otimes \id)(z)= (\id \otimes \bar{\alpha})(z)\}.$
It is then shown that $\Delta_{\XGX}^r$ is an extension of $\Delta_{\XGX}: \mcO(\XGX)\to \mcO(\XGX)\odot \mcO(\XGX)$.
In fact, the construction of \eqref{comulti} is carried out in the more general setting where $\G$ is a locally compact quantum group and $L^\infty(\X)\stackrel{\alpha}\curvearrowleft \G$ is an ergodic, integrable action (see Theorem \ref{mainsec3}). The failure of multiplicativity of \eqref{comulti} reflects the fact that the Galois map need not be unitary.

In \emph{Section \ref{examples}}, we work out two examples of general nature. In the first example, we look at the general class of coideal von Neumann algebras $L^\infty(\X)\curvearrowleft \G$, so that $\X= \H \backslash \G$ where $\H$ is a `generalized quantum subgroup of $\G$'. In that case, we prove that $\mcO(\XGX)\cong \mcO(\H\backslash \G/\H):=\mcO(\H\backslash \G)\cap R_\G(\mcO(\H\backslash \G))$. In particular, our discussion of this example illuminates how to turn the `double coset space' $\H\backslash \G/\H$ into a compact quantum hypergroup beyond the case where $\H$ is a compact (quasi-)subgroup of $\G$. In the second example, we consider the natural ergodic action $L^\infty(\G)\curvearrowleft \G^{\op}\times \G$. We make a careful analysis of the spectral data of this action (Proposition \ref{spectraldata}). The  $*$-algebra of the associated algebraic compact quantum hypergroup can be identified with the fusion $*$-algebra $\mcO(\operatorname{Fus}(\G))$. Under this identification, we recover the natural compact quantum hypergroup structure on $\mcO(\operatorname{Fus}(\G))$ for which the (normalized) irreducible characters become group-like. This example of an algebraic compact quantum hypergroup is certainly not new. Rather, the interesting part is that the quantum hypergroup structure carries over to its relevant completions, which appears to be non-obvious if one does not pass through the dynamical formalism developed in this paper. 

As discussed above, the space $\mcO(\XGX)$ has two natural $C^*$-norms: the universal norm $\|\cdot\|_u$, coming from the embedding $\mcO(\XGX)\hookrightarrow C_0^u(\XrX)$, and the reduced norm $\|\cdot\|_r$, coming from the inclusion $\mcO(\XGX)\subseteq L^\infty(\X)\ovot L^\infty(\bar{\X})\subseteq B(L^2(\X)\otimes L^2(\bar{\X}))$. Comparing these two norms leads to a natural notion of \emph{coamenability} for the compact quantum hypergroup $\XGX$. This is discussed in detail in \emph{Section \ref{coamenable}}, where it is proven that $\|\cdot\|_r= \|\cdot\|_u$ on $\mcO(\XGX)$ if and only if the counit $\epsilon_{\XGX}: \mcO(\XGX)\to \C$ is $\|\cdot\|_r$-bounded if and only if $L^\infty(\X)$ is $\G$-injective (Theorem \ref{counitbounded}) . Recalling that $L^\infty(\G)$ is $\G$-injective if and only if $\hat{\G}$ is amenable \cites{Cr17, DH24, DR26}, this can be seen as a dynamical generalization of Tomatsu's celebrated result on the equivalence of amenability and strong amenability for discrete quantum groups \cite{Tom06}.

\subsection{Conventions and notations} All vector spaces in this paper are defined over the field $\C$. The symbol $\odot$ denotes the (algebraic) tensor product (over $\C$). We assume that inner products of pre-Hilbert spaces are anti-linear in the first variable. Given a subset $S$ of a normed linear space $V$, we write $[S]$ for the norm-closure of the linear span of $S$. More generally, if $(V, \tau)$ is a topological vector space and $S\subseteq V$, we write $[S]^\tau$ for the $\tau$-closure of the linear span of $S$ inside $V$. Given $C^*$-algebras $C,D$, their minimal tensor product is denoted by $C\otimes D$. The multiplier $C^*$-algebra of $C$ is denoted by $M(C)$. Given von Neumann algebras $A$ and $B$, their von Neumann algebra tensor product is denoted by $A \ovot B$.  Given Hilbert spaces $\mcH, \mcK$, their Hilbert space tensor product is written by $\mcH\otimes \mcK$, and we write $\Sigma= \Sigma_{\mcH, \mcK}: \mcH\otimes \mcK\to \mcK\otimes \mcH$ for the switch map. 

\subsection{Locally compact quantum groups} \cites{KV00, KV03, VV03} A \emph{von Neumann bialgebra} is a pair $(M, \Delta)$ where $M$ is a von Neumann algebra and $\Delta: M \to M \ovot M$ is a unital, normal, isometric $*$-homomorphism such that $(\Delta \otimes \id)\circ \Delta = (\id \otimes \Delta)\circ \Delta$. A \emph{locally compact quantum group} $\G$ is a von Neumann bialgebra $(L^\infty(\G), \Delta_\G)$ for which there exist normal, semifinite, faithful weights $\varphi_\G, \psi_\G: L^\infty(\G)_+\to [0, \infty]$ such that $(\id\otimes \varphi_\G)\Delta_\G(x)= \varphi_\G(x)1$ for all $x\in \mathscr{M}_\varphi^+$ and $(\psi_\G\otimes \id)\Delta_\G(x)= \psi_\G(x)1$ for all $x\in \mathscr{M}_\psi^+$. These are called the \emph{left Haar weight} and the \emph{right Haar weight}, and they can be shown to be unique up to a non-zero positive scalar multiple. We denote the predual of the von Neumann algebra $L^\infty(\G)$ by $L^1(\G)$, so that $L^1(\G)^*\cong L^\infty(\G)$.

We denote the GNS-Hilbert space of $\varphi_\G$ by $L^2(\G)$. We use it to view $L^\infty(\G)\subseteq B(L^2(\G))$. One can canonically identify the GNS-Hilbert space of $\psi_\G$ with $L^2(\G)$. Using the GNS-maps
$\Lambda_{\varphi_\G}: \mathscr{N}_{\varphi_\G}\to L^2(\G)$ and  $\Lambda_{\psi_\G}: \mathscr{N}_{\psi_\G}\to L^2(\G)$
we can then define the unitaries 
$V_\G\in B(L^2(\G))\ovot L^ \infty(\G)$ and $W_\G \in L^\infty(\G)\ovot B(L^2(\G))$, 
called respectively \emph{right} and \emph{left} regular unitary representation of $\G$. They are uniquely determined by
\begin{align*}
(\id\otimes \omega)(V_\G) \Lambda_{\psi_\G}(x) &= \Lambda_{\psi_\G}((\id\otimes \omega)\Delta_\G(x)),
\qquad \omega \in L^1(\G),\quad x\in \mathscr{N}_{\psi_\G},\\
(\omega \otimes \id)(W_{\G}^*)\Lambda_{\varphi_\G}(x) &= \Lambda_{\varphi_\G}((\omega\otimes \id)\Delta_\G(x)),\qquad \omega \in L^1(\G),\quad x\in \mathscr{N}_{\varphi_\G}.
\end{align*} They are \emph{multiplicative unitaries} meaning that
\[
V_{\G,12}V_{\G,13}V_{\G,23} = V_{\G,23}V_{\G,12},\qquad W_{\G,12}W_{\G,13}W_{\G,23}= W_{\G,23}W_{\G,12},
\]
and they implement the coproduct of $L^\infty(\G)$ in the sense that
$$\label{EqComultImpl}
W_\G^*(1\otimes x)W_\G = \Delta_\G(x) = V_\G(x\otimes 1)V_\G^*,\qquad x\in L^\infty(\G).$$
Moreover, we have
$C_0(\G):=[(\omega\otimes \id)(V_\G) \mid \omega \in B(L^2(\G))_*] = [(\id\otimes \omega)(W_\G) \mid \omega \in B(L^2(\G))_*],$
which is a $\sigma$-weakly dense C$^*$-subalgebra of $L^\infty(\G)$. Then $\Delta_\G(C_0(\G))\subseteq M(C_0(\G)\otimes C_0(\G))$. 

We also define the von Neumann algebra
    $L^\infty(\hat{\G}) := [(\omega\otimes \id)(W_\G) \mid \omega \in L^1(\G)]^{\sigma\textrm{-weak}}$ together with the coproduct
$\Delta_{\hat{\G}}: L^\infty(\hat{\G})\to L^\infty(\hat{\G})\ovot L^\infty(\hat{\G})$ given by $\Delta_{\hat{\G}}(\hat{x})= \Sigma W_\G(\hat{x}\otimes 1)W_\G^*\Sigma.$
The pair $(L^\infty(\hat{\G}), \Delta_{\hat{\G}})$ then defines the dual locally compact quantum group $\hat{\G}$. The left invariant weight $\varphi_{\hat{\G}}$ is constructed in a way that ensures the canonical identification $L^2(\G)= L^2(\hat{\G})$. Pontryagin duality then holds: $\hat{\hat{\G}}= \G$.

We write $J_\G$ for the modular conjugation on $L^2(\G)$ associated to the weight $\varphi_\G$ and we write $J_{\hat{\G}}$ for the modular conjugation on $L^2(\G)$ associated with the weight $\varphi_{\hat{\G}}$. We have 
$J_{\hat{\G}}L^\infty(\G)J_{\hat{\G}}=L^\infty(\G),$
so we obtain the anti-$*$-homomorphism
$R_\G: L^\infty(\G)\to L^\infty(\G): x \mapsto J_{\hat{\G}}x^*J_{\hat{\G}}.$
We call $R_\G$ the \emph{unitary antipode} of $\G$. There is a canonical unimodular complex number $c\in \mathbb{C}$ (cf.\ \cite{KV03}*{Corollary 2.12}) satisfying 
$c J_{\hat{\G}}J_\G=\overline{c}J_\G J_{\hat{\G}}.$
We write $u_\G:= c J_{\hat{\G}} J_\G$ for the associated self-adjoint unitary. The following identities will be useful:
\begin{align*}
    (J_{\hat{\G}}\otimes J_\G)W_\G(J_{\hat{\G}}\otimes J_\G)=W_\G^*, \quad (J_\G\otimes J_{\hat{\G}})V_\G(J_\G\otimes J_{\hat{\G}})=V_\G^*, \quad (u_\G\otimes 1)V_\G(u_\G\otimes 1)=W_{\G,21}.
\end{align*} 

\subsection{Actions of locally compact quantum groups and equivariant correspondences}\label{equivariant correspondences}

Given a von Neumann algebra $A$, consider its standard form $(L^2(a), \pi_A, J_A, L^2(A)_+)$. We then also consider the standard anti-$*$-representation
$\rho_A: A\to B(L^2(A)): a \mapsto J_A\pi_A(a^*)J_A$, so that $\pi_A(A)'= \rho_A(A)$. Often, we will suppress $\pi_A$ from the notation and identify $A\subseteq B(L^2(A))$.

Let $\G$ be a locally compact quantum group. A (right) unitary $\G$-representation on the Hilbert space $\mcH$ consists of a unitary element $U\in B(\mcH)\ovot L^\infty(\G)$ such that $(\id \otimes \Delta_\G)(U) = U_{12}U_{13}$. Given $\omega \in L^1(\G)$, let us then write $U(\omega):= (\id \otimes \omega)(U)\in B(\mcH)$.

 A (right) $\G$-$W^*$-algebra consists of a pair $(A, \alpha)$ such that $A$ is a von Neumann algebra and $\alpha: A \to A\ovot L^\infty(\G)$ is a unital, isometric, normal $*$-homomorphism satisfying $(\id \otimes \Delta_\G)\alpha= (\alpha\otimes \id)\alpha$. We will often employ the more intuitive notation $A\stackrel{\alpha}\curvearrowleft \G$ and refer to $\alpha$ as a $\G$-action on $A$. 
 
In \cite{Vae01}, it is proven that every $\G$-$W^*$-algebra $(A,\alpha)$ induces a canonical unitary $\G$-representation $U_\alpha\in B(L^2(A))\ovot L^\infty(\G)$, called the \emph{(canonical) unitary implementation of $\alpha$}, such that
\begin{align}
    (\pi_A\otimes \id)\alpha(a)&= U_\alpha(\pi_A(a)\otimes 1)U_\alpha^*, \quad a\in A.\\
\label{unitaryflip}(J_A\otimes J_{\hat{\G}})U_\alpha(J_A\otimes J_{\hat{\G}})&= U_\alpha^*.
\end{align}
Note that these two identities imply that $(\rho_A\otimes R_\G)\alpha(a)= U_\alpha^*(\rho_A(a)\otimes  1)U_\alpha$ for all $a\in A$.

If $(A, \alpha)$ is a (right) $\G$-$W^*$-algebra, 
 there is a natural (left) action $\G \stackrel{\bar{\alpha}}\curvearrowright \overline{A}$, where $\overline{A}$ is the conjugate von Neumann algebra (meaning that only scalar multiplication is altered). It is formally given by $\bar{\alpha}(\bar{a})= R_\G(a_{(1)})^* \otimes \overline{a_{(0)}}$, where $\alpha(a)= a_{(0)}\otimes a_{(1)}$ and $a\in A$.
Note also that $\G\stackrel{\alpha'}\curvearrowright A'$ via
$$\alpha'(a')= U_{\alpha,21}^*(1\otimes a')U_{\alpha,21}, \quad a' \in A',$$
where we view $A\subseteq B(L^2(A))$. The $*$-isomorphism
$I: A'\cong \overline{A}: \rho_A(a^*)\mapsto \overline{a}$ is then $\G$-equivariant.

We call $\alpha$  \emph{ergodic} if the space of $\G$-fixed points $A^\alpha:=\{a\in A\mid \alpha(a)= a\otimes 1\}$ is equal to $\C1$. In that case, we will write $A= L^\infty(\X)$, thinking about $\X$ as the underlying \emph{quantum homogeneous space}. We then also write $(\overline{A}, L^2(A),\pi_A, J_A, L^2(A)_+, \rho_A, U_\alpha)= (L^\infty(\bar{\X}),L^2(\X), \pi_\X, J_\X, L^2(\X)_+, \rho_\X, U_\X)$.

Given right $\G$-$W^*$-algebras $(A, \alpha)$ and $(B, \beta)$, a \emph{$\G$-$A$-$B$-correspondence} \cite{DCDR24}*{Definition 0.4} is a quadruplet $(\mcH, \pi, \rho, U)$ where $\mcH$ is a Hilbert space, $\pi: A \to B(\mcH)$ is a unital, normal $*$-representation, $\rho: B \to B(\mcH)$ is a unital, normal anti-$*$-representation and $U\in B(\mcH)\ovot L^\infty(\G)$ is a unitary $\G$-representation such that
$$\pi(a)\rho(b)= \rho(b)\pi(a), \quad (\pi\otimes \id)\alpha(a)= U(\pi(a)\otimes 1)U^*, \quad (\rho\otimes R_\G)\beta(b)= U^*(\rho(b)\otimes 1)U, \quad a\in A, \quad b \in B.$$
In that case, we write  $\mcH = (\mcH, \pi, \rho, U)\in \Corr^\G(A,B)$.

Let us consider two important examples, of primary interest for this paper \cite{DCDR24}*{Examples 2.2 \& 2.5}:
\begin{enumerate}
    \item The trivial $\G$-$A$-$A$-correspondence is $L^2(A)= (L^2(A), \pi_A, \rho_A, U_\alpha)$.
    \item\vspace{-1.5mm} The \emph{coarse (= regular) $\G$-$A$-$A$-correspondence} is given by the quadruplet
\begin{equation}\label{coarse}
    C_A^\G=(L^2(A)\otimes L^2(\G) \otimes L^2(A), a \mapsto (\pi_A\otimes \id)\alpha(a)\otimes 1, a \mapsto 1\otimes (\rho_\G \otimes \rho_A)(\alpha^{\op}(a)), V_{\G,24}).
\end{equation}
\end{enumerate}

If $\mcH, \mcH'\in \Corr^\G(A,B)$, we write ${}_A\mathscr{L}_B^\G(\mcH, \mcH')$ for the space of bounded linear operators $x: \mcH \to \mcH'$ such that $x\pi(a)= \pi'(a)x, x\rho(b)= \rho'(b)x$ and $(x\otimes 1)U= U'(x\otimes 1)$ for all $a\in A$ and all $b\in B$. In this way, $\Corr^\G(A,B)$ becomes a $W^*$-category. Note also that taking $\G$ to be the trivial group, we simply recover the category $\Corr(A,B)$ of $A$-$B$-correspondence as introduced by Connes \cite{Con80}. When $\G$ is trivial or/and one of the von Neumann algebras $A,B$ is $\C$, we leave it out of the notation. For example, $\mathscr{L}_B(\mcH, \mcH')= \{x\in B(\mcH, \mcH')\mid  \forall b\in B: x\rho(b)= \rho'(b)x\}$.

There is a natural operation
$\boxtimes_B: \Corr^\G(A,B)\times \Corr^\G(B, C)\to \Corr^\G(A,C)$ \cite{DCDR24}*{Section 5.2}.
More concretely, if $\mcH= (\mcH, \pi_\mcH, \rho_\mcH, U_\mcH)\in \Corr^\G(A,B)$ and $\mcK= (\mcK, \pi_\mcK, \rho_\mcK, U_\mcK)\in \Corr^\G(B,C)$, then $\mathscr{L}_B(L^2(B), \mcH)\curvearrowleft B$ via $xb:= x\pi_B(b)$ and $B\curvearrowright \mcK$ via $b\xi:= \pi_\mcK(b)\xi$ where $b\in B, x\in \mathscr{L}_B(L^2(B), \mcH)$ and $\xi \in \mcK$. It thus makes sense to form the (algebraic) balanced tensor product $\mathscr{L}_B(L^2(B), \mcH)\odot_B \mcK$, and we endow it with the semi-inner product uniquely determined by
$$\langle x\otimes_B \xi, x'\otimes_B \xi'\rangle:= \langle \xi, \pi_\mcK(\langle x, x'\rangle_B)\xi'\rangle, \quad x, x'\in \mathscr{L}_B(L^2(B), \mcH), \quad \xi, \xi'\in \mcK,$$
where $\langle x,x'\rangle_B$ is the unique element of $B$ satisfying $\pi_B(\langle x,x'\rangle_B)= x^*x'$. By separation-completion, we obtain the Hilbert space $\mcH\boxtimes_B \mcK$. It carries a natural $A$-$C$-correspondence structure, given by
$$\pi_\boxtimes(a)\rho_\boxtimes(c) (x\otimes_B \xi):= \pi_\mcH(a)x\otimes_B \rho_\mcK(c)\xi, \quad a\in A,\quad c\in C, \quad x\in \mathscr{L}_B(L^2(B), \mcH), \quad \xi \in \mcK.$$
Further, viewing $\mcH\otimes L^2(\G)\in \Corr(A\ovot L^\infty(\G), B \ovot L^\infty(\G))$ and $\mcK\otimes L^2(\G)\in \Corr(B \ovot L^\infty(\G), C\ovot L^\infty(\G))$ in the obvious way, we have $$\mathscr{L}_{B\ovot L^\infty(\G)}(L^2(B\ovot L^\infty(\G)), \mcH\otimes L^2(\G))= \mathscr{L}_{B\ovot L^\infty(\G)}(L^2(B)\otimes L^2(\G), \mcH\otimes L^2(\G))= \mathscr{L}_B(L^2(B),\mcH)\ovot L^\infty(\G),$$
together with the natural identification
\begin{equation}\label{naturalunitary}(\mcH\boxtimes_B \mcK)\otimes L^2(\G)\to ( \mcH\otimes L^2(\G))\boxtimes_{B\ovot L^\infty(\G)} (\mcK\otimes L^2(\G)): (x\otimes_B \xi)\otimes \eta\mapsto (x\otimes 1)\otimes_{B\ovot L^\infty(\G)} (\xi \otimes \eta)\end{equation}
of $A\ovot L^\infty(\G)$-$C \ovot L^\infty(\G)$-correspondences. The unitary 
$$U_\boxtimes: (\mcH\boxtimes_B \mcK)\otimes L^2(\G) \to (\mcH\otimes L^2(\G))\boxtimes_{B\ovot L^\infty(\G)} (\mcK\otimes L^2(\G)) \cong  (\mcH\boxtimes_B \mcK)\otimes L^2(\G)$$
uniquely determined by
$$U_\boxtimes((x\otimes_B\xi)\otimes \eta)= U_\mcH(x\otimes 1)U_\beta^* \otimes_{B\ovot L^\infty(\G)} U_\mcK(\xi \otimes \eta), \quad x\in \mathscr{L}_B(L^2(B), \mcH), \quad \xi \in \mcK, \quad \eta \in L^2(\G),$$
endows the Hilbert space $\mcH\boxtimes_B \mcK$ with a $\G$-representation such that $\mcH\boxtimes_B \mcK := (\mcH\boxtimes_B \mcK, \pi_\boxtimes, \rho_\boxtimes, U_\boxtimes)\in \Corr^\G(A,C)$. We refer to this operation as the \emph{Connes fusion tensor product} of equivariant correspondences. The Connes fusion tensor product is associative in a natural way. 

There is also a natural operation $\Corr^\G(A,B)\to \Corr^\G(B,A)$ \cite{DCDR24}*{Section 5.1}. More precisely, given $\mcH = (\mcH, \pi, \rho, U)\in \Corr^\G(A,B)$, consider the canonical anti-unitary $C_\mcH: \mcH\to \overline{\mcH}$. Then $(\overline{\mcH}, \overline{\pi}, \overline{\rho}, \overline{U})\in \Corr^\G(B,A)$, where 
$$\overline{\pi}(b) = C_\mcH\rho(b^*)C_\mcH^*, \quad \overline{\rho}(a)= C_\mcH \pi(a^*)C_\mcH^*, \quad \overline{U}= (C_\mcH \otimes J_{\hat{\G}})U^*(C_\mcH^*\otimes J_{\hat{\G}}), \quad a\in A, \quad b\in B.$$
We call $\overline{\mcH}= (\overline{\mcH}, \overline{\pi}, \overline{\rho}, \overline{U})$ the conjugate $\G$-$B$-$A$-correspondence associated to $\mcH$.

In \cite{DCDR24}*{Section 3}, the notion of \emph{weak containment for equivariant correspondences} is introduced, simultaneously generalizing the notions of weak containment for von Neumann correspondences and for locally compact quantum group representations. If $\mcH, \mcG \in \Corr^\G(A,B)$ and $\mcH$ is weakly contained in $\mcG$, we write $\mcH\preccurlyeq \mcG$.

A (unital) inclusion $A\subseteq B$ of $\G$-$W^*$-algebras is called (cf.\ \cite{DCDR24}*{Definition 4.1})
\begin{itemize}
    \item \emph{$\G$-amenable} if there exists a $\G$-equivariant conditional expectation $E: B\to A$, and 
    \item\vspace{-1.5mm} \emph{strongly $\G$-amenable} if $L^2(A)\preccurlyeq L^2(B)$ as $\G$-$A$-$A$-correspondences.
\end{itemize} 
As the terminology suggests, strong $\G$-amenability of the $\G$-equivariant inclusion $A\subseteq B$ implies its $\G$-amenability \cite{DCDR24}*{Theorem 4.3}. The converse of this is unknown to be true, as is seen by considering the inclusion $\C\subseteq L^\infty(\G)$. In that case, one recovers the longstanding open problem if the notions of amenability and strong amenability for a locally compact quantum group $\G$ coincide. However, $\G$-amenability of the inclusion $A\subseteq B$ turns out to be equivalent with strong $\G$-amenability of the inclusion $A\subseteq B$ if $\G$ is compact/discrete and $B$ is $\sigma$-finite \cite{DCDR25}*{Theorem 3.1 \& Theorem 3.5}.

We call $(A, \alpha)$ \emph{(strongly) $\G$-injective} \cite{DCDR24}*{Definition 6.6}\footnote{In \cite{DCDR24}, the terminology (strong) $\G$-$W^*$-amenability was used instead. In the meantime, the connection with $\G$-equivariant injectivity has been completely clarified. Therefore, we prefer to use other terminology instead.} if the inclusion $\pi_A(A) \ovot \C1 \subseteq B(L^2(A))\ovot L^\infty(\G)$ is (strongly) $\G$-amenable, where the von Neumann algebra $B(L^2(A)) \ovot L^\infty(\G)$ carries the $\G$-action
$$B(L^2(A))\ovot L^\infty(\G)\to B(L^2(A))\ovot L^\infty(\G)\ovot L^\infty(\G): z \mapsto U_{\alpha,13}V_{\G,23}z_{12}V_{\G,23}^* U_{\alpha,13}^*.$$ Equivalently, $(A, \alpha)$ is strongly $\G$-injective if and only if $L^2(A)\preccurlyeq C_A^\G$ as $\G$-$A$-$A$-correspondences. In \cite{DR26}*{Proposition 3.13}, it was shown that $\G$-equivariant injectivity of $(A, \alpha)$ (in the above sense) is equivalent with the usual notion of $\G$-injectivity of the $\G$-$W^*$-algebra $(A, \alpha)$, defined as an injective object in an appropriate category of equivariant spaces \cites{Cr17, DH24, DR26}.

The following result generalizes the fact that $\G$ is strongly amenable if and only if the regular $\G$-representation weakly contains every $\G$-representation. In the proof, we will write $\Rep(M)$ for the $W^*$-category of unital normal $*$-representations of the von Neumann algebra $M$ on Hilbert spaces. 

\begin{Prop}\label{charstronginj}
    Let $\G$ be a locally compact quantum group and let $A$ be a $\G$-$W^*$-algebra. Then $A$ is strongly $\G$-injective if and only if $\mcH \preccurlyeq C_A^\G$ for every $\mcH \in \Corr^\G(A,A)$.
\end{Prop}
\begin{proof} 

From \cite{DR25b}*{Lemma 3.8}, we may view $\Rep(A\rtimes \G) \subseteq \Corr^\G(A,\C)$. From \cite{DCDR24}*{Remark 2.8}, we have that $C_A^\G \cong L^2(A)\boxtimes_\C L^2(\G)\boxtimes_\C L^2(A)$ as $\G$-$A$-$A$-correspondences. 

Assume that $L^2(A)\preccurlyeq C_{A}^\G$ as $\G$-$A$-$A$-correspondences. Given $\mcH \in \Corr^\G(A,A)$, we have $$\mcH \cong \mcH \boxtimes_A L^2(A) \preccurlyeq \mcH\boxtimes_A L^2(A) \boxtimes_\C L^2(\G)\boxtimes_\C L^2(A)\cong \mcH\boxtimes_\C L^2(\G) \boxtimes_\C L^2(A)$$ as $\G$-$A$-$A$-correspondences, where we used that the Connes fusion tensor product preserves equivariant weak containment \cite{DCDR24}*{Proposition 6.3}. It follows from \cite{DR25c}*{Lemma 3.1} that $\mcH \boxtimes_\C L^2(\G)\in \Rep(A\rtimes \G)$. The $W^*$-category $\Rep(A\rtimes \G)$ has the generator $L^2(A)\boxtimes_\C L^2(\G)\in \Rep(A\rtimes \G)\subseteq \Corr^\G(A,\C)$. Consequently, there exists an index set $I$ such that 
    $\mcH\boxtimes_\C L^2(\G) \subseteq \bigoplus_{i\in I} L^2(A)\boxtimes_\C L^2(\G)$
    as $\G$-$A$-$\C$-correspondences. In particular, it follows from \cite{DCDR24}*{Remark 3.2} that $\mcH \boxtimes_\C L^2(\G)\preccurlyeq L^2(A)\boxtimes_\C L^2(\G)$ as $\G$-$A$-$\C$-correspondences. Consequently,
    $$\mcH \preccurlyeq \mcH \boxtimes_\C L^2(\G)\boxtimes_\C L^2(A)\preccurlyeq L^2(A)\boxtimes_\C L^2(\G)\boxtimes_\C L^2(A)\cong C_A^\G$$
    as $\G$-$A$-$A$-correspondences, finishing the proof.
\end{proof}
\subsection{Compact quantum groups} A locally compact quantum group $\G$ is called \emph{compact} if the $C^*$-algebra $C_0(\G)$ is unital. In that case, we denote it simply by $C(\G)$. For the basic theory of compact quantum groups, we refer the reader to \cite{NT14}*{Chapter 1}. We collect here some notations, conventions and useful facts. Let $\G$ be a compact quantum group with (normal, faithful) Haar state $\varphi_\G: L^\infty(\G)\to \C$ so that
$$(\varphi_\G\otimes \id)\Delta_\G(x)= (\id \otimes \varphi_\G)\Delta_\G(x)= \varphi_\G(x)1, \quad x\in L^\infty(\G).$$
We call $\G$ of \emph{Kac type} if $\varphi_\G$ is tracial. We write $\Rep(\G)$ for the $W^*$-category of finite-dimensional unitary $\G$-representations. Its objects are pairs $\pi= (\mcH_\pi, U_\pi)$ where $\mcH_\pi$ is a finite-dimensional Hilbert space and $U_\pi\in B(\mcH_\pi)\odot C(\G)$ is a unitary satisfying $(\id \otimes \Delta_\G)(U_\pi)= U_{\pi,12}U_{\pi,13}$. We write $d_\pi:= \dim(\mcH_\pi)$.
If $\pi_1, \pi_2 \in \Rep(\G)$, we write $\pi_1 \oplus \pi_2\in \Rep(\G)$ for their direct sum (on the Hilbert space $\mcH_{\pi_1}\oplus \mcH_{\pi_2}$) and $\pi_1 \otimes \pi_2= (\mcH_{\pi_1}\otimes \mcH_{\pi_2}, U_{\pi_1,13} U_{\pi_2,23})\in \Rep(\G)$ for their tensor product.
Given $\pi \in \Rep(\G)$ and $\xi, \eta \in \mcH_\pi$, write
$U_\pi(\xi, \eta):= (\omega_{\xi, \eta}\otimes \id)(U_\pi)\in C(\G)$. The linear subspace of $C(\G)$ generated by the matrix coefficients $\{U_\pi(\xi,\eta): \pi \in \Rep(\G), \xi, \eta \in \mcH_\pi\}$, will be denoted by $\mathcal{O}(\G)$.  We have an associated Hopf $^*$-algebraic coaction
$$\delta_\pi: \mcH_\pi\to \mcH_\pi \odot \mcO(\G): \xi \mapsto U_\pi(\xi \otimes 
1).$$ The space  $\mathcal{O}(\G)$ is a norm-dense unital $*$-subalgebra of $C(\G)$. If $\pi\in \Rep(\G)$ and if $\{e_j^\pi\}_{j=1}^{d_\pi}$ is an orthonormal basis for $\mH_\pi$, then
\begin{equation}\label{comultiplication action}
    \Delta_\G(U_\pi(\xi, \eta)) = \sum_{j=1}^{d_\pi} U_\pi(\xi, e_j^\pi)\otimes U_\pi(e_j^\pi, \eta), \quad \xi, \eta \in \mH_\pi.
\end{equation}
In particular, $\Delta_\G(\mathcal{O}(\G))\subseteq \mathcal{O}(\G)\odot \mathcal{O}(\G)$. The pair $(\mathcal{O}(\G), \Delta_\G)$ has the structure of a Hopf $*$-algebra. Concretely, the counit $\epsilon_\G: \mathcal{O}(\G)\to \mathbb{C}$ and the antipode $S_\G: \mathcal{O}(\G)\to \mathcal{O}(\G)$ satisfy
\begin{equation}\label{counit and antipode}\epsilon_\G(U_\pi(\xi, \eta))= \langle \xi, \eta\rangle, \quad S_\G(U_\pi(\xi, \eta)) = U_\pi(\eta, \xi)^*, \quad  \pi \in \Rep(\G), \quad \xi, \eta \in \mH_\pi.\end{equation}
We fix a maximal collection $\Irr(\G)$ of irreducible $\G$-representations $\{U_\pi\in B(\mathcal{H}_\pi)\odot \mathcal{O}(\G)\}_{\pi\in \Irr(\G)}$ that are pairwise non-isomorphic. Then 
$\{U_\pi(e_i^\pi, e_j^\pi): \pi\in \Irr(\G), 1 \le i,j \le d_\pi\}$
forms a Hamel basis for $\mcO(\G)$. 

Let us also write $\mathscr{U}_\G$ for the algebraic linear dual of $\mcO(\G)$. We then have the isomorphism
\begin{equation}\label{identification}
    \mathscr{U}_\G\cong \prod_{\pi \in \Irr(\G)} B(\mcH_\pi): \omega \mapsto ((\id \odot \omega)(U_\pi))_{\pi \in \Irr(\G)}
\end{equation}
of $*$-algebras (here $\prod$ denotes the algebraic direct product), where $\mathscr{U}_\G$ becomes a $*$-algebra for the product and involution given by 
$$(\omega \star \omega')(z)= (\omega \odot \omega')\Delta_{\G}(z), \quad \omega^*(z)= \overline{\omega(S_{\G}(z)^*)}, \quad \omega, \omega'\in \mathscr{U}_{\G}, \quad z \in \mcO(\G).$$
The $*$-subalgebra of $\mathscr{U}_\G$ corresponding to the algebraic direct sum $\bigoplus_{\pi\in \Irr(\G)}^{\operatorname{alg}} B(\mcH_\pi)$ will be denoted by $\mathcal{U}_\G$. We then have $\mathcal{U}_\G=\{\varphi_\G(a-): a \in \mcO(\G)\}$, and we can naturally view $\mathcal{U}_\G\subseteq L^1(\G)$ as a norm-dense subspace.

With respect to the Haar state $\varphi_\G$ on $L^\infty(\G)$, we can consider the modular one-parameter group $$\{\sigma_t^\G:= \sigma_t^{\varphi_\G}: L^\infty(\G)\to L^\infty(\G)\}_{t\in \mathbb{R}}.$$ The elements of $\mathcal{O}(\G)$ are analytic for $\{\sigma_t^\G\}_{t\in \mathbb{R}}$ and $\sigma_z^\G(\mathcal{O}(\G)) \subseteq \mathcal{O}(\G)$ for all $z\in \C$. This leads to the \emph{Woronowicz characters} 
\begin{equation}\label{DefWorChar}
\delta^z_\G := \varepsilon_\G \circ \sigma_{iz}^\G \in \msU_\G.
\end{equation}
We have the following formulas for the modular automorphism group and \emph{scaling group} of $\mathcal{O}(\G)$:
\begin{equation}\label{EqFormSigmaScaling}
\sigma_{z}^\G(x) =  \delta^{-iz/2}_\G(x_{(1)})x_{(2)}\delta^{-iz/2}_\G(x_{(3)}),\qquad \tau_{z}^\G(x) =  \delta^{-iz/2}_\G(x_{(1)})x_{(2)}\delta^{iz/2}_\G(x_{(3)})\qquad x\in \mathcal{O}(\G),
\end{equation}
where the (sumless)  Sweedler notation was used. Also, with $R_\G: L^\infty(\G)\to L^\infty(\G)$ the unitary antipode, we have $R_\G(\mathcal{O}(\G))= \mathcal{O}(\G)$ and  
\begin{equation}
S_\G = \tau_{-i/2}^\G\circ R_\G = R_\G\circ \tau_{-i/2}^\G,\qquad S_\G^2 = \tau_{-i}^\G. 
\end{equation}

We will often view 
$\delta_\G\in  B(\mH_\pi)$ for $\pi\in \Irr(\G)$ using the identification \eqref{identification}.  If we wish to emphasize the dependency on $\pi$, we will sometimes also write $\delta_\G^\pi$. We have $\operatorname{Tr}((\delta_\G^{\pi})^{1/2})= \operatorname{Tr}((\delta_\G^{\pi})^{-1/2})$ and we call this common value the \emph{quantum dimension of $\pi$} and we denote it by $\dim_q(\pi)$. 

Given $\pi\in \Rep(\G)$, define $\bar{\pi}= (\overline{\mcH_\pi}, \overline{U_\pi})\in \Rep(\G)$ by
$\overline{U_\pi}:= (j_{\mcH_\pi}\odot R_\G)(U_\pi) \in B(\overline{\mcH_\pi})\odot \mcO(\G),$ where $j_{\mcH_\pi}: B(\mcH_\pi)\to B(\overline{\mcH_\pi})$ is given by $j_{\mcH_\pi}(x)\bar{\xi}= \overline{x^*\xi}$ for $x\in B(\mcH_\pi)$ and $\xi \in \mcH_\pi$. We have the following useful formulas, for every $\pi \in \Rep(\G)$ and $\xi, \eta \in \mcH_\pi$:
\begin{align}
    U_\pi(\xi, \eta)^* &= U_{\bar{\pi}}(\overline{\delta^{1/4}_\G\xi}, \overline{\delta^{-1/4}_\G\eta})= U_{\bar{\pi}}(\delta^{-1/4}_\G \overline{\xi}, \delta^{1/4}_\G \overline{\eta}),\\
    \label{aa}R_\G(U_\pi(\xi, \eta))&= U_{\bar{\pi}}(\overline{\eta}, \overline{\xi})= U_\pi(\delta_\G^{-1/4}\eta, \delta_\G^{1/4}\xi)^*,\\
\label{bb}\sigma_z^\G(U_{\pi}(\xi,\eta)) &= U_{\pi}(\delta_\G^{i\overline{z}/2}\xi,\delta^{-iz/2}_\G\eta),\\
\label{cc}\tau_z^\G(U_{\pi}(\xi,\eta)) &= U_{\pi}(\delta_\G^{i\overline{z}/2}\xi,\delta_\G^{iz/2}\eta).
\end{align}

The \emph{orthogonality relations} for irreducible $\G$-representations $\pi,\pi'\in \Irr(\G)$ are given for $\xi, \eta \in \mcH_\pi$ and $\xi', \eta'\in \mcH_{\pi'}$ by:
\begin{align}\label{EqPW1}
\varphi_\G(U_{\pi}(\xi,\eta)U_{\pi'}(\xi',\eta')^*) &= \delta_{\pi,\pi'} \dim_q(\pi)^{-1}\langle\xi, \xi'\rangle \langle \eta',\delta_\G^{-1/2}\eta\rangle\\
\label{EqPW2}
\varphi_\G(U_{\pi}(\xi,\eta)^*U_{\pi'}(\xi',\eta')) &=\delta_{\pi,\pi'} \dim_q(\pi)^{-1} \langle \xi',\delta_\G^{1/2}\xi\rangle \langle \eta, \eta'\rangle
\end{align}

\section{Construction of an algebraic compact quantum hypergroup}\label{matrixcoefficients}

\subsection{Ergodic compact quantum group actions} In this subsection, we recall some theory of ergodic compact quantum group actions, following \cites{Boc95, BDRV06, DC17}.

Fix a compact quantum group $\G$ and suppose that an ergodic action $L^\infty(\X)\stackrel{\alpha}\curvearrowleft \G$ is given. 

If $\pi\in \Rep(\G)$, we define 
$\Mor(\pi, \X)$ to be the space of linear maps $T: \mcH_\pi\to L^\infty(\X)$ satisfying $(T \odot \id)\delta_\pi(\xi) = \alpha (T(\xi))$ for all $\xi \in \mcH_\pi$, and we define the $\pi$-spectral subspace
$$\mcO(\X)_\pi:= \operatorname{span}\{T\xi: T \in \Mor(\pi, \X), \xi \in \mcH_\pi\}.$$
We then put $\mcO(\X):= \bigoplus_{\pi\in \Irr(\G)} \mcO(\X)_\pi$ and we call $\mcO(\X)$ the \emph{algebraic core} of the action. It is a $\sigma$-weakly dense unital $*$-subalgebra of $L^\infty(\X)$. Its norm-closure inside $L^\infty(\X)$ will be denoted by $C(\X)$. We then have $\alpha(\mcO(\X))\subseteq \mcO(\X) \odot \mcO(\G)$ and $\alpha(C(\X))\subseteq C(\X) \otimes C(\G)$, so that $\alpha$ defines a coaction of the Hopf $*$-algebra $\mcO(\G)$ on $\mcO(\X)$ and so that $C(\X)$ becomes a $\G$-$C^*$-algebra. If $x\in \mcO(\X)$, we will sometimes use the sumless Sweedler notation $\alpha(x)= x_{(0)}\otimes x_{(1)}$.

Given $\pi\in \Rep(\G)$, we define the space
$$\mcG_\pi:= \{\mu \in \mcH_\pi \odot \mcO(\X): (\id \otimes \alpha)(\mu)= U_{\pi,13}^*\mu_{12}\}.$$
In other words, if $\{e_j^\pi\}_{j=1}^{d_\pi}$ is an orthonormal basis for $\mcH_\pi$ and $x_1, \dots, x_{d_\pi}\in \mcO(\X)$, then
\begin{equation}\label{coordinates}\sum_{j=1}^{d_\pi} e_j^\pi\otimes x_j \in \mcG_\pi \iff \forall j \in \{1, \dots, d_\pi\}: \alpha(x_j) = \sum_{k=1}^{d_\pi} x_k\otimes U_\pi(e_k^\pi, e_j^\pi)^*.\end{equation}
It is clear that $\mcG_\pi\subseteq \mcH_\pi \odot \mcO(\X)_\pi^*= \mcH_\pi \odot \mcO(\X)_{\bar{\pi}}$.
We emphasize that this space may very well be zero for certain choices of (irreducible) representations (see Section \ref{fusion algebra} for a concrete example).

The vector space $\mcH_\pi \odot \mcO(\X)$ carries the $\mcO(\X)$-valued inner product given on elementary tensors by
$$\langle \xi \otimes x, \eta \otimes y\rangle := \langle \xi, \eta\rangle x^* y, \quad \xi, \eta \in \mcH_\pi, \quad x,y \in \mcO(\X).$$
If $\mu, \nu \in \mcG_\pi$, then $\langle \mu, \nu\rangle \in L^\infty(\X)^\alpha = \C1$, so that $\mcG_\pi$ carries a canonical complex-valued inner product. If $\mu \in \mcG_\pi$ and $\xi \in \mcH_\pi$, we define
$X_\pi(\mu, \xi):= \langle \mu, \xi \otimes 1\rangle \in \mcO(\X).$ Then
\begin{align}\label{actionmatrixcoeff}
    \alpha(X_\pi(\mu, \xi)) = \sum_{j=1}^{d_\pi} X_\pi(\mu, e_j^\pi)\otimes U_\pi(e_j^\pi, \xi).
\end{align}
From this, it follows that
$X_\pi(\mu, -)\in \Mor(\pi, \X)$.
It is then easy to verify that the assignment
$$\mcG_\pi \ni \mu\mapsto X_\pi(\mu,-)\in \Mor(\pi, \X)$$
defines an anti-linear bijection, with inverse given by $T \mapsto \sum_{j=1}^{d_\pi} e_j^\pi \otimes T(e_j^\pi)^*$. By \cite{Boc95}*{Theorem 17}, $\dim(\mcO(\X)_\pi)< \infty$ for all $\pi \in \Rep(\G)$. Thus, $\mcG_\pi$ is a finite-dimensional Hilbert space for every $\pi\in \Rep(\G)$. In the sequel, we will write $m_\pi:= \dim(\mcG_\pi)$ for $\pi\in \Rep(\G)$ and we will fix an orthonormal basis $\{f_j^\pi\}_{j=1}^{m_\pi}$ for $\mcG_\pi$. Then $\{X_\pi(f_j^\pi, e_k^\pi): \pi\in \Irr(\G), 1\le j \le m_\pi, 1 \le k\le  d_\pi\}$ forms a Hamel basis for $\mcO(\X)$.

The following example shows that we have introduced a bona fide generalization for the matrix coefficients of $\G$:

\begin{Exa}\label{exmatrixcoeff}
    Consider the action $ L^\infty(\G)\stackrel{\Delta_\G}\curvearrowleft \G$ and fix $\pi \in \Rep(\G)$. It is easy to verify that the maps
    \begin{equation}\label{unitary}
        \mcH_\pi\to \mcG_\pi: \xi \mapsto \mu_\xi:= \sum_{j=1}^{d_\pi} e_j^\pi\otimes U_\pi(\xi, e_j^\pi)^*, \quad \mcG_\pi\to \mcH_\pi: \mu \mapsto \xi_\mu:= (\id \odot \epsilon_\G)(\mu)
    \end{equation}
    are unitaries that are inverse to each other. Given $\xi, \eta \in \mcH_\pi$, we have
    $X_\pi(\mu_\xi, \eta)=U_\pi(\xi, \eta).$
\end{Exa}
Returning back to the general theory, there is a unique (normal, faithful) state $\varphi_\X: L^\infty(\X)\to \C$ such that $(\id \otimes \varphi_\G)\alpha(x)= \varphi_\X(x)1$ for all $x\in L^\infty(\X)$. Then $\varphi_\X$ is KMS, in the sense that there exists a  unique algebra automorphism $\sigma_\X: \mcO(\X)\to \mcO(\X)$ satisfying
\begin{equation}\label{KMS}\varphi_\X(ab)= \varphi_\X(b\sigma_\X(a)), \quad a,b\in \mcO(\X).\end{equation}
Given $\pi \in \Irr(\G)$, there exists a unique invertible, positive operator $\delta_\X\in B(\mcG_\pi)$ such that 
\begin{equation}\label{determines uniquely}
    \sigma_\X(X_\pi(\mu, \xi))= X_\pi(\delta_\X^{-1/2}\mu, \delta_\G^{-1/2}\xi), \quad \mu \in \mcG_\pi, \quad \xi \in \mcH_\pi.
\end{equation}
If we want to emphasize the representation $\pi$, we will sometimes also write $\delta_\X^\pi\in B(\mcG_\pi)$. The existence of these operators allows us to define for every $z\in \C$, the map
\begin{equation}\label{modularanalytic}
    \sigma_z^\X: \mcO(\X)\to \mcO(\X), \quad  \sigma_z^\X(X_\pi(\mu, \xi))= X_\pi(\delta_\X^{i\bar{z}/2}\mu, \delta_\G^{-iz/2}\xi),\quad \pi \in \Irr(\G), \quad \mu \in \mcG_\pi, \quad \xi \in \mcH_\pi.
\end{equation}
These maps satisfy
$$\sigma_{\X}= \sigma_{-i}^\X, \quad \sigma_z^\X(ab)= \sigma_z^\X(a)\sigma_z^\X(b), \quad \sigma_z^\X(a^*)= \sigma_{\bar{z}}^\X(a)^*, \quad z\in \C, \quad a,b\in \mcO(\X).$$
Moreover, $\varphi_\X(\sigma_z^\X(a))=\varphi_\X(a)$ for all $a\in \mcO(\X)$. Therefore, the $*$-automorphisms
$\{\sigma_t^\X: \mcO(\X)\to \mcO(\X)\}_{t\in \R}$ extend to (normal) $*$-automorphisms
$\{\sigma_t^\X: L^\infty(\X)\to L^\infty(\X)\}_{t\in \R}.$
It can then be shown that these $*$-automorphisms coincide with the modular group on $L^\infty(\X)$ associated to the faithful normal state $\varphi_\X$.

The following orthogonality relations hold for $\pi, \pi' \in \Irr(\G)$,  $\mu\in \mcG_\pi, \mu'\in \mcG_{\pi'}, \xi \in \mcH_\pi$ and $\xi' \in \mcH_{\pi'}$:
    \begin{align}\label{orthogonality1}\varphi_{\X}(X_\pi(\mu, \xi)X_{\pi'}(\mu', \xi')^*)&= \delta_{\pi, \pi'}\dim_q(\pi)^{-1}\langle \xi', \delta_\G^{-1/2} \xi\rangle \langle \mu, \mu'\rangle,\\\label{orthogonality2}\varphi_\mathcal{\X}(X_\pi(\mu, \xi)^*X_{\pi'}(\mu', \xi')) &= \delta_{\pi, \pi'} \dim_q(\pi)^{-1}\langle\mu', \delta_\X^{1/2}\mu\rangle\langle\xi, \xi'\rangle. \end{align}
    To prove \eqref{orthogonality1}, we simply use that $\varphi_\X= (\varphi_\G\otimes \id)\circ \alpha$, the formula \eqref{actionmatrixcoeff} and the orthogonality relation \eqref{EqPW1}. The equation \eqref{orthogonality2} then follows from \eqref{orthogonality1} and the KMS-condition \eqref{KMS}.

It is also possible to define the operators $\delta_\X^\pi\in B(\mcH_\pi)$ for non-irreducible $\pi\in \Rep(\G)$. Indeed, choose $\pi_1, \dots, \pi_n \in \Irr(\G)$ (repetition allowed) and intertwining isometries $w_j: \mcH_{\pi_j}\to \mcH_\pi$ satisfying $\sum_{j=1}^n w_j w_j^*= 1$. The vector space isomorphism
\begin{align*}
    \mcH_\pi\odot \mcO(\X)\to \bigoplus_{j=1}^n (\mcH_{\pi_j}\odot \mcO(\X)): \mu \mapsto ((w_j^*\otimes 1)\mu)_{j=1}^n
\end{align*}
restricts to a unitary $\Phi: \mcG_\pi \to \bigoplus_{j=1}^n \mcG_{\pi_j}$. We then see that
$X_\pi(\mu, \xi) = \sum_{j=1}^n X_{\pi_j}((w_j^*\otimes 1)\mu, w_j^*\xi)$ for all $\mu \in \mcG_\pi$ and all $\xi \in \mcH_\pi$. Consequently, the positive, invertible operator
$\delta_\X^\pi:= \Phi^{-1}\circ (\bigoplus_{j=1}^n \delta_\X^{\pi_j})\circ \Phi\in B(\mcG_\pi)$
satisfies \eqref{determines uniquely}, and this condition also determines it uniquely. Then also formulas such as \eqref{modularanalytic} hold for non-irreducible $\pi\in \Rep(\G)$.

\begin{Prop}\label{1}
    Let $\pi, \pi_1, \pi_2 \in \Rep(\G)$.
    \begin{enumerate}
        \item The canonical vector space isomorphism
    $$(\mathcal{H}_{\pi_1}\oplus\mathcal{H}_{\pi_2})\odot \mcO(\X)\to (\mathcal{H}_{\pi_1}\odot \mcO(\X))\oplus (\mathcal{H}_{\pi_2}\odot \mcO(\X))$$ restricts to a unitary
    $\mathcal{G}_{\pi_1\oplus \pi_2}\cong \mathcal{G}_{\pi_1}\oplus \mathcal{G}_{\pi_2}.$ Under this identification, we have $\delta_\X^{\pi_1\oplus \pi_2}= \delta_\X^{\pi_1}\oplus \delta_{\X}^{\pi_2}$, and
    $$X_{\pi_1\oplus \pi_2}(\mu_1\oplus \mu_2, \xi_1 \oplus \xi_2) = X_{\pi_1}(\mu_1, \xi_1)+ X_{\pi_2}(\mu_2, \xi_2), \quad \mu_i \in \mcG_{\pi_i}, \quad \xi_i \in \mcH_{\pi_i}.$$
    \item The map
    $\mathcal{G}_{\pi_1}\otimes \mathcal{G}_{\pi_2}\hookrightarrow \mathcal{G}_{\pi_1\otimes \pi_2}: \mu_1 \otimes \mu_2 \mapsto (\mu_{2})_{23}(\mu_1)_{13}$ is isometric. Under this embedding, we have that $\delta_{\X}^{\pi_1\otimes \pi_2}\vert_{\mcG_{\pi_1}\otimes \mcG_{\pi_2}}= \delta_\X^{\pi_1}\otimes \delta_\X^{\pi_2}$, and  
    $$X_{\pi_1\otimes \pi_2}(\mu_1\otimes \mu_2, \xi_1 \otimes \xi_2) = X_{\pi_1}(\mu_1, \xi_1)X_{\pi_2}(\mu_2, \xi_2), \quad \quad \mu_i \in \mcG_{\pi_i}, \quad \xi_i \in \mcH_{\pi_i}.$$
    \item There is a canonical unitary
    $ \overline{\mathcal{G}_\pi}\to \mathcal{G}_{\bar{\pi}}: \overline{\mu} \mapsto \sum_{j=1}^{d_\pi}\overline{e_j^\pi}\otimes X_\pi(\delta_\X^{-1/4}\mu, \delta_\G^{1/4}e_j^\pi).$
    Under this identification, we have that $\delta^{\bar{\pi}}_\X \bar{\mu}= \overline{(\delta_\X^{\pi})^{-1} \mu}$ for $\mu \in \mcG_\pi$, and 
    $$X_\pi(\mu, \xi)^* = X_{\bar{\pi}}(\overline{\delta^{1/4}_\X\mu}, \overline{\delta_\G^{-1/4}\xi}), \quad \mu \in \mcG_\pi, \quad \xi \in \mcH_\pi.$$
    \end{enumerate}
\end{Prop}
\begin{proof} (1) This follows similarly as in  the discussion preceding the Proposition.

(2) It is straightforward to verify that the map $\mathcal{G}_{\pi_1}\otimes \mathcal{G}_{\pi_2}\to \mathcal{G}_{\pi_1\otimes \pi_2}: \mu_1 \otimes \mu_2 \mapsto (\mu_{2})_{23}(\mu_1)_{13}$ is well-defined and isometric. Using this embedding, we then calculate
\begin{align*}
    X_{\pi_1\otimes \pi_2}(\mu_1\otimes \mu_2, \xi_1\otimes \xi_2) &= \langle (\mu_{2})_{23}(\mu_{1})_{13}, \xi_1\otimes \xi_2 \otimes 1\rangle = \langle \mu_1, \xi_1\otimes 1\rangle \langle \mu_2, \xi_2\otimes 1\rangle  = X_{\pi_1}(\mu_1, \xi_2)X_{\pi_2}(\mu_2, \xi_2).
\end{align*}
Using this, the fact that $\sigma_{2i}^\X$ is multiplicative and the formula \eqref{modularanalytic} (which holds for arbitrary $\pi\in \Rep(\G)$), a simple calculation shows that
\begin{align*}
    X_{\pi_1\otimes \pi_2}(\delta_\X^{\pi_1\otimes \pi_2}(\mu_1\otimes \mu_2), \delta_\G^{\pi_1}\xi_1 \otimes \delta_\G^{\pi_2}\xi_2) = X_{\pi_1\otimes \pi_2}(\delta_\X^{\pi_1}\mu_1\otimes \delta_\X^{\pi_2} \mu_2, \delta_\G^{\pi_1}\xi_1\otimes \delta_\G^{\pi_2}\xi_2)
\end{align*}
for all $\mu_i\in \mcG_{\pi_i}$ and $\xi_i \in \mcH_{\pi_i}$, from which we conclude that $\delta_{\X}^{\pi_1\otimes \pi_2}\vert_{\mcG_{\pi_1}\otimes \mcG_{\pi_2}}= \delta_\X^{\pi_1}\otimes \delta_\X^{\pi_2}$.

(3) Given $\mu \in \mcG_\pi$, we calculate for $j\in \{1, \dots, d_\pi\}$ that
\begin{align*}
    \alpha(X_\pi(\delta_\X^{-1/4}\mu, \delta_\G^{1/4}e_j^\pi)) &= \sum_{k=1}^{d_\pi} X_\pi(\delta_\X^{-1/4} \mu, \delta_\G^{1/4}e_k^\pi) \otimes U_\pi(\delta_\G^{-1/4}e_k^\pi, \delta_\G^{1/4}e_j^\pi)\\
    &= \sum_{k=1}^{d_\pi} X_\pi(\delta_\X^{-1/4}\mu, \delta_\G^{1/4}e_k^\pi)\otimes U_{\bar{\pi}}(\overline{e_k^\pi}, \overline{e_j^\pi})^*,
\end{align*}
so that \eqref{coordinates} shows that $\sum_{j=1}^{d_\pi} \overline{e_j^\pi}\otimes X_\pi(\delta_\X^{-1/4}\mu, \delta_\G^{1/4}e_j^\pi)\in \mcG_{\bar{\pi}}$. A calculation using the orthogonality relation \eqref{orthogonality2} shows that the map in $(3)$ is isometric, so that $\dim(\mcG_\pi)\le \dim(\mcG_{\bar{\pi}})$. By symmetry, it follows that $\dim(\mcG_\pi)= \dim(\mcG_{\bar{\pi}})$, and the map in $(3)$ is a unitary isomorphism of Hilbert spaces. Under this identification, we then have
\begin{align*}
    X_{\bar{\pi}}(\overline{\delta_\X^{1/4}\mu}, \overline{\delta_\G^{-1/4}\xi}) &= \left\langle \sum_{j=1}^{d_\pi} \overline{e_j^\pi}\otimes X_\pi(\mu, \delta_\G^{1/4}e_j^\pi), \overline{\delta_\G^{-1/4}\xi}\otimes 1\right\rangle= \sum_{j=1}^{d_\pi} \langle \delta_\G^{-1/4}\xi, e_j^\pi\rangle X_\pi(\mu, \delta_\G^{1/4}e_j^\pi)^*= X_\pi(\mu, \xi)^*.
\end{align*}
The identity $\delta^{\bar{\pi}}_\X \bar{\mu}= \overline{(\delta_\X^{\pi})^{-1} \mu}$ for $\mu \in \mcG_\pi$ then follows similarly as in $(2)$, making use of the fact that $\sigma_{2i}^\X(a^*) = \sigma_{-2i}^\X(a)^*$ for $a\in \mcO(\X)$. 
\end{proof}

We write $(L^2(\X), \pi_\X, \xi_\X)$ for the GNS-triplet associated to the normal faithful state $\varphi_\X$ and $\Lambda_\X: L^\infty(\X)\to L^2(\X)$ for the corresponding GNS-map. It follows from \eqref{orthogonality2} that we have the unitary isomorphism
\begin{equation}\label{unitaryiso}
    \bigoplus_{\pi \in \Irr(\G)} (\overline{\mcG_\pi}\otimes \mcH_\pi) \to L^2(\X):\quad \overline{\mcG_\pi}\otimes \mcH_\pi\ni \overline{\mu} \otimes \xi \mapsto \dim_q(\pi)^{1/2}\Lambda_\X(X_\pi(\delta_\X^{-1/4}\mu, \xi)).
\end{equation} 

The canonical unitary implementation $U_\X\in B(L^2(\X))\ovot L^\infty(\G)$ associated to $L^\infty(\X)\stackrel{\alpha}\curvearrowleft \G$ is given by
\begin{equation}\label{formulaimplementation}
    U_\X(\Lambda_\X(x)\otimes \Lambda_\G(g)) = (\Lambda_\X\otimes \Lambda_\G)(\alpha(x)(1\otimes g)), \quad x\in L^\infty(\X), \quad g\in L^\infty(\G).
\end{equation}
Given $\pi \in \Rep(\G)$, $\mu\in \mcG_\pi$ and $\xi \in \mcH_\pi$, put
$$Y_\pi(\xi, \mu):= \overline{X_\pi(\delta_\X^{-1/4}\mu, \delta_\G^{1/4}\xi)}\in \mcO(\bar{\X}).$$
It is then easily verified that
\begin{equation}\label{baractionmatrixcoeff}\bar{\alpha}(Y_\pi(\xi, \mu))= \sum_{j=1}^{d_\pi} U_\pi(\xi, e_j^\pi)\otimes Y_\pi(e_j^\pi, \mu).\end{equation} The invariant functional for $\G\stackrel{\bar{\alpha}}\curvearrowright L^\infty(\bar{\X})$ is given by $\varphi_{\bar{\X}}(\bar{x})= \varphi_\X(x^*)= \overline{\varphi_\X(x)}$ for $x\in L^\infty(\X)$. We then have 
$\sigma_{\bar{\X}}(\bar{x})= \overline{\sigma_\X(x)}$, or equivalently 
$$\sigma_{\bar{\X}}(Y_\pi(\xi, \mu))  = Y_\pi(\delta_\G^{-1/2}\xi, \delta_\X^{-1/2}\mu), \quad \pi \in \Rep(\G), \quad \xi\in \mcH_\pi, \quad \mu \in \mcG_\pi.$$
We also have the following orthogonality relations, where $\pi, \pi'\in \Irr(\G)$, $\xi \in \mcH_\pi, \xi' \in \mcH_{\pi'}$, $\mu\in \mcG_\pi$ and $\mu'\in \mcG_{\pi'}$:
\begin{align}
    \label{orthogonality3}\varphi_{\bar{\X}}(Y_\pi(\xi, \mu) Y_{\pi'}(\xi',\mu')^*) &= \delta_{\pi, \pi'} \dim_q(\pi)^{-1}\langle \mu', \delta_\X^{-1/2}\mu\rangle \langle \xi, \xi'\rangle,\\
    \label{orthogonality4}\varphi_{\bar{\X}}(Y_\pi(\xi, \mu)^*Y_{\pi'}(\xi', \mu'))&= \delta_{\pi, \pi'}\dim_q(\pi)^{-1}\langle \xi', \delta_\G^{1/2}\xi\rangle\langle\mu, \mu'\rangle.
\end{align}
These relations follow from \eqref{orthogonality1} and \eqref{orthogonality2} by making some straightforward calculations.

We write $(L^2(\bar{\X}), \pi_{\bar{\X}}, \xi_{\bar{\X}})$ for the GNS-triplet associated to the normal faithful state $\varphi_{\bar{\X}}$ and $\Lambda_{\bar{\X}}: L^\infty(\bar{\X})\to L^2(\bar{\X})$ for the corresponding GNS-map. It follows from \eqref{orthogonality4} that we have the unitary isomorphism
\begin{equation}\label{unitaryiso2}
    \bigoplus_{\pi \in \Irr(\G)} (\overline{\mcH_\pi}\otimes \mcG_\pi) \to L^2(\bar{\X}):\quad \overline{\mcH_\pi}\otimes \mcG_\pi\ni \overline{\xi} \otimes \mu \mapsto \dim_q(\pi)^{1/2}\Lambda_{\bar{\X}}(Y_\pi(\delta_\G^{-1/4}\xi, \mu)).
\end{equation}

The following result follows immediately from Proposition \ref{1}. Note that the identifications from Proposition \ref{1} are implicit.

\begin{Prop}\label{2}
    Let $\pi_1, \pi_2, \pi\in \Rep(\G)$. Then 
    \begin{enumerate}
        \item $Y_{\pi_1\oplus \pi_2}(\xi_1\oplus \xi_2, \mu_1\oplus \mu_2)= Y_{\pi_1}(\xi_1, \mu_1)+ Y_{\pi_2}(\xi_2, \mu_2)$ for $\xi_i\in \mcH_{\pi_i}$ and $\mu_i \in \mcG_{\pi_i}.$
        \item\vspace{-1.5mm} $Y_{\pi_1\otimes \pi_2}(\xi_1\otimes \xi_2, \mu_1\otimes \mu_2) = Y_{\pi_1}(\xi_1, \mu_1)Y_{\pi_2}(\xi_2, \mu_2)$  for $\xi_i \in \mcH_{\pi_i}$ and $\mu_i \in \mcG_{\pi_i}.$
        \item\vspace{-1.5mm} $Y_{\pi}(\xi, \mu)^* = Y_{\bar{\pi}}(\overline{\delta_\G^{1/4}\xi}, \overline{\delta_\X^{-1/4}\mu})$ for $\xi \in \mcH_\pi$ and $\mu \in \mcG_\pi$.
    \end{enumerate}
\end{Prop}

\subsection{Matrix coefficients for the cotensor product}

Define the unital $*$-algebra
$$\mcO(\XGX):= \{z\in \mcO(\X)\odot \mcO(\bar{\X}): (\alpha\odot \id)(z) = (\id \odot\bar{\alpha})(z)\}.$$

Given $\pi\in \Rep(\G)$ and $\mu, \nu \in \mcG_\pi$, it follows from \eqref{actionmatrixcoeff} and \eqref{baractionmatrixcoeff} that
$$Z_\pi(\mu, \nu):= \sum_{j=1}^{d_\pi} X_\pi(\mu, e_j^\pi)\otimes Y_\pi(e_j^\pi, \nu) \in \mcO(\XGX).$$
\begin{Prop}\label{Hamel}
   $\{Z_\pi(f_j^\pi, f_k^\pi): \pi\in \Irr(\G), 1\le j,k\le m_\pi\}$ is a (Hamel) basis for $\mcO(\XGX)$. 
\end{Prop}
\begin{proof}
    Linear independence is clear. Let $z\in \mcO(\XGX)\subseteq \mcO(\X)\odot \mcO(\bar{\X})$, so that we can write
    $$z = \sum_{\pi_1, \pi_2\in \Irr(\G)} \sum_{j=1}^{m_{\pi_1}} \sum_{k=1}^{d_{\pi_1}} \sum_{s=1}^{d_{\pi_2}} \sum_{t=1}^{m_{\pi_2}}  \lambda_{\pi_1, \pi_2, j,k,s,t}X_{\pi_1}(f_j^{\pi_1}, e_k^{\pi_1}) \otimes Y_{\pi_2}(e_s^{\pi_2}, f_t^{\pi_2})$$
    for certain scalars $\lambda_{\pi_1, \pi_2, j,k,s,t}\in \C$. Since $z\in \mcO(\XGX)$, the expressions
    \begin{align*}
        &\sum_{\pi_1, \pi_2,j,k,s,t,p} \lambda_{\pi_1, \pi_2,j,k,s,t} X_{\pi_1}(f_j^{\pi_1}, e_p^{\pi_1}) \otimes U_{\pi_1}(e_p^{\pi_1}, e_k^{\pi_1}) \otimes Y_{\pi_2}(e_s^{\pi_2}, f_t^{\pi_2}), \\
        &\sum_{\pi_1, \pi_2, j,k,s,t,q} \lambda_{\pi_1, \pi_2, j,k,s,t} X_{\pi_1} (f_j^{\pi_1}, e_{k}^{\pi_1})\otimes U_{\pi_2}(e_s^{\pi_2}, e_q^{\pi_2})\otimes Y_{\pi_2}(e_q^{\pi_2}, f_t^{\pi_2})
    \end{align*}
    are equal. Invoking linear independence of the matrix coefficients $U_\pi, X_\pi$ and $Y_\pi$, it follows that $\lambda_{\pi_1, \pi_2,j,k,s,t}= 0$ if $\pi_1\ne \pi_2$ or $k\ne s$, and that $\lambda_{\pi,j,k,k,t}$ does not depend on the choice of $k$. Therefore,
    $$z= \sum_{\pi, j,k,t}  \lambda_{\pi, \pi, j,k,k,t} X_{\pi}(f_j^\pi, e_k^\pi) \otimes Y_{\pi}(e_k^\pi, f_t^\pi)= \sum_{\pi,j,t}  \lambda_{\pi, \pi, j,1,1,t} Z_\pi(f_j^\pi, f_t^\pi),$$
    so $\{Z_\pi(f_j^\pi, f_k^\pi): \pi\in \Irr(\G), 1\le j,k\le m_\pi\}$ spans $\mcO(\XGX)$.
\end{proof}

The following is now an immediate consequence of Proposition \ref{1} and Proposition \ref{2}. Again, the identifications from Proposition \ref{1} are implicit. 
\begin{Prop}\label{3} If $\pi, \pi_1, \pi_2\in\Rep(\G)$, we have:
\begin{enumerate} 
    \item $Z_{\pi_1\oplus \pi_2}(\mu_1\oplus \mu_2, \nu_1\oplus \nu_2)= Z_{\pi_1}(\mu_1, \nu_1)+ Z_{\pi_2}(\mu_2, \nu_2)$ for $\mu_i, \nu_i\in \mcG_{\pi_i}$.
     \item\vspace{-1.5mm} $Z_{\pi_1\otimes \pi_2}(\mu_1\otimes \mu_2, \nu_1\otimes \nu_2) = Z_{\pi_1}(\mu_1, \nu_1)Z_{\pi_2}(\mu_2, \nu_2)$ for $\mu_i, \nu_i\in \mcG_{\pi_i}$.
        \item\vspace{-1.5mm} $Z_\pi(\mu, \nu)^* = Z_{\bar{\pi}}(\overline{\delta_\X^{1/4}\mu}, \overline{\delta_\X^{-1/4}\nu})$ for $\mu, \nu \in \mcG_\pi$.
\end{enumerate}
\end{Prop}

Next, consider the positive faithful state
$\varphi_{\XGX}: \mcO(\XGX)\to \C: z \mapsto (\varphi_\X \odot \varphi_{\bar{\X}})(z).$
We then have the following orthogonality relations for $\pi, \pi' \in \Irr(\G)$, $\mu,\nu \in \mcG_\pi$ and $\mu', \nu'\in \mcG_{\pi'}$: 
    \begin{align}
        \label{orthogonality5}\varphi_{\XGX}(Z_\pi(\mu, \nu)Z_{\pi'}(\mu', \nu')^*)&= \delta_{\pi, \pi'}\dim_q(\pi)^{-1}\langle \nu', \delta_\X^{-1/2}\nu\rangle\langle \mu, \mu'\rangle,\\
        \label{orthogonality6} \varphi_{\XGX}(Z_\pi(\mu, \nu)^*Z_{\pi'}(\mu', \nu')) &= \delta_{\pi, \pi'}\dim_q(\pi)^{-1}\langle \mu', \delta_\X^{1/2}\mu\rangle\langle \nu, \nu' \rangle.
    \end{align}
The equation \eqref{orthogonality5} is an immediate consequence of \eqref{orthogonality1} and \eqref{orthogonality3}, while equation \eqref{orthogonality6} follows immediately from \eqref{orthogonality2} and \eqref{orthogonality4}. 

\subsection{Algebraic compact quantum hypergroup}

The following definition is found in \cite{DVD11a}:

\begin{Def} \label{compacthyperdef}
    An \emph{algebraic compact quantum hypergroup} consists of the data $(A, \Delta, \epsilon, \varphi, S)$, where $A$ is a unital $*$-algebra, $\Delta: A \to A\odot A$ is a $*$-preserving linear map (not assumed to be multiplicative!) satisfying $(\Delta \odot \id)\Delta = (\id \odot \Delta)\Delta$, $\epsilon: A \to \C$ is an algebra morphism (called counit)  satisfying $(\epsilon\odot \id)\Delta= \id = (\id \odot \epsilon)\Delta$, $\varphi: A\to \C$ is a unital faithful\footnote{Recall that a functional $\varphi: A \to \C$ is called faithful if $f(xA)= 0$ implies $x= 0$ and $f(Ax)= 0$ implies $x= 0$.} positive functional (called integral) satisfying $(\id \odot \varphi)\Delta(a)= \varphi(a)1$ for all $a\in A$ and $S: A \to A$ is an antimultiplicative linear bijection (called antipode) satisfying the strong left invariance condition
    $$S((\id \odot \varphi)(\Delta(a)(1\otimes b)))= (\id \odot \varphi)((1\otimes a)\Delta(b)), \quad a,b\in A.$$
\end{Def}

\begin{Rem} The algebra morphism $\epsilon: A \to \C$ is automatically $*$-preserving \cite{DVD11a}*{Proposition 1.4}. Moreover, by \cite{DVD11a}*{Proposition 2.2}, $\psi:= \varphi\circ S$ is right invariant  and satisfies the strong right invariance condition
$$S((\psi \odot \id)((a\otimes 1)\Delta(b)))= (\psi \odot \id)(\Delta(a)(b\otimes 1)), \quad a,b\in A.$$ Consequently, if $a\in A$, we have $\psi(a) = \varphi(\psi(a)1) = \varphi((\psi \odot \id)\Delta(a)) = \psi((\id \odot \varphi)\Delta(a)) = \psi(\varphi(a)1) = \varphi(a)$, and thus $\varphi= \psi$. In particular, $\varphi$ is right invariant and also satisfies the strong right invariance condition.
This restores the `asymmetry' in Definition \ref{compacthyperdef}, where `left' is favoured over `right'.
\end{Rem}

By analogy with the formulas \eqref{comultiplication action} and \eqref{counit and antipode}, we now define the linear maps
\begin{align*}
&\Delta_{\XGX}: \mcO(\XGX)\to \mcO(\XGX)\odot \mcO(\XGX): Z_\pi(\mu, \nu)\mapsto \sum_{j=1}^{m_\pi} Z_\pi(\mu, f_j^\pi)\otimes Z_\pi(f_j^\pi, \nu),\\
    &\epsilon_{\XGX}: \mcO(\XGX)\to \C: Z_\pi(\mu, \nu)\mapsto \langle \mu, \nu\rangle, \quad S_{\XGX}: \mcO(\XGX)\to \mcO(\XGX): Z_\pi(\mu, \nu)\mapsto Z_\pi(\nu, \mu)^*.
\end{align*}
It follows from Proposition \ref{3} that $\Delta_{\XGX}$ is $*$-preserving, that $\epsilon_{\XGX}$ is multiplicative and that $S_{\XGX}$ is antimultiplicative.

\begin{Theorem}\label{hypergroup}
    $(\mcO(\XGX), \Delta_{\XGX}, \epsilon_{\XGX}, \varphi_{\XGX}, S_{\XGX})$ is an algebraic compact quantum hypergroup.
\end{Theorem}
\begin{proof} The only non-trivial thing left to verify is the strong left invariance
$$S_{\XGX}((\id \odot \varphi_{\XGX})(\Delta_{\XGX}(a)(1\otimes b))) = (\id \odot \varphi_{\XGX})((1\otimes a)\Delta_{\XGX}(b)), \quad a,b \in \mcO(\XGX).$$ It suffices to check this for $a= Z_\pi(\mu, \nu)$ and $b= Z_{\pi'}(\kappa, \lambda)^*$, where $\pi, \pi' \in \Irr(\G), \mu, \nu \in \mcG_\pi$ and $\kappa, \lambda \in \mcG_{\pi'}$. Using the orthogonality relation \eqref{orthogonality5}, we compute
\begin{align*}
    S_{\XGX}((\id \odot \varphi_{\XGX})(\Delta_{\XGX}(a)(1\otimes b))) &= \sum_{j=1}^{m_\pi} Z_\pi(f_j^\pi, \mu)^* \varphi_{\XGX} (Z_\pi(f_j^\pi, \nu)Z_{\pi'}(\kappa, \lambda)^*)\\
    &= \dim_q(\pi)^{-1}\sum_{j=1}^{m_\pi} Z_\pi(f_j^\pi, \mu)^* \delta_{\pi, \pi'} \langle f_j^\pi, \kappa\rangle \langle \lambda, \delta_\X^{-1/2}\nu\rangle\\
    &= \dim_q(\pi)^{-1} \delta_{\pi, \pi'} Z_\pi(\kappa, \mu)^* \langle \lambda, \delta_\X^{-1/2}\nu\rangle\\
    &=  \dim_q(\pi)^{-1} \delta_{\pi, \pi'} \sum_{j=1}^{m_\pi} Z_{\pi'}(\kappa, f_j^\pi)^*\langle \mu, f_j^\pi\rangle\langle \lambda, \delta_\X^{-1/2} \nu\rangle\\
    \vspace{-1.0mm}&= \sum_{j=1}^{m_{\pi'}} Z_{\pi'}(\kappa, f_j^{\pi'})^* \varphi_{\XGX}(Z_\pi(\mu, \nu)Z_{\pi'}(f_j^{\pi'},\lambda)^*)\\
    &= (\id \odot \varphi_{\XGX})((1\otimes a)\Delta_{\XGX}(b)),
\end{align*}
as desired.
\end{proof}

 Let us write $\mathscr{U}_{\XGX}:= \operatorname{Lin}_\C(\mcO(\XGX), \C)$ for the algebraic dual of $\mcO(\XGX)$. We endow it with the (unital) $*$-algebra structure given by
$$(\omega \star \omega')(z)= (\omega \odot \omega')\Delta_{\XGX}(z), \quad \omega^*(z)= \overline{\omega(S_{\XGX}(z)^*)}, \quad \omega, \omega'\in \mathscr{U}_{\XGX}, \quad z \in \mcO(\XGX).$$

The following result is simply a reformulation of the work that has been done so far:

\begin{Prop}\label{duality}
    Given  $\pi \in \Rep(\G)$, there is a unique unital $*$-representation
    $\mathscr{F}_\pi: \mathscr{U}_{\XGX}\to B(\mcG_\pi)$ satisfying
    $\langle \mu, \mathscr{F}_\pi(\omega)\nu\rangle = \omega(Z_\pi(\mu, \nu))$ for all $\mu, \nu \in \mcG_\pi$ and all $\omega \in \mathscr{U}_{\XGX}.$
    The induced map
    $$\mathscr{F}: \mathscr{U}_{\XGX}\to \prod_{\pi\in \Irr(\G)} B(\mcG_\pi): \omega \mapsto (\mathscr{F}_{\pi}(\omega))_{\pi\in \Irr(\G)}$$
    is a $*$-algebra isomorphism.
\end{Prop}

\subsection{Modular data} By analogy with the formulas \eqref{aa}, \eqref{bb} and \eqref{cc}, we define the modular group, the scaling group and the unitary antipode on $\mcO(\XGX)$ via the formulas:
\begin{align*}
    &\sigma_z^{\XGX}: \mcO(\XGX)\to \mcO(\XGX): Z_\pi(\mu, \nu) \mapsto Z_\pi(\delta_\X^{i\bar{z}/2}\mu, \delta_\X^{-iz/2}\nu), \\
    &\tau_z^{\XGX}: \mcO(\XGX)\to \mcO(\XGX): Z_\pi(\mu, \nu)\mapsto  Z_\pi(\delta_\X^{i\bar{z}/2}\mu, \delta_\X^{iz/2}\nu),\\
    &R_{\XGX}: \mcO(\XGX)\to \mcO(\XGX): Z_\pi(\mu, \nu) \mapsto Z_\pi(\delta_\X^{-1/4}\nu, \delta_\X^{1/4}\mu)^*.
\end{align*}
It is clear that the maps $\sigma_z^{\XGX}$ and $\tau_z^{\XGX}$ define algebra homomorphisms satisfying
$$\sigma_z^{\XGX}(a)^* = \sigma_{\bar{z}}^{\XGX}(a^*), \quad \tau_z^{\XGX}(a)^*= \tau_{\bar{z}}^{\XGX}(a^*), \quad z\in \C, \quad a \in \mcO(\XGX).$$
It is easy to verify that $\sigma_{-i}^{\XGX}$ is KMS for $\varphi_{\XGX}$.

We have expected identities, such as $R_{\XGX}^2 =  \id$, $S_{\XGX}= R_{\XGX}\circ \tau_{-i/2}^{\XGX}, S_{\XGX}^2 = \tau_{-i}^{\XGX}$ as well as $\Delta_{\XGX}\circ \tau_z^{\XGX}= (\tau_z^{\XGX}\odot \tau_z^{\XGX})\circ \Delta_{\XGX}$ for $z\in \C$.

\section{Universal version of the compact quantum hypergroup}
\label{universalversion}

In general, the $*$-algebra $\mcO(\XGX)$ does not admit a universal $C^*$-envelope \cite{DCDR25}*{Remark 2.10}. However, $\mcO(\XGX)$ can be embedded in a useful larger $*$-algebra, which does admit a universal $C^*$-envelope. More concretely, we endow the algebraic tensor product $\mcO(\X)\odot \mathscr{U}_\G\odot \mcO(\bar{\X})$ with the unital $*$-algebra structure
\begin{align*}
    (x\otimes \chi \otimes \overline{y})(x'\otimes \chi'\otimes \overline{y'}) &= xx_{(0)}' \otimes \chi(x_{(1)}'-)\chi'(-R_\G(y_{(1)})^*) \otimes \overline{y_{(0)}y'},\\
    (x\otimes \chi \otimes \overline{y})^* &= x_{(0)}^* \otimes \chi^*(x_{(1)}^* - R_\G(y_{(1)})) \otimes \overline{y_{(0)}^*}.
\end{align*}
The (non-unital) $*$-subalgebra $\mcO(\X)\odot \mathcal{U}_\G \odot \mcO(\bar{\X})$ will be denoted by $\mcO(\XrX)$ and we will call it the \emph{double crossed product $*$-algebra} \cites{AS21,DCDR25}. It is idempotent and non-degenerate \cite{DR25d}*{Lemma 3.4.5}, so one can consider its multiplier $*$-algebra $M(\mcO(\XrX))$. The three $*$-algebras $\mcO(\X), \mathscr{U}_\G, \mcO(\bar{\X})$ can then be naturally identified as unital $*$-subalgebras of $M(\mcO(\XrX))$ \cite{DR25d}*{Proposition 3.4.6}. Under these embeddings, the product $x\chi\overline{y}$ inside the multiplier algebra corresponds exactly to the elementary tensor $x\otimes \chi \otimes \bar{y}$, where $x\in \mcO(\X), \chi\in \mathcal{U}_\G, y \in \mcO(\bar{\X})$. Further, inside $M(\mcO(\XrX))$, we have the following useful commutation rules:
$$x\overline{y}= \overline{y}x, \quad \chi x = x_{(0)}\chi(x_{(1)}-), \quad \bar{y}\chi = \chi(-R_\G(y_{(1)})^*)\overline{y_{(0)}}, \quad x,y \in \mcO(\X), \quad \chi\in \mathscr{U}_\G.$$
The $*$-algebra $\mcO(\X\rtimes \G \ltimes \bar{\X})$ admits a universal $C^*$-envelope, denoted by $C^u_0(\XrX)$, in which $\mcO(\XrX)$ embeds \cite{DR25d}*{Proposition 3.4.7}.

Given $\mcH=(\mcH, \pi_\mcH, \rho_\mcH, U_\mcH) \in \Corr^\G(L^\infty(\X), L^\infty(\X))$, there is an induced non-degenerate $*$-representation
$$\theta^\mcH: \mcO(\XrX)\to B(\mcH): x\chi\bar{y}\mapsto \pi_\mcH(x)U_\mcH(\chi)\rho_\mcH(y^*).$$
Here, we recall the natural inclusion $\mathcal{U}_\G\subseteq L^1(\G)$ and the notation $U_\mcH(\chi) = (\id \otimes \chi)(U_\mcH)$.
In fact, every non-degenerate $*$-representation $\mcO(\XrX)\to B(\mcH)$ arises in this way \cite{DCDR25}*{Proposition 2.14}. It follows that
\begin{equation}\label{universalnormcrossed}
    \|s\|_{C^u_0(\XrX)}= \sup_{\mcH\in \Corr^\G(L^\infty(\X), L^\infty(\X))} \|\theta^\mcH(s)\|, \quad s\in \mcO(\XrX).
\end{equation}

\begin{Prop}\label{corner} \cite{DCDR25}*{Proposition 2.8}
    Consider the linear maps
\begin{align*}
    &\kappa\colon \mcO(\XGX)\to \varphi_\G\mcO(\XrX)\varphi_\G\colon \sum_{j=1}^{n} x_j\otimes \overline{y_j} \mapsto \sum_{j=1}^n x_j \delta_\G^{1/2} \overline{y_j} \varphi_\G,\\
    &E\colon \mcO(\XrX)\to \mcO(\XGX)\colon x\overline{y}\chi\mapsto \chi(1)\varphi_\G(x_{(1)}\sigma_{-i/2}^\G(y_{(1)}^*))x_{(0)}\otimes  \overline{y_{(0)}}.
\end{align*}
Then $\kappa(E(z))= \varphi_\G z \varphi_\G$  for all $z\in \mcO(\XrX)$ and $\kappa$ is a $*$-isomorphism which has the restriction of $E$ as inverse.
\end{Prop}
The map $\kappa$ acts on matrix coefficients via
\begin{align}
    \label{kappamatrix}\kappa(Z_\pi(\mu, \nu)) &= \sum_{j=1}^{d_\pi} X_\pi(\mu, e_j^\pi) Y_\pi(\delta_\G^{-1/2} e_j^\pi,\nu)\varphi_\G, \quad \pi \in \Rep(\G), \quad \mu, \nu \in \mcG_\pi.
\end{align}
We now endow $\mcO(\XGX)$ with the $C^*$-norm
$$\|z\|_u:= \|\kappa(z)\|_{C^u_0(\XrX)}, \quad z\in \mcO(\XGX).$$
The $C^*$-completion of $\mcO(\XGX)$ with respect to this norm will be denoted by $C^u(\XGX)$. Evidently, the $*$-isomorphism $\kappa$ from Proposition \ref{corner} then extends to a $*$-isomorphism
$$\kappa^u: C^u(\XGX)\to \varphi_\G C^u_0(\XrX)\varphi_\G.$$
Given $\mcH\in \Corr^\G(L^\infty(\X), L^\infty(\X))$, the composition $\theta^\mcH\circ \kappa$ leads to the unital $*$-representation $$\theta_\square^\mcH: \mcO(\XGX)\to B(U_\mcH(\varphi_\G)\mcH)$$ which acts on matrix coefficients via
\begin{equation}\label{usethis}
\theta_\square^{\mcH}(Z_\pi(\mu, \nu)) = \sum_{j=1}^{d_\pi} \pi_\mcH(X_\pi(\mu, e_j^\pi)) \rho_\mcH(X_\pi(\delta_\X^{-1/4}\nu, \delta_\G^{-1/4}e_j^\pi)^*), \quad \pi \in \Rep(\G), \quad \mu, \nu \in \mcG_\pi.    
\end{equation}

We then find, making use of \eqref{universalnormcrossed},
\begin{equation}\label{cornernorm}
    \|z\|_u = \sup_{\mcH \in \Corr^\G(L^\infty(\X), L^\infty(\X))} \|\theta_\square^{\mcH}(z)\|, \quad z \in \mcO(\XGX).
\end{equation}
In particular, given $\mcH\in \Corr^\G(L^\infty(\X), L^\infty(\X))$, the $*$-representation $\theta_\square^\mcH: \mcO(\XGX)\to B(U_\mcH(\varphi_\G)\mcH)$ extends uniquely to a $*$-representation $\theta_\square^\mcH: C^u(\XGX)\to B(U_\mcH(\varphi_\G)\mcH)$.

We now show that the algebraic compact quantum hypergroup structure on $\mcO(\XGX)$ carries over to the $C^*$-algebra $C^u(\XGX)$, so that we obtain a $C^*$-algebraic compact quantum hypergroup in the sense of Chapovsky and Vainerman \cite{CV99}. Let us first recall the relevant definitions.

\begin{Def}\cite{CV99}*{Definition 1.1}  We call $(A, \Delta, \epsilon, \star)$ a hypergroup structure on the unital $C^*$-algebra $A$ if $\Delta: A\to A\otimes A$ is a unital positive linear map satisfying $(\Delta\otimes \id)\Delta=  (\id \otimes \Delta)\Delta$, $\epsilon: A \to \C$ is an algebra homomorphism satisfying $(\epsilon\otimes \id)\Delta= \id = (\id \otimes \epsilon)\Delta$ and $\star: A\to A$ is an anti-linear, multiplicative, $*$-preserving involution satisfying $\Delta^{\op}\circ \star= (\star\otimes \star)\circ \Delta$
\end{Def}

\begin{Def} \cite{CV99}*{Definition 2.2}
    Let $(A, \Delta, \epsilon, \star)$ be a hypergroup structure on the unital $C^*$-algebra $A$. Given $\omega \in A^*$, define $\omega^+\in A^*$ by $\omega^+(a)= \overline{\omega(a^\star)}$ for $a\in A$. An element $a\in A$ is called \emph{positive definite} if $(\omega \otimes \omega^+)\Delta(a)\ge 0$ for all $\omega \in A^*$.
\end{Def}

\begin{Rem}\label{hame}
    For a compact quantum group $\G$, one has the density conditions
\begin{equation}\label{quantumdensity}
    [\Delta_\G(C(\G))(C(\G)\otimes 1)]= C(\G)\otimes C(\G)= [\Delta_\G(C(\G))(1\otimes C(\G))].
\end{equation}
However, in \cite{CV99}*{Example 2.5}, it was argued that asking for such a condition in the context of hypergroup structures is too strong. Rather, one proceeds as follows: consider a hypergroup structure $(A, \Delta, \epsilon, \star)$. If the linear span of the positive definite elements of $A$ is norm-dense in $A$, there exists a unique state $\varphi\in A^*$ such that $(\id \otimes \varphi)\Delta(a)= \varphi(a)1= (\varphi\otimes \id)\Delta(a)$ for all $a\in A$ \cite{CV99}*{Theorem 2.3}. The existence of the Haar state in the theory of $C^*$-algebraic compact quantum groups is typically proven using the density conditions \eqref{quantumdensity}. Thus, the density condition involving positive definite elements can be seen as a  replacement for \eqref{quantumdensity}. 
\end{Rem}

The following definition is slightly weaker than the one presented in \cite{CV99}\footnote{The author of this paper could not make rigorous sense of the expression $(\kappa\otimes  \id)\delta(a)$ for $a\in A_0$ which appears in \cite{CV99}*{Definition 4.1 (4.5)}.}, but is much more user-friendly. To the best of our knowledge, all known examples satisfying \cite{CV99}*{Definition 4.1} also satisfy the definition presented below.

\begin{Def}[cf.\ \cite{CV99}*{Definition 4.1}]\label{complicated}
    A $C^*$-algebraic compact quantum hypergroup consists of a hypergroup structure $(A, \Delta, \epsilon,\star)$ on the unital $C^*$-algebra $A$ together with a continuous one-parameter group of $*$-automorphisms $\tau_t: A\to A$ ($t\in \R$), such that the following properties are satisfied:
    \begin{itemize}
        \item $\Delta: A\to A\otimes A$ is completely positive.
        \vspace{-1.5mm}\item The linear span of the positive definite elements is norm-dense in $A$.
        \vspace{-1.5mm}\item There exists a norm-dense $*$-subalgebra $A_0\subseteq A$  such that $A_0^\star\subseteq A_0$ and $\Delta(A_0)\subseteq A_0\odot A_0$ and such that the one-parameter group $\tau_t$ can be extended to a complex one-parameter group $\tau_z$ of algebra automorphisms of $A_0$.
       \vspace{-1.5mm} \item For all $z\in \C$, we have $\Delta\circ \tau_z= (\tau_z\odot\tau_z)\circ \Delta$ and $\varphi\circ \tau_z=\varphi$ on $A_0$.\footnote{Here, $\varphi: A\to \C$ is the Haar state, which automatically exists by the preceding assumptions (see Remark \ref{hame}).}
        \vspace{-1.5mm}\item $\varphi$ is faithful on $A_0$.
        \vspace{-1.5mm}\item There exists $z_0\in \C$ such that with $S:= *\circ \tau_{z_0}\circ \star$, we have the strong left invariance
        $$(\id \otimes \varphi)((S\odot \id)(\Delta(a))(1\otimes b)) = (\id \otimes \varphi)((1\otimes a)\Delta(b)),\quad a,b\in A_0.$$
    \end{itemize}
\end{Def}
\begin{Rem}\label{completions}
The counit in the above definitions is assumed to exist as a character on the $C^*$-algebraic level. In many situations, this requirement is too strong. For example, it excludes non-coamenable reduced compact quantum groups as examples of $C^*$-algebraic compact quantum hypergroups. When working in the reduced setting, the same problem also persists for the examples of double coset spaces $\H\backslash \G/\H$ arising from an inclusion $\H\le \G$ of compact quantum groups, whenever $\H\backslash \G$ is not coamenable in the sense of \cite{AK24}. Therefore, Definition \ref{complicated} should rather be thought of as describing a `universal' version of a $C^*$-algebraic compact quantum hypergroup. It would be desirable to come up with a definition of a $C^*$-algebraic compact quantum hypergroup with less axioms than Definition \ref{complicated}, and which is rich enough to include examples where the counit does not exist on the $C^*$-algebraic level. 
\end{Rem}

Showing that the coproduct on $\mcO(\XGX)$ extends to a ucp map on the $C^*$-algebra $C^u(\XGX)$ will follow from the following key result, for which we recall the notion of the Connes fusion tensor product of equivariant correspondences from Subsection \ref{equivariant correspondences}.

 \begin{Prop}\label{keyresult}
  Given $\mcH, \mcK\in \Corr^\G(L^\infty(\X), L^\infty(\X))$, consider $\mcH\boxtimes\mcK:= (\mcH\boxtimes_{L^\infty(\X)}\mcK, \pi_\boxtimes, \rho_\boxtimes, U_\boxtimes)$. There is a natural isometry
  $\mathcal{I}: U_\mcH(\varphi_\G)\mcH\otimes U_\mcK(\varphi_\G)\mcK\to U_{\boxtimes}(\varphi_\G)(\mcH\boxtimes\mcK).$
  It satisfies
  \begin{equation}\label{intertwining}
      \mathcal{I}^* \theta_\square^{\mcH\boxtimes \mcK}(z) \mathcal{I} = (\theta_\square^\mcH\odot \theta_\square^\mcK)\Delta_{\XGX}(z), \quad z\in \mcO(\XGX).
  \end{equation}
 \end{Prop}

 \begin{proof} Let us write $\mathfrak{D}(\mcH, \varphi_\X):= \{\xi \in \mcH\mid \exists C\ge 0: \forall x\in L^\infty(\X):\|\rho_\mcH(x)\xi\|\le C\|\Lambda_\X(x^*)\|\}$ for the norm-dense subspace of $\mcH$ of \emph{left-bounded} vectors \cite{Tak03}*{IX.3 Lemma 3.3}. Each $\xi \in \mathfrak{D}(\mcH, \varphi_\X)$ induces an operator $L_\xi^{\varphi_\X}\in \mathscr{L}_{L^\infty(\X)}(L^2(\X), \mcH)$ defined by $L_\xi^{\varphi_\X}(J_\X \Lambda_\X(x^*))= \rho_\mcH(x)\xi$ for $x\in L^\infty(\X)$.

Consider the conditional expectation
    $$\mathbb{E}: \mathscr{L}_{L^\infty(\X)}(L^2(\X), \mcH)\to \mathscr{L}_{L^\infty(\X)}^{\G}(L^2(\X), \mcH): x\mapsto (\id \otimes \varphi_\G)(U_\mcH(x\otimes 1)U_\X^*).$$ 
    Fix $\xi \in \mathfrak{D}(\mcH, \varphi_\X)$.
     We compute for $a\in \mcO(\X)$ and $\eta \in \mcH$ that
    \begin{align*}
        \langle \eta, \mathbb{E}(L_\xi^{\varphi_\X})J_\X \Lambda_\X(a^*)\rangle &= \langle \eta \otimes \xi_\G, U_\mcH(L_\xi^{\varphi_\X}\otimes 1)U_\X^*(J_\X\Lambda_\X(a^*)\otimes \xi_\G)\rangle\\
        &= \langle \eta \otimes \xi_\G, U_\mcH(L_\xi^{\varphi_\X}\otimes 1) U_\X^*(J_\X \otimes J_{\hat{\G}})(\Lambda_\X(a^*)\otimes \xi_\G)\rangle\\
        &= \langle \eta \otimes \xi_\G, U_\mcH(L_\xi^{\varphi_\X}\otimes 1)(J_\X\otimes J_{\hat{\G}})U_\X(\Lambda_\X(a^*)\otimes \xi_\G)\rangle\\
        &= \langle \eta \otimes \xi_\G, U_\mcH(L_\xi^{\varphi_\X}\otimes 1)(J_\X\otimes J_{\hat{\G}})(\Lambda_\X(a_{(0)}^*)\otimes \Lambda_\G(a_{(1)}^*))\rangle\\
        &= \langle \eta \otimes \xi_\G, U_\mcH(\rho_\mcH(a_{(0)})\xi \otimes R_\G(a_{(1)})\xi_\G\rangle\\
        &= \langle \eta \otimes \xi_\G, (\rho_\mcH(a)\otimes 1)U_\mcH(\xi \otimes \xi_\G)\rangle \\
        &= \langle \eta, \rho_\mcH(a)U_\mcH(\varphi_\G)\xi\rangle,
    \end{align*}
    where we used \eqref{unitaryflip} in the third equality. It follows that
    $$\mathbb{E}(L_\xi^{\varphi_\X})J_\X\Lambda_\X(a^*)= \rho_\mcH(a)U_\mcH(\varphi_\G)\xi, \quad a \in L^\infty(\X).$$
Thus, we conclude that  $U_\mcH(\varphi_\G)\xi \in \mathfrak{D}(\mcH,\varphi_\X)$, and that $L_{U_\mcH(\varphi_\G)\xi}^{\varphi_\X}= \mathbb{E}(L_\xi^{\varphi_\X})\in \mathscr{L}_{L^\infty(\X)}^\G(L^2(\X), \mcH)$.

 If $x\in \mathscr{L}_{L^\infty(\X)}^\G(L^2(\X), \mcH)$, $y\in \mathscr{L}_{L^\infty(\X)}(L^2(\X), \mcH)$, $\eta\in U_\mcK(\varphi_\G)\mcK$ and $\zeta\in \mcK$, we calculate
 \begin{align*}
     \langle y\otimes_{L^\infty(\X)} \zeta, U_\boxtimes(\varphi_\G)(x\otimes_{L^\infty(\X)}\eta)\rangle &= \langle (y\otimes_{L^\infty(\X)} \zeta)\otimes \xi_\G, U_\boxtimes((x\otimes_{L^\infty(\X)}\eta)\otimes \xi_\G)\rangle\\
     &= \langle (y\otimes 1)\otimes_{L^\infty(\X)\ovot L^\infty(\G)} (\zeta\otimes \xi_\G), (x\otimes 1)\otimes_{L^\infty(\X)\ovot L^\infty(\G)} U_\mcK(\eta \otimes \xi_\G)\rangle\\
     &= \langle \zeta \otimes \xi_\G, (\pi_\mcK(\langle y, x\rangle_{L^\infty(\X)})\otimes 1) U_\mcK(\eta\otimes \xi_\G)\rangle\\
     &= \langle \zeta, \pi_\mcK(\langle y,x\rangle_{L^\infty(\X)})\eta\rangle,
 \end{align*}
 where the second equality uses that $x$ is intertwiner of $\G$-representations and the last equality uses that $\eta \in U_\mcK(\varphi_\G)\mcK$. We conclude that $x\otimes_{L^\infty(\X)}\eta\in U_{\boxtimes}(\varphi_\G)(\mcH\boxtimes\mcK)$ if $x\in \mathscr{L}_{L^\infty(\X)}^\G(L^2(\X), \mcH)$ and $\eta\in U_\mcK(\varphi_\G)\mcK$. Further, if $x,y \in \mathscr{L}^\G_{L^\infty(\X)}(L^2(\X), \mcH)$, we have
\begin{align*}
    (\pi_\X\otimes  \id)\alpha(\langle x,y\rangle_{L^\infty(\X)})&= U_\X(x^*y \otimes 1)U_\X^*= x^*y\otimes 1= \pi_\X(\langle x,y\rangle_{L^\infty(\X)})\otimes 1,
\end{align*}
so that $\langle x,y\rangle_{L^\infty(\X)}\in \C$ by the ergodicity of $\alpha$. We then obtain a unique isometry $\mathcal{I}: U_\mcH(\varphi_\G)\mcH\otimes U_\mcK(\varphi_\G)\mcK\to U_{\boxtimes}(\varphi_\G)(\mcH\boxtimes\mcK)$ such that
 $$\mathcal{I}(\xi \otimes \eta) = L_\xi^{\varphi_\X}\otimes_{L^\infty(\X)} \eta, \quad \xi \in U_\mcH(\varphi_\G)\mathfrak{D}(\mcH, \varphi_\X), \quad \eta \in U_\mcK(\varphi_\G)\mcK.$$

Let us now fix $\xi, \xi'\in U_{\mcH}(\varphi_\G)\mathfrak{D}(\mcH, \varphi_\X)$ and $\eta, \eta'\in U_{\mcK}(\varphi_\G)\mcK$. Put $x:= L_{\xi}^{\varphi_\X}$ and $x':= L_{\xi'}^{\varphi_\X}$. Given $\pi\in \Irr(\G)$, $\mu\in\mcG_\pi$ and $\zeta\in \mcH_\pi$, we compute
 \begin{align*}
        (\pi_\X\otimes \id)\alpha(\langle x', \pi_\mcH(X_\pi(\mu, \zeta)) x\rangle_{L^\infty(\X)})&= U_\X((x')^*\pi_\mcH(X_\pi(\mu, \zeta))x\otimes 1)U_\X^*\\
        &= ((x')^*\otimes 1)U_\mcH(\pi_\mcH(X_\pi(\mu, \zeta))\otimes 1)U_\mcH^*(x\otimes 1)\\
        &= \sum_{j=1}^{d_\pi}(x')^*\pi_\mcH(X_\pi(\mu, e_j^\pi)) x \otimes U_\pi(e_j^\pi, \zeta)\\
        &= \sum_{j=1}^{d_\pi}\pi_\X(\langle x', \pi_\mcH(X_\pi(\mu, e_j^\pi)) x\rangle_{L^\infty(\X)})\otimes U_\pi(e_j^\pi, \zeta), 
    \end{align*}
    so that the map $\mcH_\pi\ni \zeta\mapsto \langle x', \pi_\mcH(X_\pi(\mu, \zeta)) x\rangle_{L^\infty(\X)}\in L^\infty(\X)$ defines an element of $\operatorname{Mor}(\pi, \X)$. Consequently, there exists a unique $\mu^{x',x}\in \mcG_\pi$ such that $\langle x', \pi_\mcH(X_\pi(\mu, \zeta)) x\rangle_{L^\infty(\X)} =  X_\pi(\mu^{x',x},\zeta)$ for all $\zeta \in \mcH_\pi$. If then also $\nu \in \mcG_\pi$, we compute (use \eqref{usethis} for the second and seventh equality and use \eqref{orthogonality1} for the sixth equality):
\begin{align*}
    &\langle \xi'\otimes \eta', (\theta_\square^\mcH \odot \theta_\square^\mcK)(\Delta_{\XGX}(Z_\pi(\mu, \nu)))(\xi\otimes \eta)\rangle\\
    = &\sum_{j=1}^{m_\pi}\langle \xi', \theta_\square^\mcH(Z_\pi(\mu, f_j^\pi))\xi\rangle \langle \eta', \theta_\square^\mcK(Z_\pi(f_j^\pi, \nu))\eta\rangle\\
    = &\sum_{j=1}^{m_\pi}\sum_{k=1}^{d_\pi} \langle \xi', \pi_\mcH(X_\pi(\mu, e_k^\pi))\rho_\mcH(X_\pi(\delta_\X^{-1/4}f_j^\pi, \delta_\G^{-1/4}e_k^\pi)^*)\xi\rangle \langle \eta', \theta_\square^\mcK(Z_\pi(f_j^\pi, \nu))\eta\rangle\\
    = &\sum_{j=1}^{m_\pi}\sum_{k=1}^{d_\pi}\langle x'\xi_\X, \pi_\mcH(X_\pi(\mu, e_k^\pi)) xJ_\X \Lambda_\X(X_\pi(\delta_\X^{-1/4}f_j^\pi, \delta_\G^{-1/4}e_k^\pi))\rangle\langle \eta', \theta_\square^\mcK(Z_\pi(f_j^\pi, \nu))\eta\rangle\\
    = &\sum_{j=1}^{m_\pi}\sum_{k=1}^{d_\pi}\langle \xi_\X, \pi_\X(\langle x', \pi_\mcH(X_\pi(\mu, e_k^\pi))x\rangle_{L^\infty(\X)} )\Lambda_\X(X_\pi(f_j^\pi, e_k^\pi)^*)\rangle\langle \eta', \theta_\square^\mcK(Z_\pi(f_j^\pi, \nu))\eta\rangle\\
    = &\sum_{j=1}^{m_\pi}\sum_{k=1}^{d_\pi}\varphi_\X(X_\pi(\mu^{x',x}, e_k^\pi)X_\pi(f_j^\pi, e_k^\pi)^*)\langle \eta', \theta_\square^\mcK(Z_\pi(f_j^\pi, \nu))\eta\rangle\\
    = & \langle \eta', \theta_\square^\mcK(Z_\pi(\mu^{x',x}, \nu))\eta\rangle\\
    = & \sum_{k=1}^{d_\pi} \langle \eta', \pi_\mcK(X_\pi(\mu^{x',x}, e_k^\pi))\rho_\mcK(X_\pi(\delta_\X^{-1/4}\nu, \delta_\G^{-1/4}e_k^\pi)^*)\eta\rangle\\
    = & \sum_{k=1}^{d_\pi}\langle \eta', \pi_\mcK(\langle x', \pi_\mcH(X_\pi(\mu, e_k^\pi))x\rangle_{L^\infty(\X)})\rho_\mcK(X_\pi(\delta_\X^{-1/4}\nu, \delta_\G^{-1/4}e_k^\pi)^*)\eta\rangle\\
    =& \sum_{k=1}^{d_\pi} \langle x'\otimes_{L^\infty(\X)} \eta', \pi_\mcH(X_\pi(\mu, e_k^\pi))x\otimes_{L^\infty(\X)} \rho_\mcK(X_\pi(\delta_\X^{-1/4}\nu, \delta_\G^{-1/4}e_k^\pi)^*)\eta\rangle\\
    =&\sum_{k=1}^{d_\pi} \langle x'\otimes_{L^\infty(\X)} \eta', \pi_\boxtimes(X_\pi(\mu, e_k^\pi))\rho_\boxtimes(X_\pi(\delta_\X^{-1/4}\nu, \delta_\G^{-1/4}e_k^\pi)^*)(x\otimes_{L^\infty(\X)} \eta)\rangle\\
    =& \langle \xi'\otimes \eta', \mathcal{I}^*\theta_\square^{\mcH\boxtimes \mcK}(Z_\pi(\mu, \nu)) \mathcal{I}(\xi \otimes \eta)\rangle.
\end{align*}

From this, we conclude that \eqref{intertwining} holds.
 \end{proof}

Showing that the unitary antipode $R_{\XGX}: \mcO(\XGX)\to \mcO(\XGX)$ extends to the universal level will follow from the following result. Recall the construction of the conjugate equivariant correspondence from Subsection \ref{equivariant correspondences}.

\begin{Prop} Given $\mcH \in \Corr^\G(L^\infty(\X), L^\infty(\X))$, we have
\begin{equation}\label{unitaryantipode}
    \|\theta_\square^{\overline{\mcH}}(z)\| = \|\theta_\square^\mcH(R_{\XGX}(z))\|, \quad z\in \mcO(\XGX).
\end{equation}
\end{Prop}
\begin{proof} We have $\overline{U_\mcH}(\varphi_\G)= C_\mcH U_\mcH(\varphi_\G)^*C_\mcH = C_\mcH U_\mcH(\varphi_\G) C_\mcH^*$, so it is clear that $\overline{U_\mcH}(\varphi_\G)\overline{\mcH}= \overline{U_\mcH(\varphi_\G)\mcH}$.

Given $\xi\in U_\mcH(\varphi_\G)\mcH$, a standard computation using matrix coefficients and \eqref{usethis} shows that
$$\theta_\square^{\overline{\mcH}}(z)\overline{\xi}= \overline{\theta^\mcH_\square(R_{\XGX}(z))^*\xi}, \quad z\in \mcO(\XGX).$$
    Equality \eqref{unitaryantipode} then immediately follows.
\end{proof}

 \begin{Theorem}\label{mainCV} The following properties hold:
 \begin{enumerate}
     \item  The counit $\epsilon_{\XGX}: \mcO(\XGX)\to \C$ extends to a character $\epsilon_{\XGX}^u: C^u(\XGX)\to \C$.
  \item\vspace{-1.5mm} The coproduct $\Delta_{\XGX}: \mcO(\XGX)\to \mcO(\XGX)\odot \mcO(\XGX)$ extends to a ucp map $\Delta_{\XGX}^u: C^u(\XGX)\to C^u(\XGX)\otimes C^u(\XGX)$.
\item\vspace{-1.5mm} Given $t\in \R$, the $*$-automorphism $\tau^{\XGX}_t: \mcO(\XGX)\to \mcO(\XGX)$ extends to a $*$-automorphism $\tau_{t}^{\XGX,u}: C^u(\XGX)\to C^u(\XGX)$.
\item\vspace{-1.5mm} The unitary antipode $R_{\XGX}: \mcO(\XGX)\to \mcO(\XGX)$ extends to an anti-$*$-isomorphism $R_{\XGX}^u: C^u(\XGX)\to C^u(\XGX)$.
\item\vspace{-1.5mm} Writing $a^\star:= R_{\XGX}^u(a)^*$ for $a\in C^u(\XGX)$, the data $(C^u(\XGX), \Delta_{\XGX}^u, \epsilon_{\XGX}^u, \star, \tau_t^{\XGX,u})$ defines a $C^*$-algebraic compact quantum hypergroup in the sense of Definition \ref{complicated}.
     \end{enumerate}
 \end{Theorem}

 \begin{proof} (1) Given the trivial $\G$-$L^\infty(\X)$-$L^\infty(\X)$-correspondence $L^2(\X)$, we observe that $U_\X(\varphi_\G)L^2(\X)= \C \xi_\X\cong \C$. Therefore, using the formula \eqref{usethis}, we compute for $\pi \in \Rep(\G)$ and $\mu, \nu \in \mcG_\pi$ that
\begin{align*}
    \theta_\square^{L^2(\X)}(Z_\pi(\mu, \nu)) \xi_\X &= \sum_{j=1}^{d_\pi} \pi_\X(X_\pi(\mu, e_j^\pi)) \Lambda_\X(\sigma_{i/2}^\X(X_\pi(\delta_\X^{-1/4}\nu, \delta_\G^{-1/4}e_j^\pi))^*)\\
    &= \sum_{j=1}^{d_\pi} \Lambda_\X(X_\pi(\mu, e_j^\pi)X_\pi(\nu, e_j^\pi)^*) = \langle \mu, \nu\rangle \xi_\X,
\end{align*}
so that 
$\theta_\square^{L^2(\X)}: \mcO(\XGX)\to \C$ coincides with the counit $\epsilon_{\XGX}: \mcO(\XGX)\to \C$. It then follows from \eqref{cornernorm} that $\epsilon_{\XGX}: \mcO(\XGX)\to \C$ is bounded with respect to $\|\cdot\|_u$.
 
 (2) Consider a universal $\G$-$L^\infty(\X)$-$L^\infty(\X)$ correspondence $\mcH_u$, in the sense that $\mcH_u$ contains a unitary copy of every cyclic $\G$-$L^\infty(\X)$-$L^\infty(\X)$-correspondence \cite{DCDR24}*{Section 2}, \cite{DR25d}*{Proposition 2.3.4}. Then for $z\in \mcO(\XGX)$, we have $\|z\|_u = \|\theta^{\mcH_u}_\square(z)\|$. Thus, $\theta_\square^{\mcH_u}: C^u(\XGX)\to B(U_{\mcH_u}(\varphi_\G)\mcH_u)$ is a faithful $*$-representation. Consider the natural isometry $\mathcal{I}: U_{\mcH_u}(\varphi_\G)\mcH_u\otimes U_{\mcH_u}(\varphi_\G)\mcH_u\to U_{\mcH_u\boxtimes  \mcH_u}(\varphi_\G)(\mcH_u\boxtimes \mcH_u)$ from Proposition \ref{keyresult}. Given $z\in \mcO(\XGX)$, making use of a standard property of the minimal $C^*$-algebra tensor product, as well as using \eqref{intertwining}, we find
     \begin{align*}
         \|\Delta_{\XGX}(z)\|_{C^u(\XGX)\otimes C^u(\XGX)}&=\|(\theta_{\square}^{\mcH_u}\odot \theta_\square^{\mcH_u})\Delta_{\XGX}(z)\|_{B(U_{\mcH_u}(\varphi_\G)\mcH_u\otimes U_{\mcH_u}(\varphi_\G)\mcH_u)}\\
         &= \|\mathcal{I}^* \theta_\square^{\mcH_u\boxtimes \mcH_u}(z)\mathcal{I}\|_{B(U_{\mcH_u}(\varphi_\G)\mcH_u\otimes U_{\mcH_u}(\varphi_\G)\mcH_u)}\\
         &\le \|\theta_\square^{\mcH_u\boxtimes \mcH_u}(z)\|_{B(U_{\mcH_u\boxtimes \mcH_u}(\varphi_\G)(\mcH_u\boxtimes \mcH_u))} \le \|z\|_{C^u(\XGX)}.
         \end{align*}
Thus, $\Delta_{\XGX}: \mcO(\XGX)\to \mcO(\XGX)\odot \mcO(\XGX)$ extends uniquely to a contraction
$$\Delta_{\XGX}^u: C^u(\XGX)\to C^u(\XGX)\otimes C^u(\XGX)$$
which satisfies
$(\theta_\square^{\mcH_u}\otimes \theta_{\square}^{\mcH_u})\Delta_{\XGX}^u(z) = \mathcal{I}^*\theta_\square^{\mcH_u\boxtimes \mcH_u}(z)\mathcal{I}$ for all $z\in C^u(\XGX)$. From the latter identity, it follows that $\Delta_{\XGX}^u$ is a ucp map.

(3) Consider the scaling groups of $*$-automorphisms
\begin{align*}
    &\tau_t^\X: \mcO(\X)\to \mcO(\X): X_\pi(\mu, \xi)\mapsto X_\pi(\delta_\X^{it/2} \mu,\delta_\G^{it/2}\xi),\\
    &\tau_t^{\hat{\G}}: \mathscr{U}_\G\to \mathscr{U}_\G: \chi\mapsto \chi\circ \tau_{-t}^\G,\\
    &\tau_t^{\bar{\X}}: \mcO(\bar{\X})\to \mcO(\bar{\X}): Y_\pi(\xi, \mu)\mapsto Y_\pi(\delta_\G^{it/2}\xi, \delta_\X^{it/2}\mu),
\end{align*}
and define the linear map $\tau_t^{\XrX}: \mcO(\X\rtimes \G \ltimes \bar{\X})\to \mcO(\X\rtimes \G \ltimes \bar{\X}): x\chi\bar{y}\mapsto \tau_t^\X(x) \tau^{\hat{\G}}_t(\chi)\tau_t^{\bar{\X}}(\bar{y}).$
Using the identities $(\tau_t^\X\odot \tau_t^\G)\circ \alpha= \alpha\circ \tau_t^\X$ and $(\tau_t^\G \odot \tau_t^{\bar{\X}})\circ  \bar{\alpha}= \bar{\alpha}\circ \tau_t^{\bar{\X}}$, it follows that $\tau_t^{\XrX}$ is a $*$-isomorphism. Hence, it extends to a $*$-isomorphism
$\tau_t^{\XrX,u}: C^u_0(\XrX)\to C^u_0(\XrX).$
Using \eqref{kappamatrix}, we calculate for $\pi\in \Rep(\G)$ and $\mu, \nu\in \mcG_\pi$ that
\begin{align*}
    \tau_t^{\XrX,u} \kappa(Z_\pi(\mu, \nu)) &= \sum_{j=1}^{m_\pi} \tau_t^\X(X_\pi(\mu, e_j^\pi))\tau_t^{\bar{\X}}(Y_\pi(\delta_\G^{-1/2}e_j^\pi, \nu))\tau_t^{\hat{\G}}(\varphi_\G)\\
    &= \sum_{j=1}^{m_\pi} X_\pi(\delta_\X^{it/2}\mu, e_j^\pi)Y_\pi(\delta_\G^{-1/2}e_j^\pi, \delta_\X^{it/2}\nu) \varphi_\G = \kappa(\tau_t^{\XGX}(Z_\pi(\mu, \nu))),
\end{align*}
so $(\kappa^{u})^{-1}\circ \tau_t^{\XrX,u}\circ \kappa^u: C^u(\XGX)\to C^u(\XGX)$ extends $\tau_t^{\XGX}: \mcO(\XGX)\to \mcO(\XGX)$.

(4) Given $z\in \mcO(\XGX)$, we have
\begin{align*}
    \|R_{\XGX}(z)\|_u &= \sup_{\mcH\in \Corr^\G(L^\infty(\X), L^\infty(\X))}\|\theta_\square^\mcH(R_{\XGX}(z))\|=  \sup_{\mcH\in \Corr^\G(L^\infty(\X), L^\infty(\X))}\|\theta_\square^{\overline{\mcH}}(z)\|= \|z\|_u,
\end{align*}
where the second equality follows from \eqref{unitaryantipode} and the last equality follows from the fact that $\overline{\overline{\mcH}}\cong \mcH$ as $\G$-$L^\infty(\X)$-$L^\infty(\X)$-correspondences.

(5) Given $\pi \in \Rep(\G)$, choose an orthonormal basis $\{f_j^\pi\}_{j=1}^{m_\pi}$ for $\mcG_\pi$ for which $\delta_\X$ becomes diagonal, say with (positive) eigenvalues $\delta_\X f_j^\pi = \delta_{\X,j} f_j^\pi$. If $\mu \in \mcG_\pi$, then for all $\omega \in \mathscr{U}_{\XGX}$,
    \begin{align*}
        \sum_{j=1}^{m_\pi} \omega(Z_\pi(\delta_\X^{-1/4}\mu, f_j^\pi)) \overline{\omega(Z_\pi(\delta_\X^{-1/4}\mu, \delta_\X^{1/4}f_j^\pi))} = \sum_{j=1}^{m_\pi} \delta_{\X,j}^{1/4}|\omega(Z_\pi(\delta_\X^{-1/4}\mu, f_j^\pi))|^2\ge 0.
    \end{align*} 
    Thus, the element $Z_\pi(\delta_\X^{-1/4}\mu, \mu)$ is positive definite. By polarization, it follows that $Z_\pi(\mu, \nu)$ is a linear combination of four positive definite elements for all $\pi \in \Rep(\G)$ and all $\mu, \nu \in \mcG_\pi$. Consequently, the linear span of the positive definite elements contains $\mcO(\XGX)$ and is thus norm-dense in $C^u(\XGX)$.

  With $A:= C^u(\XGX)$, one can then take $A_0:= \mcO(\XGX)$ in Definition \ref{complicated}. With $z_0:= i/2$, we have $S_{\XGX}= *\circ \tau_{z_0}^{\XGX}\circ \star$. Let us remark that by uniqueness of the Haar state on an algebraic compact quantum hypergroup, the Haar state $\varphi_{\XGX}^u: C^u(\XGX)\to \C$ (which exists by the result discussed in Remark \ref{hame}) necessarily extends $\varphi_{\XGX}: \mcO(\XGX)\to \C$, so it is faithful on $A_0= \mcO(\XGX)$. Checking the remaining axioms in Definition \ref{complicated} is then straightforward using the results in Section \ref{matrixcoefficients}. 
 \end{proof}

\section{Reduced version of the compact quantum hypergroup}
\label{reducedversion}
Given a compact quantum group $\G$ and an ergodic action $L^\infty(\X)\stackrel{\alpha}\curvearrowleft \G$, the $*$-algebra
$\mcO(\XGX)$ obtains a norm $\|\cdot\|_r$ through the inclusion $\mcO(\XGX)\subseteq L^\infty(\X)\ovot L^\infty(\bar{\X})$, which we will call the \emph{reduced norm}. The norm-closure (resp.\ the $\sigma$-weak closure) of $\mcO(\XGX)$ inside $L^\infty(\X)\ovot L^\infty(\bar{\X})$ will be denoted by $C_\mcO(\XGX)$ (resp.\ $L^\infty_\mcO(\XGX)$). It is then natural to ask if the coassociative map
$\Delta_{\XGX}: \mcO(\XGX)\to \mcO(\XGX)\odot \mcO(\XGX)$ extends to a ucp map
$$\Delta_{\XGX}^r: C_\mcO(\XGX)\to C_\mcO(\XGX)\otimes C_\mcO(\XGX),$$
or even to a normal ucp map
$$\Delta_{\XGX}^r: L^\infty_\mcO(\XGX)\to L^\infty_\mcO(\XGX)\ovot L^\infty_\mcO(\XGX).$$
This turns out to be true. In fact, we will construct an appropriate coassociative normal ucp map in the general framework of locally compact quantum groups (under an additional integrability assumption on the action, which is automatically satisfied in the compact setting).

\subsection{Construction of a coassociative map on the cotensor product}\label{sec3}

Let $\G$ be a locally compact quantum group and let $L^\infty(\X)\stackrel{\alpha}\curvearrowleft \G$ be an ergodic action.

We define the von Neumann algebra
$$L^\infty(\XGX):= L^\infty(\X) \overset{\G}\square L^\infty(\bar{\X})= \{z\in L^\infty(\X)\ovot L^\infty(\bar{\X}): (\alpha\otimes \id)(z)= (\id \otimes \bar{\alpha})(z)\}.$$

\begin{Exa}\label{coidealsex}
    Let $L^\infty(\X)\subseteq L^\infty(\G)$ be a coideal von Neumann algebra, i.e.\ $L^\infty(\X)$ is a von Neumann subalgebra of $L^\infty(\G)$ such that $\Delta_\G(L^\infty(\X))\subseteq L^\infty(\X)\ovot L^\infty(\G)$. We then obtain an induced ergodic action $L^\infty(\X)\stackrel{\Delta_\G}\curvearrowleft\G$. The $*$-isomorphism
    $\phi: R_\G(L^\infty(\X))\to \overline{L^\infty(\X)}: x\mapsto \overline{R_\G(x^*)}$
    is $\G$-equivariant: \begin{equation}\label{sat}
    \overline{\Delta_\G} \phi(x)= (\id \otimes \phi) \Delta_\G(x), \quad x\in R_\G(L^\infty(\X)).
\end{equation}
The equivariance \eqref{sat} allows us to define the isometric $*$-homomorphism
\begin{equation}\label{identification2}
    L^\infty(\X)\cap R_\G(L^\infty(\X))\to L^\infty(\XGX): x \mapsto (\id \otimes \phi)\Delta_\G(x).
\end{equation}
We now argue that it is surjective. To do this, fix $z\in L^\infty(\XGX)\subseteq L^\infty(\X)\ovot L^\infty(\bar{\X})$. We then compute, making use of \eqref{sat} again, that
\begin{align*}
    (\Delta_\G \otimes \id)(\id \otimes \phi^{-1})(z)&= (\id \otimes \id \otimes \phi^{-1})(\Delta_\G \otimes \id)(z) = (\id \otimes \id \otimes \phi^{-1})(\id \otimes \overline{\Delta_\G})(z) = (\id \otimes \Delta_\G)(\id \otimes \phi^{-1})(z),
\end{align*}
so it follows from \cite{Vae01}*{Theorem 2.7} (applied to both the right action $L^\infty(\X)\curvearrowleft\G$ and the left action $\G\curvearrowright R_\G(L^\infty(\X))$) that there exists $x \in L^\infty(\X)\cap R_\G(L^\infty(\X))$ such that $z= (\id \otimes \phi)\Delta_\G(x)$. Thus, \eqref{identification2} is a $*$-isomorphism.
\end{Exa}

Let us now return to the general theory, for which we impose one further assumption: we assume that
the (ergodic) action $\alpha$ is \emph{integrable}, meaning that the normal and faithful weight $\varphi_\X:= (\id \otimes \varphi_\G)\circ \alpha: L^\infty(\X)_+\to [0, \infty]$ determined by
    $$\varphi_\G((\omega \otimes \id)\alpha(x)) = \omega(1)  \varphi_\X(x), \quad x\in L^\infty(\X)_+, \quad \omega \in L^1(\X)_+,$$
    is semi-finite. This allows us to identify $L^2(\X)=L^2(L^\infty(\X), \varphi_\X)$. In this generality, we show that there is a natural normal ucp map
$\Delta_{\XGX}^r: L^\infty(\XGX)\to L^\infty(\XGX)\ovot L^\infty(\XGX)$
which is coassociative (but in general not multiplicative). 

 To construct this map, we make use of the Galois map associated to the action $L^\infty(\X)\stackrel{\alpha}\curvearrowleft \G$ \cites{DC09a, DC11}. More concretely, consider the isometry
$$G: L^2(\X)\otimes L^2(\X)\to L^2(\X)\otimes L^2(\G), \quad G(\Lambda_{\varphi_\X}(x)\otimes \Lambda_{\varphi_\X}(y))= (\Lambda_{\varphi_\X}\otimes \Lambda_{\varphi_\G})(\alpha(x)(y\otimes 1)), \quad x,y \in \mathscr{N}_{\varphi_\X}.$$
Define the \emph{Galois isometry} $\tilde{G}:= \Sigma \circ G: L^2(\X)\otimes L^2(\X)\to L^2(\G)\otimes L^2(\X)$. 
We summarize some relevant properties of these maps in the following result. See \cite{DC09a}*{Sections 6.4 \& 7.2} and \cite{DC11}*{Section 2} for much more information and complete proofs.

\begin{Prop}\label{Galois} The following properties hold:
\begin{enumerate}[label=(\Alph*)]
    \item  $G(x\otimes 1)= \alpha(x) G$ for $x\in L^\infty(\X)$, \label{11}
    \item \label{22} \vspace{-1.5mm}$G(1\otimes x')= (x'\otimes 1)G$ for $x'\in L^\infty(\X)'$,
    \item \label{33} \vspace{-1.5mm}$G^*(1\otimes g')G \in L^\infty(\X)'\ovot L^\infty(\X)$ for $g'\in L^\infty(\G)'$,
\item \label{44} \vspace{-1.5mm}$\tilde{G}_{12}U_{\X,13}= V_{\G, 13}\tilde{G}_{12},$
    \item \label{55} \vspace{-1.5mm}
$(\id \otimes \alpha^{\op})(\tilde{G})= W_{\hat{\G},12}\tilde{G}_{13}$.\footnote{Note that \ref{22} ensures that $\tilde{G}\in B(L^2(\X), L^2(\G))\ovot L^ \infty(\X)$.}
\end{enumerate}
\end{Prop}
\begin{proof} Statement \ref{11} is obvious. The statement \ref{22} is proven in \cite{DC09a}*{Lemma 6.4.10 (3)}. 
If $g'\in L^\infty(\G)'$, $x\in L^\infty(\X)$ and $x'\in L^\infty(\X)'$, then 
\begin{align*}
    G^*(1\otimes g')G(x\otimes 1) \stackrel{\ref{11}}=& G^*\alpha(x) (1\otimes g') G \stackrel{\ref{11}}= (x\otimes 1)G^*(1\otimes g')G,\\
    G^*(1\otimes  g')G(1\otimes x') \stackrel{\ref{22}}=& G^*(x'\otimes g')G \stackrel{\ref{22}}= (1\otimes x')G^*(1\otimes  g')G,
\end{align*}
which shows that \ref{33}  holds. The statement \ref{44} is proven in \cite{DC09a}*{Lemma 7.2.3} and the statement \ref{55} is proven in \cite{DC09a}*{Proposition 7.2.5} (an inspection of the proofs of the latter two statements shows that unitarity of $\tilde{G}$ is not required).
\end{proof}

Since $u_\G L^\infty(\G) u_\G = L^\infty(\G)'$, \ref{33} allows us to define the normal ucp map
\begin{equation}\label{theta}
    \theta: L^\infty(\G)\to L^\infty(\bar{\X})\ovot L^\infty(\X): x \mapsto (I \otimes \id)(\tilde{G}^*(u_\G x u_\G\otimes 1)\tilde{G}),
\end{equation}
where we recall the $\G$-equivariant $*$-isomorphism $I: L^\infty(\X)'\cong L^\infty(\bar{\X}): \rho_\X(x^*)\mapsto \overline{x}$.

\begin{Prop}\label{identities} We have
$(\theta \otimes \id)\circ \Delta_\G= (\id \otimes \alpha)\circ \theta$ and $(\bar{\alpha}\otimes \id)\circ \theta = (\id \otimes \theta)\circ \Delta_\G.$
\end{Prop}
\begin{proof} Fix $x\in L^\infty(\G)$. We then calculate
    \begin{align*}
        (\alpha'\otimes \id)(\tilde{G}^*(u_\G xu_\G \otimes 1)\tilde{G})&= U_{\X,21}^* \tilde{G}_{23}^* (1\otimes u_\G x u_\G\otimes 1) \tilde{G}_{23}U_{\X,21}\\
        &\stackrel{\ref{44}}= \tilde{G}_{23}^* V_{\G,21}^* (1\otimes u_\G x u_\G \otimes 1)V_{\G,21}\tilde{G}_{23}\\
        &= \tilde{G}_{23}^* (1\otimes u_\G \otimes 1)W_{\G,12}^*(1\otimes x\otimes 1)W_{\G,12} (1\otimes u_\G \otimes 1)\tilde{G}_{23}\\
        &= \tilde{G}_{23}^*(1\otimes u_\G \otimes 1)(\Delta_\G(x)\otimes 1)(1\otimes u_\G \otimes 1)\tilde{G}_{23},
    \end{align*}
    so that 
   $(\bar{\alpha}\otimes \id)\theta(x)= (\id \otimes I \otimes \id)(\alpha'\otimes \id)(\tilde{G}^*(u_\G x u_\G \otimes 1)\tilde{G})= (\id \otimes \theta)\Delta_\G(x)$. On the other hand,
   \begin{align*}
       (\id \otimes \alpha^{\op})(\tilde{G}^*(u_\G x u_\G \otimes 1)\tilde{G}) &\stackrel{\ref{55}}= \tilde{G}_{13}^*W_{\hat{\G},12}^* (u_\G x u_\G \otimes 1 \otimes 1) W_{\hat{\G},12}\tilde{G}_{13}\\
       &= \tilde{G}_{13}^* W_{\G,21}(u_\G x u_\G \otimes 1\otimes 1)W_{\G,21}^* \tilde{G}_{13}\\
       &= \tilde{G}_{13}^* (u_\G \otimes 1 \otimes 1)V_{\G,12}(x\otimes 1 \otimes 1)V_{\G,12}^* (u_\G \otimes 1 \otimes 1)\tilde{G}_{13}\\
       &= \tilde{G}_{13}^*(u_\G \otimes 1 \otimes 1)(\Delta_\G(x)\otimes 1)(u_\G \otimes 1 \otimes 1)\tilde{G}_{13},
   \end{align*}
   so that
   $(\id \otimes \alpha)(\tilde{G}^*(u_\G xu_\G \otimes 1)\tilde{G}) = \tilde{G}_{12}^* (u_\G \otimes 1 \otimes 1)\Delta_\G(x)_{13}(u_\G \otimes 1 \otimes 1)\tilde{G}_{12}.$
   Applying $I\otimes \id \otimes \id$ to this expression leads to $(\id \otimes \alpha)\theta(x)=(\theta \otimes \id)\Delta_\G(x)$.
\end{proof}

\begin{Theorem}\label{mainsec3} Given $z\in L^\infty(\XGX)$, we have that
$$\Delta_{\XGX}^r(z):= (\id \otimes \theta \otimes \id)(\alpha\otimes \id)(z) = (\id \otimes \theta \otimes \id)(\id \otimes \bar{\alpha})(z)\in L^\infty(\XGX)\ovot L^\infty(\XGX).$$
Moreover, $\Delta_{\XGX}^r: L^\infty(\XGX)\to L^\infty(\XGX)\ovot L^\infty(\XGX)$ is coassociative.
\end{Theorem}

\begin{proof} By Proposition \ref{identities}, we see that
    $$(\id \otimes \theta)\alpha(L^\infty(\X))\subseteq  L^\infty(\XGX)\ovot L^\infty(\X), \quad (\theta \otimes \id)\bar{\alpha}(L^\infty(\bar{\X}))\subseteq L^\infty(\bar{\X})\ovot L^\infty(\XGX).$$
    Consequently
    $\Delta_{\XGX}^r(L^\infty(\XGX))\subseteq L^\infty(\XGX)\ovot L^\infty(\XGX).$
    
    Next, we view $L^\infty(\XGX)$ as having two tensor legs. Then for $z\in L^\infty(\XGX)$, we find
    \begin{align*}
        (\Delta_{\XGX}^r\otimes \id_{L^\infty(\XGX)})\Delta_{\XGX}^r(z)&= (\id\otimes \theta \otimes \id \otimes \id \otimes \id)(\alpha \otimes \id \otimes \id \otimes \id)(\id \otimes \theta \otimes \id)(\alpha\otimes \id)(z)\\
        &= (\id \otimes \theta \otimes \theta \otimes \id)(\alpha \otimes \id\otimes \id)(\alpha\otimes \id)(z)\\
        &= (\id \otimes \theta \otimes \theta \otimes \id)(\id \otimes \Delta_\G\otimes \id)(\alpha\otimes \id)(z)\\
        &=  (\id \otimes \theta \otimes \theta \otimes \id)(\id \otimes \Delta_\G \otimes \id)(\id \otimes \bar{\alpha})(z)\\
        &= (\id \otimes \theta \otimes \theta \otimes \id)(\id \otimes \id \otimes \bar{\alpha})(\id \otimes \bar{\alpha})(z)\\
         &= (\id \otimes \id \otimes \id \otimes \theta \otimes \id)(\id \otimes \id \otimes \id \otimes \bar{\alpha})(\id \otimes \theta \otimes \id)(\id \otimes \bar{\alpha})(z)\\
        &=   (\id_{L^\infty(\XGX)}\otimes \Delta_{\XGX}^r)\Delta_{\XGX}^r(z),
    \end{align*}
    and the coassociativity is proven.
\end{proof}

\subsection{Reduced compact quantum hypergroup structure}

We now restrict again to the case where $\G$ is a compact quantum group and an ergodic action $L^\infty(\X)\stackrel{\alpha}\curvearrowleft\G$ is given.

Note that we have the inclusions
\begin{equation}\label{inclusions}
    C_\mcO(\XGX)\subseteq C(\XGX), \quad L^\infty_\mcO(\XGX)\subseteq  L^\infty(\XGX),
\end{equation}
where we define $C(\XGX):= \{z\in C(\X)\otimes C(\bar{\X}): (\alpha\otimes \id)(z)= (\id \otimes \bar{\alpha})(z)\}$. It is not clear if the inclusions \eqref{inclusions} can be strict (see Remark \ref{completionss} for a brief discussion).

The next result shows that the normal ucp map
$\Delta_{\X\times_\G \bar{\X}}^r: L^\infty(\XGX)\to L^\infty(\XGX)\ovot L^\infty(\XGX)$
from Theorem \ref{mainsec3} extends $\Delta_{\XGX}: \mcO(\XGX)\to \mcO(\XGX)\odot \mcO(\XGX)$.
\begin{Prop} Given $\pi\in \Rep(\G)$ and $\xi, \eta \in \mcH_\pi$, we have
    $\theta(U_\pi(\xi, \eta))= \sum_{j=1}^{m_\pi} Y_\pi(\xi, f_j^\pi)\otimes X_\pi(f_j^\pi, \eta)$. Consequently, 
    \begin{equation}\label{actionmatrixcoefficients}
        \Delta_{\XGX}^r(Z_\pi(\mu, \nu))= \sum_{j=1}^{m_\pi} Z_\pi(\mu, f_j^\pi)\otimes Z_\pi(f_j^\pi, \nu), \quad \pi\in \Rep(\G), \quad \mu, \nu \in \mcG_\pi.
    \end{equation}
\end{Prop}
\begin{proof} Define the linear map
$$\theta': \mcO(\G)\to \mcO(\bar{\X})\odot \mcO(\X): U_\pi(\xi, \eta) \mapsto \sum_{j=1}^{m_\pi} Y_\pi(\xi, f_j^\pi)\otimes X_\pi(f_j^\pi, \eta).$$
We will show that $\theta = \theta'$ on $\mcO(\G)$.
Considering the natural isomorphism $L^\infty(\bar{\X})\cong L^\infty(\X)': \overline{x} \mapsto \rho_\X(x^*) = J_\X \pi_\X(x) J_\X$, we may as well regard $\theta'$ as a map $\mathcal{O}(\G)\to L^\infty(\X)'\ovot L^\infty(\X)$. We shall prove that
 $$G^*(1\otimes u_\G g u_\G)G = \theta'(g), \quad g \in \mathcal{O}(\G).$$
  Since $\xi_\X \otimes \xi_\X$ is separating for the von Neumann algebra $L^\infty(\X)'\ovot L^\infty(\X)\subseteq B(L^2(\X)\otimes L^2(\X))$, it is therefore sufficient to prove that
 $$G^*(1\otimes u_\G g u_\G) G(\xi_\X \otimes \xi_\X) = \theta'(g)(\xi_\X \otimes \xi_\X),$$
 or equivalently
 \begin{equation}\label{toprove}
     \langle \Lambda_\X(c)\otimes \Lambda_\X(d), G^*(1\otimes u_\G g u_\G )G(\xi_\X \otimes \xi_\X)\rangle = \langle \Lambda_\X(c)\otimes \Lambda_\X(d), \theta'(g)(\xi_\X\otimes \xi_\X)\rangle, \quad c,d \in \mathcal{O}(\X).
 \end{equation}
 We may further specify $g = U_{\pi_g}(\xi_g, \eta_g), c= X_{\pi_c}(\mu_c, \xi_c)^*$ and $d= X_{\pi_d}(\mu_d, \xi_d)$, where $\pi_g, \pi_c, \pi_d\in \Irr(\G)$. We calculate, making use of the orthogonality relations \eqref{EqPW1} and \eqref{orthogonality2}, that
 \begin{align*}
     &\langle \Lambda_\X(c)\otimes \Lambda_\X(d), G^*(1\otimes u_\G  g u_\G)G(\xi_\X \otimes \xi_\X)\rangle \\
     &= \varphi_\X(d^*c_{(0)}^*) \varphi_\G(c_{(1)}^* \sigma_{-i/2}^\G(R_\G(g)))\\
     &=\sum_{j=1}^{d_{\pi_c}} \varphi_\X(X_{\pi_d}(\mu_d, \xi_d)^*X_{\pi_c}(\mu_c, e_j^{\pi_c})) \varphi_\G(U_{\pi_c}(e_j^{\pi_c}, \xi_c)U_{\pi_g}(\eta_g, \delta_\G^{1/2} \xi_g)^*)\\
     &= \delta_{\pi_d, \pi_c}\delta_{\pi_c, \pi_g}\dim_q(\pi_c)^{-2}\sum_{j=1}^{d_{\pi_c}} \langle \mu_c, \delta_\X^{1/2}\mu_d\rangle \langle \xi_d, e_j^{\pi_c}\rangle \langle \xi_g, \xi_c\rangle \langle e_j^{\pi_c}, \eta_g\rangle\\
     &= \delta_{\pi_d, \pi_c}\delta_{\pi_c, \pi_g}\dim_q(\pi_c)^{-2} \langle \mu_c, \delta_\X^{1/2}\mu_d\rangle \langle \xi_g, \xi_c\rangle \langle \xi_d, \eta_g\rangle.
 \end{align*}
On the other hand, under the identification $\overline{L^\infty(\X)}\cong L^\infty(\X)'$, the map $\theta'$ is given by
$$\theta'(g) = \sum_{j=1}^{m_{\pi_g}} \rho_\X(X_{\pi_g}(\delta_\X^{-1/4}f_j^{\pi_g}, \delta_\G^{1/4}\xi_g)^*) \otimes X_{\pi_g}(f_j^{\pi_g}, \eta_g).$$
Consequently, making use of both the orthogonality relations \eqref{orthogonality1} and \eqref{orthogonality2}, we find:
\begin{align*}
    &\langle \Lambda_\X(c)\otimes \Lambda_\X(d), \theta'(g)(\xi_\X \otimes \xi_\X)\rangle\\
    &= \sum_{j=1}^{m_{\pi_g}} \langle \Lambda_\X(c)\otimes \Lambda_\X(d), \Lambda_\X(\sigma_{-i/2}^\X(X_{\pi_g}(\delta_\X^{-1/4}f_j^{\pi_g}, \delta_\G^{1/4}\xi_g)^*)) \otimes \Lambda_\X(X_{\pi_g}(f_j^{\pi_g}, \eta_g))\rangle\\
    &= \sum_{j=1}^{m_{\pi_g}} \varphi_\X(X_{\pi_c}(\mu_c, \xi_c) X_{\pi_g}(f_j^{\pi_g}, \delta_\G^{1/2}\xi_g)^*) \varphi_\X(X_{\pi_d}(\mu_d, \xi_d)^* X_{\pi_g}(f_j^{\pi_g}, \eta_g))\\
    &= \dim_q(\pi_c)^{-2} \delta_{\pi_c, \pi_g} \delta_{\pi_g, \pi_d}\sum_{j=1}^{m_{\pi_g}} \langle \xi_g, \xi_c\rangle \langle \mu_c, f_j^{\pi_g}\rangle \langle f_j^{\pi_g}, \delta_\X^{1/2}\mu_d\rangle \langle \xi_d, \eta_g\rangle\\
    &= \dim_q(\pi_c)^{-2} \delta_{\pi_c, \pi_g} \delta_{\pi_g, \pi_d} \langle \xi_g, \xi_c\rangle \langle \mu_c, \delta_\X^{1/2}\mu_d\rangle \langle \xi_d, \eta_g\rangle.
\end{align*}
These calculations prove \eqref{toprove}.
\end{proof}

Note that
$$\Delta_{\XGX}^r(C_\mcO(\XGX))\subseteq C_\mcO(\XGX) \otimes C_\mcO(\XGX), \quad \Delta_{\XGX}^r(L^\infty_\mcO(\XGX))\subseteq L^\infty_\mcO(\XGX)\ovot L^\infty_\mcO(\XGX).$$

The unitary antipode $R_{\XGX}: \mcO(\XGX)\to \mcO(\XGX)$ extends uniquely to a normal anti-$*$-isomorphism $R_{\XGX}^r: L^\infty_\mcO(\XGX)\to L^\infty_\mcO(\XGX)$. Indeed, using that $\varphi_{\XGX}\circ R_{\XGX}= \varphi_{\XGX}$ on $\mcO(\XGX)$, we can define the anti-unitary
$\hat{J}_{\XGX}: L^2(\XGX)\to L^2(\XGX)$ via
$$\hat{J}_{\XGX}\Lambda_{\XGX}(a)= \Lambda_{\XGX}(R_{\XGX}(a)^*),  \quad a\in \mcO(\XGX).$$ Then $x\mapsto \hat{J}_{\XGX} x^* \hat{J}_{\XGX}$ implements the desired normal extension of $R_{\XGX}$. Similarly, the  scaling group $\tau_t^{\XGX}: \mcO(\XGX)\to \mcO(\XGX)$ extends to the von Neumann algebra $L^\infty_\mcO(\XGX)$.

In conclusion, the compact quantum hypergroup $\XGX$ also has a reduced $C^*$-algebraic and a $W^*$-algebraic version. However, its reduced version does not fit into Definition \ref{complicated} (cf.\ Remark \ref{completions}), unless the counit $\epsilon_{\XGX}:\mcO(\XGX)\to \C$ is bounded for the reduced norm (see Theorem \ref{counitbounded} for a characterization). However, there is little doubt that the reduced version of $\XGX$ should qualify as a prime example of a  $C^*$-algebraic compact quantum hypergroup.

\section{Examples}\label{examples}

Fix a compact quantum group $\G$. We discuss two examples of compact quantum hypergroups of a general nature.

\subsection{Coideals}\label{coideals}  Let $L^\infty(\X)\subseteq L^\infty(\G)$ be a (right) coideal von Neumann algebra, i.e.\ $L^\infty(\X)$ is a von Neumann subalgebra of $L^\infty(\G)$ such that $\Delta_\G(L^\infty(\X))\subseteq L^\infty(\X)\ovot L^\infty(\G)$.
In that case, we write $\X = \H \backslash \G$, and we think of the object $\H$ as being a `generalized closed quantum subgroup' of $\G$. We then have the natural (ergodic) action $L^\infty(\H\backslash \G)\stackrel{\Delta_\G}\curvearrowleft \G$. Sometimes, we will also denote this action with $\alpha$.

Let us start by recalling some theory from the algebraic theory of coideals \cites{Ch18, DCDT24}. Consider the two-sided coideal $\mcO(\H\backslash \G)_+:= \mcO(\H\backslash \G)\cap \operatorname{Ker}(\epsilon_\G)\subseteq \mcO(\G)$, so that we can define the quotient coalgebra
$$\mcO(\H):= \mcO(\G)/\mcO(\G) \mcO(\H\backslash \G)_+.$$
We write $q: \mcO(\G)\to \mcO(\H)$ for the associated quotient map and $1_\H= q(1)$.
We then have
\begin{equation}\label{coidealreconstruction}
    \mcO(\H \backslash \G)= \{a\in \mcO(\G): q(a_{(1)})\otimes a_{(2)}= 1_\H \otimes a\}.
\end{equation}
The space $\mcO(\H)$ carries the involution $\dag$ defined by $q(a)^\dag = q(S_\G(a)^*)$ for $a\in \mcO(\G)$. The algebraic dual of $\mcO(\H)$ will be denoted by $\mathscr{U}_\H$, which becomes a $*$-algebra for the product and involution given by
$$(\omega \star \omega')(z)= (\omega \otimes \omega')\Delta_{\H}(z), \quad \omega^*(z)= \overline{\omega(z^\dag)}, \quad \omega, \omega'\in \mathscr{U}_{\H}, \quad z \in \mcO(\H).$$

The quotient map $q: \mcO(\G)\to \mcO(\H)$ dualizes to a $*$-algebra embedding
$\mathscr{U}_\H\hookrightarrow \mathscr{U}_\G.$ We will use this to view functionals on $\mcO(\H)$ as functionals on $\mcO(\G)$, without explicit mention.  We also recall the natural action
$$\mathscr{U}_{\G}\curvearrowright \mcH_\pi,  \quad \omega \xi = (\id \odot \omega)(U_\pi)\xi,  \quad \pi \in \Rep(\G), \quad \omega \in \mathscr{U}_\G, \quad \xi \in \mcH_\pi.$$
Through the canonical embedding $\mathscr{U}_\H\hookrightarrow \mathscr{U}_\G$, we then obtain canonical actions $\mathscr{U}_\H\curvearrowright \mcH_\pi$ as well, for every $\pi \in \Rep(\G)$.

There is a unique functional $\varphi_\H: \mcO(\H)\to \C$ (not necessarily positive, cf.\ Proposition \ref{compact quasi-subgroup}) such that 
\begin{equation}\label{invariance}
    \varphi_\H(1_\H)= 1, \quad (\varphi_\H \odot \id)\Delta_\H(c) = \varphi_\H(c) 1= (\id \odot \varphi_\H)\Delta_\H(c), \quad c \in \mcO(\H).
\end{equation}
It allows us to define the projection
\begin{equation}\label{expectation}
    E: \mcO(\G)\to \mcO(\H \backslash \G): a \mapsto (\varphi_\H  \odot \id)\Delta_\G(a).
\end{equation}
It is right $\mcO(\H\backslash \G)$-linear, preserves $\varphi_\G$ and acts on matrix coefficients via
\begin{equation}
    \label{expmat} E(U_\pi(\xi, \eta)) = U_\pi(\varphi_\H\xi, \eta), \quad \pi\in \Rep(\G), \quad \xi, \eta \in \mcH_\pi.
\end{equation}
Moreover, since $\epsilon_\G\circ E = \varphi_\H$, we see that $\varphi_\H= \epsilon_\G$ on $\mcO(\H\backslash \G)$.

In the coideal case, we can view the spaces $\mcG_\pi$ as subspaces of $\mcH_\pi$. Here is the concrete statement:
\begin{Lem} Given $\pi\in \Rep(\G)$ and $\xi \in \mcH_\pi$, the following are equivalent:
\begin{enumerate}
    \item $U_\pi(\xi, \eta)\in \mcO(\H\backslash \G)$ for all $\eta \in \mcH_\pi$.
    \item \vspace{-1.5mm}$\varphi_\H\xi = \xi$.
    \item \vspace{-1.5mm}For all $\omega \in \mathscr{U}_\H\subseteq \mathscr{U}_\G$, we have $\omega\xi= \omega(1_\H)\xi$.
\end{enumerate}
The map
    \begin{equation}\label{spectral spaces}
        \mcG_\pi\ni \mu \mapsto \xi_\mu = (\id \odot \epsilon_\G)(\mu)\in \varphi_\H \mcH_\pi
    \end{equation}
    is a well-defined unitary. We have $X_\pi(\mu, \xi) = U_\pi(\xi_\mu, \xi)$ for $\mu \in \mcG_\pi$ and $\xi \in \mcH_\pi$.
\end{Lem}
\begin{proof} The equivalence $(1)\iff (2)$ follows from \eqref{expectation} and \eqref{expmat} and the equivalence $(2)\iff (3)$ follows from \eqref{invariance}. Let us write $\alpha:= \Delta_\G\vert_{L^\infty(\H\backslash \G)}$. It follows from the equivalence $(1)\iff (2)$ that the unitary \eqref{unitary} restricts to a unitary $\varphi_\H \mcH_\pi \cong \mcG_\pi^\alpha$.
\end{proof}

Through the unitary $\mcG_\pi \cong \varphi_\H \mcH_\pi$, we transport the positive invertible operator $\delta_{\H\backslash \G}$ to the positive invertible operator $\tilde{\delta}_{\H\backslash \G} \in B(\varphi_\H \mcH_\pi)$. We then find the formula
\begin{equation}\label{v}
    \sigma_z^{\H\backslash \G}(U_\pi(\xi, \eta))= U_\pi(\tilde{\delta}_{\H\backslash \G}^{i\bar{z}/2}\xi, \delta_\G^{-iz/2}\eta), \quad \xi \in \varphi_\H \mcH_\pi, \quad \eta \in \mcH_\pi.
\end{equation}
In \cite{DCDT24}*{Lemma 1.4}, it is proven that 
\begin{equation}\label{w}
    \sigma^{\H\backslash \G}_{i}(U_\pi(\xi, \eta))= U_\pi(\varphi_\H \delta_\G^{1/2} \xi, \delta_\G^{1/2}\eta), \quad \xi \in \varphi_\H \mcH_\pi, \quad \eta \in \mcH_\pi.
\end{equation}
Consequently, comparing the expression \eqref{v} (for $z= i$) and the expression \eqref{w}, we arrive at $\tilde{\delta}_{\H\backslash \G}^{1/2} = \varphi_\H \delta_\G^{1/2}\varphi_\H$ and in particular 
$\tilde{\delta}_{\H\backslash \G}^{z}= (\varphi_\H \delta_\G^{1/2} \varphi_\H)^{2z}$ for every $z\in \C$.

We will write $\mcO(\H\backslash \G/\H):=\mcO(\H\backslash \G)\cap R_\G(\mcO(\H\backslash \G))$ and we define similarly $C(\H\backslash \G/\H)$ and $L^\infty(\H\backslash \G/\H)$.  Recall from \eqref{identification2} the canonical $*$-isomorphism
$$L^\infty(\H\backslash\G/\H)\cong L^\infty(\XGX): a \mapsto (\id \otimes \phi)\Delta_\G(a).$$

\begin{Prop}\label{needed} The $*$-isomorphism \eqref{identification2}
restricts to a $*$-isomorphism
$\mcO(\H\backslash \G/\H)\cong \mcO(\XGX)$.
Given $\pi\in \Rep(\G)$ and $\mu, \nu \in \mcG_\pi$, the element $Z_\pi(\mu, \nu)\in \mcO(\XGX)$ corresponds to the element $U_\pi(\xi_\mu, \delta_\G^{1/4} \tilde{\delta}_{\H\backslash \G}^{-1/4} \xi_\nu)$ under this $*$-isomorphism. In particular, $\mcO(\H\backslash \G/\H)$ is linearly generated by the elements $U_\pi(\xi, \delta_\G^{1/4}\eta)$ where $\pi\in \Irr(\G)$ and $\xi, \eta \in \varphi_\H\mcH_\pi$. 
\end{Prop}
\begin{proof} Clearly the restriction
$$\mcO(\H\backslash \G/\H)\to \mcO(\XGX): a \mapsto a_{(1)}\otimes \overline{R_\G(a_{(2)}^*)}$$
is a well-defined injective $*$-algebra homomorphism. We need to argue it is surjective. To this end, we define the linear map
$\bar{\epsilon}_{\G}: \overline{\mcO(\H\backslash \G)}\to \C: \bar{a} \mapsto \epsilon_\G(a^*)$, and we fix $z \in \mcO(\XGX)\subseteq \mcO(\H\backslash\G)\odot \overline{\mcO(\H\backslash \G)}$. Then we consider $a:= (\id \odot \bar{\epsilon}_{\G})(z)\in \mcO(\H\backslash \G)$.  Using that $(\id \odot \bar{\epsilon}_\G)\bar{\alpha}(\bar{a})= R_\G(a^*)$ for $a\in \mcO(\H\backslash \G)$, we then find
$$\alpha(a)= (\id \odot \id \odot \bar{\epsilon}_\G)(\alpha \odot \id)(z)= (\id \odot \id \odot \bar{\epsilon}_\G)(\id \odot \bar{\alpha})(z) \in \mcO(\H\backslash \G)\odot  R_\G(\mcO(\H\backslash \G)).$$
Consequently, 
$a = (\epsilon_\G \odot \id)\alpha(a) \in R_\G(\mcO(\H\backslash \G))\cap \mcO(\H\backslash \G)= \mcO(\H\backslash \G/\H)$. But then
\begin{align*}
    (\id \odot \phi)\Delta_\G(a)&= (\id \odot \phi)(\id \odot \id \odot \bar{\epsilon}_\G)(\Delta_\G \odot \id)(z) = (\id \odot \phi) (\id \odot \id \odot \bar{\epsilon}_\G)(\id \odot \bar{\alpha})(z) = z,
\end{align*}
and the surjectivity is proven. Finally, we compute for $\pi \in \Rep(\G)$ and $\mu, \nu \in \mcG_\pi$ that
\begin{align*}
    (\id \odot \bar{\epsilon}_\G)(Z_\pi(\mu, \nu))&= \sum_{j=1}^{d_\pi} X_\pi(\mu, e_j^\pi) \bar{\epsilon}_\G(Y_\pi(e_j^\pi, \nu)) = \sum_{j=1}^{d_\pi} X_\pi(\mu,e_j^\pi) \langle \delta_\G^{1/4}e_j^\pi, \tilde{\delta}_{\H\backslash \G}^{-1/4}\xi_\nu\rangle = U_\pi(\xi_\mu, \delta_\G^{1/4}\tilde{\delta}_{\H\backslash \G}^{-1/4}\xi_\nu),
\end{align*}
which finishes the proof.
\end{proof}
It follows from the preceding result that $\mcO(\H\backslash \G/\H)$ is nothing else than the $*$-algebra of \emph{$*$-spherical functions}  considered in \cite{DCDT24}*{Definition 1.19}. 

The coassociative normal ucp map $\Delta_{\XGX}^r: L^\infty(\XGX)\to L^\infty(\XGX)\ovot L^\infty(\XGX)$ from Theorem \ref{mainsec3} transports to a coassociative normal ucp map
$$\Delta_{\H\backslash \G/\H}^r: L^\infty(\H \backslash \G/\H)\to L^\infty(\H \backslash \G/\H)\ovot L^\infty(\H \backslash \G/\H)$$
under the isomorphism \eqref{identification2}. It acts on matrix coefficients by
\begin{equation}\label{compl}
    \Delta_{\H\backslash \G/\H}(U_\pi(\xi, \delta_\G^{1/4}\eta)) = \sum_{j=1}^{m_\pi} U_\pi(\xi, \delta_\G^{1/4} f_j^\pi)\otimes U_\pi(\tilde{\delta}_{\H\backslash \G}^{-1/4} f_j^\pi, \delta_\G^{1/4}\eta), \quad \xi, \eta \in \varphi_\H \mcH_\pi,
\end{equation} 
where $\{f_j^\pi\}_{j=1}^{m_\pi}$ is an orthonormal basis for $\varphi_\H \mcH_\pi\cong \mcG_\pi$. 
On the other hand, counit and antipode are given in this picture by
$$\epsilon_{\mathbb{H}\backslash \G/\H}(U_\pi(\xi, \delta_\G^{1/4}\eta))  = \langle \xi, \tilde{\delta}_{\H\backslash \G}^{1/4}\eta\rangle, \quad S_{\H\backslash \G/\H}(U_\pi(\tilde{\delta}_{\H\backslash \G}^{1/4}\xi, \delta_\G^{1/4}\eta)) = U_\pi(\tilde{\delta}_{\H\backslash \G}^{1/4}\eta, \delta_\G^{1/4}\xi)^* \quad \xi, \eta \in \varphi_\H \mcH_\pi.$$

This explains how to turn the `double coset space' $\H\backslash\G/\H$ into a  compact quantum hypergroup beyond the case where $\H$ is a compact (quasi-)subgroup of $\G$. Let us now explain how the easier case of compact (quasi-)subgroups follows from our general construction. For this, we need the following result:

\begin{Prop}\label{compact quasi-subgroup} Let $\G$ be a compact quantum group and let $L^\infty(\H\backslash \G)\subseteq L^\infty(\G)$ be a coideal von Neumann algebra. The following are equivalent:
    \begin{enumerate}
    \item\vspace{-1.5mm} $\sigma_t^\G(L^\infty(\H\backslash \G))= L^\infty(\H\backslash \G)$ for all $t\in \R$.
    \item\vspace{-1.5mm} There exists a (unique) normal conditional expectation $F: L^\infty(\G)\to L^\infty(\H\backslash \G)$ that preserves $\varphi_\G$.
    \item\vspace{-1.5mm} $\varphi_\H \delta_\G = \delta_\G \varphi_\H$.
 \item\vspace{-1.5mm} $\delta_\G\vert_{\varphi_\H \mcH_\pi} = \tilde{\delta}_{\H\backslash \G}$ for all $\pi \in \Irr(\G)$.
    \item\vspace{-1.5mm} $\varphi_\H$ is positive, i.e.\ $\varphi_\H(a^*a)\ge 0$ for all $a\in \mcO(\G)$.
\end{enumerate}
\end{Prop}
\begin{proof} The equivalence $(1)\iff (2)$ follows from \cite{St20}*{Theorem 10.1}. 
$(5)\implies (2)$ If $\varphi_\H: \mcO(\G)\to \C$ is positive, then it extends to an idempotent state $\tilde{\varphi}_\H\in C^u(\G)^*$, and we obtain the normal conditional expectation
$$\tilde{E}: L^\infty(\G)\to L^\infty(\H \backslash \G): x \mapsto (\tilde{\varphi}_\H\otimes \id)(\Ww_\G^*(1\otimes x)\Ww_\G)$$
extending \eqref{expectation}, where $\Ww_\G\in M(C^u(\G)\otimes C_0(\hat{\G}))$ is the half-lifted version of the multiplicative unitary $W_\G$ \cite{Kus01}*{Proposition 5.1}. $(2)\implies (5)$ Assume that $(2)$ holds and recall the canonical projection $E: \mcO(\G)\to \mcO(\H\backslash \G)$ defined in \eqref{expectation}. If $x\in \mcO(\G)$ and $y \in \mcO(\H\backslash \G)$, we have
$$\varphi_\G((E(x)-F(x))y) = \varphi_\G(E(xy)-F(xy)) = \varphi_\G(xy)- \varphi_\G(xy) = 0,$$
where we used that $L^\infty(\H\backslash \G)$ is in the multiplicative domain of $F$ and that $E$ is right $\mcO(\H\backslash \G)$-linear.
Consequently, $F\vert_{\mcO(\G)}= E$ by faithfulness of $\varphi_\G$. Given $x\in \mcO(\G)$, we then have $\varphi_\H(x) = \epsilon_\G(E(x))= \epsilon_\G (F(x))$,
so $\varphi_\H$ is positive. $(1)\implies (5)$ If $(1)$ holds, then $\sigma_t^{\G}\vert_{L^\infty(\H\backslash \G)}= \sigma_t^{\H\backslash \G}$ for all $t\in \R$ by uniqueness of the modular group. Therefore, if $\xi \in \varphi_\H \mcH_\pi$ and $\eta \in \mcH_\pi$, we find
$$U_\pi(\tilde{\delta}_{\H\backslash \G}^{1/2}\xi, \delta_\G^{1/2}\eta)=\sigma_{i}^{\H\backslash \G}(U_\pi(\xi, \eta)) = \sigma_{i}^{\G}(U_\pi(\xi, \eta))= U_\pi(\delta_\G^{1/2}\xi, \delta_\G^{1/2}\eta),$$
whence $\delta_\G^{1/2}\xi = \tilde{\delta}_{\H\backslash \G}^{1/2}\xi$, from which we conclude that $(4)$ holds. The implication $(4)\implies (1)$ is trivial and the equivalence $(3)\iff (4)$ is clear by keeping in mind that $\tilde{\delta}_{\H\backslash \G}^{1/2}= \varphi_\H \delta_\G^{1/2}\varphi_\H$.
\end{proof}

If the equivalent conditions from Proposition \ref{compact quasi-subgroup} hold, we call $\H$ a \emph{compact quasi-subgroup} of $\G$ \cite{KS20}. In that case, $\mcO(\H\backslash \G/\H)$ is generated by the matrix coefficients $\{U_\pi(\xi, \eta): \pi\in \Irr(\G), \xi, \eta \in \varphi_\H \mcH_\pi\}$ and the formula \eqref{compl} simplifies to
$$\Delta_{\H\backslash \G/\H}(U_\pi(\xi, \eta))= \sum_{j=1}^{m_\pi} U_\pi(\xi, f_j^\pi)\otimes U_\pi(f_j^\pi, \eta), \quad \xi, \eta \in \varphi_\H \mcH_\pi.$$
In other words,
$$\Delta_{\H\backslash \G/\H}^r(z)= (E\otimes E) \Delta_\G(z), \quad z \in L^\infty_\mcO(\H \backslash \G/\H),$$
where $E: L^\infty(\G)\to L^\infty(\H\backslash \G)$ is the normal conditional expectation extending \eqref{expectation}.
On the other hand, counit, antipode and invariant state of the compact quantum hypergroup are simply given by the restrictions of the counit, antipode and invariant state of the compact quantum group $\G$.

Compact quantum hypergroups arising from compact quasi-subgroups (albeit in the purely algebraic or $C^*$-algebraic framework) were first considered in \cite{Ka01} and later generalized in \cite{Zh20}. Note also that an example of a von Neumann algebraic compact quantum hypergroup arising from a normal conditional expectation was constructed in \cite{DC09b}. 

\subsection{Fusion algebra}\label{fusion algebra} Consider now the compact quantum group $\G^{\op}$ given by $L^\infty(\G^{\op})= L^\infty(\G)$ and $\Delta_{\G^{\op}}= \Delta_\G^{\op}$. As in \cite{DCDR25}*{Section 4.1}, we then consider the ergodic action $L^\infty(\G)\stackrel{\alpha}\curvearrowleft \G^{\op}\times \G$ given by
$$\alpha: L^\infty(\G)\to L^\infty(\G)^{\ovot 3}: x\mapsto \Delta_\G^{(2)}(x)_{213}.$$ 

We give a complete description of the spectral data associated with this ergodic action. Note that $\Delta_\G(L^\infty(\G))\subseteq L^\infty(\G^{\op}\times \G)$ defines a coideal von Neumann subalgebra of $L^\infty(\G^{\op}\times \G)$ and the $*$-isomorphism $\Delta_\G: (L^\infty(\G), \alpha)\to (\Delta_\G(L^\infty(\G)), \Delta_{\G^{\op}\times \G})$ is $\G^{\op}\times \G$-equivariant. Thus, we could in principle use the results in Subsection \ref{coideals} to do this. Rather, we prefer to give a self-contained ad-hoc approach. 

Given $\pi\in \Irr(\G)$, write $\pi^{*}\in \Irr(\G^{\op})$ for the irreducible $\G^{\op}$-representation determined  by $U_{\pi^*}:= U_\pi^*$. We then have $\delta_{\G^{\op}}^{\pi^*}= (\delta_\G^\pi)^{-1}$ as operators on $\mcH_\pi$.
The contragredient $\overline{\pi^*}$ of $\pi^*\in\Irr(\G^{\op})$ will be denoted by $\tilde{\pi}$, so that
$$U_{\tilde{\pi}}(\bar{\xi}, \bar{\eta})= U_\pi(\delta_\G^{-1/4}\eta, \delta_\G^{1/4}\xi), \quad \xi,\eta \in \mcH_\pi.$$
In particular, $\mcO(\G^{\op})_{\tilde{\pi}}= \mcO(\G)_\pi$ for every $\pi \in \Irr(\G)$.
The irreducible representations of $\G^{\op}\times \G$ are given by $
\tilde{\pi}\times \pi' := (\overline{\mcH_\pi}\otimes \mcH_{\pi'}, U_{\tilde{\pi},13}U_{\pi',24})$ where $\pi, \pi'\in \Irr(\G)$. Combining these facts, it is straightforward to see that the spectral subspaces of $L^\infty(\G)\stackrel{\alpha}\curvearrowleft \G^{\op}\times \G$ are given for $\pi, \pi'\in \Irr(\G)$ by
\begin{equation}\label{spectralsubspaces}
    \mcO(\G)_{\tilde{\pi}\times \pi'}= \begin{cases}
        \mcO(\G)_\pi & \pi = \pi'\\
        0 & \pi\ne \pi'
    \end{cases}
\end{equation}

\begin{Prop}\label{spectraldata}
    Given $\pi, \pi'\in \Irr(\G)$, write $\mu_\pi:= \dim_q(\pi)^{-1/2}\sum_{j,k=1}^{d_\pi} \overline{e_j^\pi}\otimes e_k^\pi\otimes U_\pi(\delta_\G^{-1/4}e_j^\pi, e_k^\pi)^*$. Then
    $$\mcG_{\tilde{\pi}\times \pi'}= \begin{cases}
        \C \mu_\pi & \pi = \pi'\\
        0 & \pi\ne \pi'
    \end{cases}$$
    and $\|\mu_\pi\|=1$. Moreover, given $\xi, \eta \in \mcH_{\pi}$, we have $X_{\tilde{\pi}\times \pi}(\mu_\pi, \overline{\xi}\otimes \eta)= \dim_q(\pi)^{-1/2} U_\pi(\delta_\G^{-1/4}\xi, \eta)$ and $$Z_{\tilde{\pi}\times \pi}(\mu_\pi, \mu_\pi)= \dim_q(\pi)^{-1} \sum_{j,k=1}^{d_\pi} U_\pi(e_j^\pi, e_k^\pi)\otimes \overline{U_{\bar{\pi}}(\overline{e_j^\pi}, \overline{e_k^\pi})^*}.$$
    \end{Prop}
\begin{proof}  Given $\pi, \pi'\in \Irr(\G)$, it follows from \eqref{spectralsubspaces}  that
$\mcG_{\tilde{\pi}\times \pi'}\subseteq \overline{\mcH_{\pi}}\odot \mcH_\pi \odot \mcO(\G)_{\tilde{\pi}\times \pi'}^*$. In particular, $\mcG_{\tilde{\pi}\times \pi'}= 0$ if $\pi \ne \pi'$ and $\mcG_{\tilde{\pi}\times \pi}\subseteq \overline{\mcH_\pi}\odot \mcH_\pi \odot \mcO(\G)_\pi^*$. 

We may assume without loss of generality that $\{e_j^\pi\}_{j=1}^{d_\pi}$ is an orthonormal basis of eigenvectors for $\delta_\G$, say $\delta_\G e_j^\pi= \delta_j e_j^\pi$ for $1\le j \le d_\pi$. We also employ the notation $u_{st}:= U_\pi(e_s^\pi, e_t^\pi)$ for $1\le s,t \le d_\pi$. Consider $\mu\in  \mcG_{\tilde{\pi}\times \pi}$. Then we may write  $\mu = \sum_{j,k,s,t=1}^{d_\pi} \lambda_{jkst}\overline{e_j^\pi}\otimes e_k^\pi \otimes u_{st}^*$ for certain scalars $\lambda_{jkst}\in \C$. On the one hand
\begin{align*}
    (\id \odot \id \odot \alpha)(\mu)&= \sum_{j,k,s,t,p,q=1}^{d_\pi} \lambda_{jkst}\overline{e_j^\pi}\otimes e_k^\pi\otimes u_{pq}^*\otimes u_{sp}^*\otimes u_{qt}^*,
\end{align*}
and on the other hand
\begin{align*}U_{\tilde{\pi}\times \pi, 1245}^*\mu_{123}&= \sum_{j,k,s,t,p,q=1}^{d_\pi}  \lambda_{stpq} \delta_j^{-1/4}\delta_s^{1/4} \overline{e_j^\pi}\otimes e_k^\pi\otimes u_{pq}^* \otimes u_{js}^*\otimes u_{tk}^*. \end{align*}
Consequently, it follows that for every $1\le j,k,p,q\le d_\pi$, 
$$\sum_{s,t=1}^{d_\pi} \lambda_{jkst}  u_{sp}^*\otimes u_{qt}^* = \sum_{s,t=1}^{d_\pi} \lambda_{stpq} \delta_j^{-1/4}\delta_s^{1/4} u_{js}^*\otimes u_{tk}^*.$$
From this, it  follows that $\lambda_{jkst}=0$ if either $t\ne k$ or $j\ne s$, and that $\lambda_{jkjk}\delta_j^{1/4}= \lambda_{pqpq}\delta_p^{1/4}$. Therefore,
\begin{align*}
    \mu = \sum_{j,k=1}^{d_\pi} \lambda_{jkjk} \overline{e_j^\pi}\otimes e_k^\pi \otimes U_\pi(e_j^\pi, e_k^\pi)^*= \lambda_{1111}\delta_1^{1/4}\sum_{j,k=1}^{d_\pi} \delta_j^{-1/4} \overline{e_j^\pi}\otimes e_k^\pi \otimes U_\pi(e_j^\pi, e_k^\pi)^* \in \C \mu_\pi.
\end{align*}
Consequently, $\mcG_{\tilde{\pi}\times  \pi}= \C \mu_\pi$. If $\xi, \eta \in \mcH_\pi$, we have
\begin{align*}
    X_{\tilde{\pi}\times \pi}(\mu_\pi, \overline{\xi}\otimes \eta)&= \dim_q(\pi)^{-1/2}\sum_{j,k=1}^{d_\pi}\langle \overline{e_j^\pi}, \overline{\xi}\rangle \langle e_k^\pi, \eta\rangle U_{\pi}(\delta_\G^{-1/4}e_j^\pi, e_k^\pi)= \dim_q(\pi)^{-1/2} U_\pi(\delta_\G^{-1/4}\xi, \eta).
\end{align*}
Since $\varphi_\X= \varphi_\G$, we have $\sigma_\X= \sigma_\G$, and consequently
\begin{align*}
    \sigma_\X(X_{\tilde{\pi}\times \pi}(\mu_\pi, \overline{\xi}\otimes \eta))&= \dim_q(\pi)^{-1/2}\sigma_\G(U_\pi(\delta_\G^{-1/4}\xi, \eta))\\
    &= \dim_q(\pi)^{-1/2}U_\pi(\delta_\G^{-3/4}\xi, \delta_\G^{-1/2}\eta)\\
    &= X_{\tilde{\pi}\times \pi}(\mu_\pi, \overline{\delta_\G^{-1/2}\xi} \otimes \delta_\G^{-1/2}\eta)= X_{\tilde{\pi}\times \pi}(\mu_\pi, \delta_{\G^{\op}\times \G}^{-1/2}(\overline{\xi}\otimes \eta)).
\end{align*}
It therefore follows that $\delta_\X = 1$. We can then calculate
\begin{align*}
    Z_{\tilde{\pi}\times \pi} (\mu_\pi, \mu_\pi)&= \sum_{j,k=1}^{d_\pi} X_{\tilde{\pi}\times\pi}(\mu_\pi, \overline{e_j^\pi}\otimes e_k^\pi)\otimes \overline{X_{\tilde{\pi}\times \pi}(\mu_\pi, \delta_{\G^{\op}\times \G}^{1/4}(\overline{e_j^\pi}\otimes e_k^\pi))}\\
    &= \dim_q(\pi)^{-1} \sum_{j,k=1}^{d_\pi} U_\pi(\delta_\G^{-1/4}e_j^\pi, e_k^\pi)\otimes \overline{U_\pi(e_j^\pi, \delta_\G^{1/4}e_k^\pi)}\\
    &= \dim_q(\pi)^{-1}\sum_{j,k=1}^{d_\pi} U_\pi(e_j^\pi, e_k^\pi) \otimes \overline{U_\pi(\delta_\G^{-1/4}e_j^\pi, \delta_\G^{1/4}e_k^\pi)}\\
    &= \dim_q(\pi)^{-1} \sum_{j,k=1}^{d_\pi} U_\pi(e_j^\pi, e_k^\pi)\otimes \overline{U_{\bar{\pi}}(\overline{e_j^\pi}, \overline{e_k^\pi})^*}.
\end{align*}
These calculations finish the proof.
\end{proof} 

Note now that
\begin{align*}
    L^\infty(\G\times_{\G^{\op}\times \G}\bar{\G})&= L^\infty(\G)\stackrel{\G^{\op}\times \G}\square L^\infty(\bar{\G})\\
    &\cong \Delta_\G(L^\infty(\G))\stackrel{\G^{\op}\times \G}\square \overline{\Delta_\G(L^\infty(\G))}\\
    &\cong \Delta_\G(L^\infty(\G))\cap R_{\G^{\op}\times \G}(\Delta_\G(L^\infty(\G)))\\
&= \Delta_\G(L^\infty(\G))\cap \Delta_\G^{\op}(L^\infty(\G)) \\
&\cong \{x\in L^\infty(\G)\mid \Delta_\G(x)= \Delta_\G^{\op}(x)\},
\end{align*}
where the second isomorphism follows from Example \ref{identification2} and the last isomorphism follows from the simple fact that if $\Delta_\G(x)= \Delta_\G^{\op}(y)$ for $x,y \in L^\infty(\G)$, then necessarily $x=y$.

Let us write $\mcO(\operatorname{Fus}[\G]):= \{a\in \mcO(\G): \Delta_\G(a)= \Delta_\G^{\op}(a)\}$, which is a unital $*$-subalgebra of $\mcO(\G)$ known as the \emph{fusion algebra} or the \emph{character algebra} of $\G$. It has a Hamel basis given by the characters $\chi(\pi):= \sum_{j=1}^{d_\pi} U_\pi(e_j^\pi, e_j^\pi)$, where $\pi\in \Irr(\G)$. We similarly define $C(\operatorname{Fus}[\G])$ and $L^\infty(\operatorname{Fus}[\G])$ for its $C^*/W^*$-version. The canonical $*$-isomorphism $L^\infty(\operatorname{Fus}[\G])\cong L^\infty(\G\times_{\G^{\op}\times \G}\bar{\G})$ considered above maps the character $\chi(\pi)$ to the element $\sum_{j,k=1}^{d_\pi} U_\pi(e_j^\pi, e_k^\pi)\otimes \overline{U_{\bar{\pi}}(\overline{e_j^\pi}, \overline{e_k^\pi})^*} = \dim_q(\pi) Z_{\tilde{\pi}\times \pi}(\mu_\pi, \mu_\pi)$ (cf.\ \cite{DCDR25}*{Proposition 4.2}).
In particular, the matrix coefficient $Z_{\tilde{\pi}\times \pi}(\mu_\pi, \mu_\pi)$ corresponds exactly to the normalized character $\chi_q(\pi):=\dim_q(\pi)^{-1}\chi(\pi)$ through this isomorphism. The compact quantum hypergroup structure of $\G\times_{\G^{\op}\times \G}\bar{\G}$ then transports to a compact quantum hypergroup structure on $\operatorname{Fus}(\G)$, given on generators by
$$\Delta_{\operatorname{Fus}(\G)}(\chi_q(\pi))= \chi_q(\pi)\otimes  \chi_q(\pi), \quad \epsilon_{\operatorname{Fus}(\G)}(\chi_q(\pi))= 1, \quad S_{\operatorname{Fus}(\G)}(\chi_q(\pi))= \chi_q(\pi)^*, \quad \pi \in \Irr(\G).$$

\begin{Rem}\label{completionss}
    Note that the equality $L_\mcO^\infty(\G\times_{\G^{\op}\times \G}\bar{\G})= L^\infty(\G\times_{\G^{\op}\times \G}\bar{\G})$ is equivalent with the equality $L^\infty(\operatorname{Fus}[\G])= L^\infty_\mcO(\operatorname{Fus}[\G])$, which was proven in \cite{AC17}*{Theorem 3.7}, by making use of the traciality of the canonical faithful normal state. In fact, returning to the case where $L^\infty(\X)$ is an arbitrary ergodic $\G$-$W^*$-algebra, one can also show that $L^\infty_\mcO(\XGX)= L^\infty(\XGX)$ when $\varphi_{\XGX}$ is tracial on $L^\infty(\XGX)$ (or more generally when $L^\infty_\mcO(\XGX)$ is invariant under the modular group of $L^\infty(\XGX)$ associated to $\varphi_{\XGX}$). Determining if $C_\mcO(\XGX)= C(\XGX)$ is an even more difficult question, related to a long-standing open question of Woronowicz \cite{Wor87a}*{Section 5}.
\end{Rem}

\section{Coamenability}\label{coamenable}

Given a (reduced) compact quantum group $\G$, we have the following equivalences \cites{BMT01, BT03, Tom06, Cr17, DH24, DCDR25, DR26}:
\begin{enumerate}
    \item\vspace{-1.5mm} $\hat{\G}$ is amenable, or equivalently $L^\infty(\G)$ is $\G$-injective.
    \item\vspace{-1.5mm} $\G$ is coamenable, or equivalently $L^\infty(\G)$ is strongly $\G$-injective.
        \item \vspace{-1.5mm}The counit on $\mcO(\G)$ is bounded for the reduced norm.
    \item \vspace{-1.5mm}$C^u(\G)\cong C(\G)$.
    \item \vspace{-1.5mm} The Banach algebra $C(\G)^*$ is unital.
\end{enumerate}

We generalize this result to the setting of ergodic actions, in the sense that the above characterization is recovered if $\X = \G$. 

Recall from Section \ref{universalversion} that every $\mcH\in \Corr^\G(L^\infty(\X), L^\infty(\X))$ gives rise to a unital $*$-representation $\theta_\square^\mcH: \mcO(\XGX)\to B(U_\mcH(\varphi_\G)\mcH)$.

\begin{Prop} Considering the coarse $\G$-$L^\infty(\X)$-$L^\infty(\X)$-correspondence $C_{L^\infty(\X)}^\G$ (see \eqref{coarse}), we have \begin{equation}\label{normcoarse}
    \|\theta_{\square}^{C_{L^\infty(\X)}^\G}(z)\| = \|z\|_r, \quad z \in \mcO(\XGX).
\end{equation}
\end{Prop}
\begin{proof}
     Making use of formula \eqref{usethis}, we find for $\pi \in \Rep(\G)$ and $\mu, \nu \in \mcG_\pi$ that
\begin{align*}
    \theta_\square^{C_{L^\infty(\X)}^\G}(Z_\pi(\mu, \nu)) &= \sum_{j=1}^{d_\pi} ((\pi_\X \otimes \id)\alpha(X_\pi(\mu, e_j^\pi)) \otimes 1)(1\otimes (\rho_\G \otimes \rho_\X)\alpha^{\op}(X_\pi(\delta_\X^{-1/4}\nu, \delta_\G^{-1/4}e_j^\pi)^*))\\
    &= \sum_{j,k,l=1}^{d_\pi} \pi_\X(X_\pi(\mu, e_k^\pi)) \otimes U_\pi(e_k^\pi, e_j^\pi) \rho_\G(U_\pi(e_l^\pi,  \delta_\G^{-1/4}e_j^\pi)^*)\otimes  \rho_\X(X_\pi(\delta_\X^{-1/4}\nu, e_l^ \pi)^*).
\end{align*}
We then note that
    $$V_{\G,24}(\varphi_\G)(L^2(\X)\otimes L^2(\G)\otimes L^2(\X))= L^2(\X)\otimes \C \xi_\G \otimes L^2(\X)\cong L^2(\X)\otimes L^2(\X),$$
so that for $x,y \in \mcO(\X)$,
\begin{align*}
    &\theta_\square^{C_{L^\infty(\X)}^\G}(Z_\pi(\mu, \nu))(\Lambda_\X(x)\otimes \xi_\G \otimes \Lambda_\X(y))\\
    &= \sum_{j,k,l=1}^{d_\pi} \pi_\X(X_\pi(\mu, e_k^\pi))\Lambda_\X(x) \otimes U_\pi(e_k^\pi, e_j^\pi) \rho_\G(U_\pi(e_l^\pi, \delta_\G^{-1/4} e_j^\pi)^*) \xi_\G \otimes \rho_\X(X_\pi(\delta_\X^{-1/4}\nu, e_l^\pi)^*)\Lambda_\X(y)\\
    &= \sum_{j,k,l=1}^{d_\pi} \pi_\X(X_\pi(\mu, e_k^\pi))\Lambda_\X(x) \otimes \Lambda_\G(U_\pi(e_k^\pi, e_j^\pi) U_\pi(\delta_\G^{1/4}e_l^\pi,  e_j^\pi)^*)\otimes \rho_\X(X_\pi(\delta_\X^{-1/4}\nu, e_l^\pi)^*)\Lambda_\X(y)\\
    &= \sum_{k,l=1}^{d_\pi} \pi_\X(X_\pi(\mu, e_k^\pi))\Lambda_\X(x) \otimes  \langle e_k^\pi, \delta_\G^{1/4}e_l^\pi\rangle  \xi_\G \otimes \rho_\X(X_\pi(\delta_\X^{-1/4}\nu, e_l^\pi)^*) \Lambda_\X(y)\\
    &= \sum_{k=1}^{d_\pi} \pi_\X(X_\pi(\mu, e_k^\pi))\Lambda_\X(x) \otimes \xi_\G \otimes \rho_\X(X_\pi(\delta_\X^{-1/4} \nu, \delta_\G^{1/4}e_k^\pi)^*)\Lambda_\X(y).
\end{align*}
Thus, the $*$-representation $\theta_\square^{C_{L^\infty(\X)}^\G}: \mcO(\XGX)\to B(L^2(\X)\otimes L^2(\X))$ is given by
$$ \theta_\square^{C_{L^\infty(\X)}^\G}(Z_\pi(\mu, \nu))= \sum_{k=1}^{d_\pi} \pi_\X(X_\pi(\mu, e_k^\pi)) \otimes \rho_\X(X_\pi(\delta_\X^{-1/4}\nu, \delta_\G^{1/4} e_k^\pi))^*, \quad \pi \in \Rep(\G), \quad \mu, \nu \in \mcG_\pi.$$
The map
$\psi: L^\infty(\X) \ovot L^\infty(\bar{\X})\to B(L^2(\X)\otimes L^2(\X)): x \otimes \bar{y} \mapsto \pi_\X(x)\otimes \rho_\X(y)^*$ 
is isometric and it satisfies 
$\theta_\square^{C_{L^\infty(\X)}^\G}(z) = \psi(z)$ for $z\in \mcO(\XGX)$. 
\end{proof} 

Recall from \eqref{cornernorm} the definition of the norm $\|\cdot\|_u$ on $\mcO(\XGX)$.
 It is a consequence of \eqref{normcoarse} that $\|z\|_r\le \|z\|_u$ for $z\in \mcO(\XGX)$. Thus, there is a unique surjective $*$-homomorphism $$\lambda: C^u(\XGX)\to C_\mcO(\XGX)$$
that is the identity on $\mcO(\XGX)$.

 It is also clear that $C_\mcO(\XGX)^*$
 becomes a Banach algebra for the convolution product coming from the coassociative map $\Delta_{\XGX}^r: C_\mcO(\XGX)\to C_\mcO(\XGX)\otimes C_\mcO(\XGX)$.
We now arrive at the main result of this section:

\begin{Theorem}\label{counitbounded} Let $\G$ be a compact quantum group and let $(L^\infty(\X), \alpha)$ be an ergodic $\G$-$W^*$-algebra. The following are equivalent:
\begin{enumerate}
    \item $(L^\infty(\X), \alpha)$ is $\G$-injective.
    \item\vspace{-1.5mm}$(L^\infty(\X), \alpha)$ is strongly $\G$-injective.
    \item\vspace{-1.5mm} $|\epsilon_{\XGX}(z)| \le \|z\|_r$ for all $z\in \mcO(\XGX)$.
    \item\vspace{-1.5mm} $\lambda: C^u(\XGX)\to C_\mcO(\XGX)$ is a $*$-isomorphism. 
     \item\vspace{-1.5mm} The Banach algebra $C_\mcO(\XGX)^*$ is unital.
\end{enumerate}

\end{Theorem}
\begin{proof} The equivalence $(2) \iff (3)$ is an immediate consequence of $\epsilon_{\XGX}= \theta_\square^{L^2(\X)}$ (see the proof of Theorem \ref{mainCV}), \eqref{normcoarse} and \cite{DCDR25}*{Proposition 2.11}. The implication $(4)\implies (3)$ is trivial. To prove that $(2)\implies (4)$, we note that strong $\G$-injectivity of $L^\infty(\X)$ implies that $$\|\theta_\square^\mcH(z)\| \le \|\theta_{\square}^{C_{L^\infty(\X)}^\G}(z)\|= \|z\|_r, \quad \mcH \in \Corr^\G(L^\infty(\X), L^\infty(\X)), \quad z \in \mcO(\XGX),$$
by using Proposition \ref{charstronginj}. Consequently, $\|z\|_u\le \|z\|_r$, and $\lambda$ is isometric. We now prove $(1)\implies (3)$. The argument proceeds in the same spirit as \cite{DCDR25}*{Theorem 3.6}, by reducing to the case where $\G$ is second countable. Fix $z= \sum_{j=1}^n x_j\otimes \overline{y_j}\in \mcO(\XGX)$. 
Choose a Hopf $*$-subalgebra $\mcO(\H)\subseteq \mcO(\G)$ of countable dimension with the property that
$$\alpha(x_j) \in \mcO(\X)\odot \mcO(\H), \quad \bar{\alpha}(\overline{y_j})\in \mcO(\H)\odot \mcO(\bar{\X}), \quad j=1, \dots, n.$$
As the notation suggests, it is not hard to see that this Hopf $*$-subalgebra is associated to a compact quantum group $\H$. Moreover, the modular data of $\H$ is given by restriction of the modular data of $\G$. We may assume that $\Irr(\H)\subseteq \Irr(\G)$. We write $L^\infty(\H)$ for the von Neumann subalgebra of $L^\infty(\G)$ generated by $\mcO(\H)$ and $L^\infty(\X_\H)$ for the von Neumann subalgebra of $L^\infty(\X)$ generated by $\sum_{\pi \in \Irr(\H)} \mcO(\X)_\pi$. Then $\alpha(L^\infty(\X_\H))\subseteq L^\infty(\X_\H) \ovot L^\infty(\H)\subseteq L^\infty(\X_\H)\ovot L^\infty(\G)$. In particular, we have ergodic actions $L^\infty(\X_\H)\stackrel{\alpha_\H}\curvearrowleft \H$ and $L^\infty(\X_\H)\stackrel{\alpha}\curvearrowleft \G$.

Considering the natural $\G$-equivariant normal conditional expectation $L^\infty(\X)\to L^\infty(\X_\H)$ (which kills the spectral subspaces associated to $\pi \in \Irr(\G)\backslash \Irr(\H)$), it follows that $L^\infty(\X_\H)$ is $\G$-injective. Making use of \cite{DR26}*{Proposition 3.10}, we see that $L^\infty(\X_\H)$ is $\H$-injective as well, which means that the inclusion $\pi_{\X_\H}(L^\infty(\X_\H))\ovot \C1 \subseteq B(L^2(\X_\H))\ovot L^\infty(\X_\H)$ is $\H$-amenable. Since $\mcO(\H)$ has countable dimension, it follows from \eqref{unitaryiso} that $L^2(\X_\H)$ is separable, so that the von Neumann algebra $B(L^2(\X_\H))\ovot L^\infty(\X_\H)$ is $\sigma$-finite. Consequently, it follows from \cite{DCDR25}*{Theorem 3.10} that the inclusion $\pi_{\X_\H}(L^\infty(\X_\H))\ovot \C1 \subseteq B(L^2(\X_\H))\ovot L^\infty(\X_\H)$ is also strongly $\H$-amenable, i.e.\ $(L^\infty(\X_\H), \alpha_\H)$ is strongly $\H$-injective. But since we already proved that $(2)\implies (3)$ holds and since $z\in \mcO(\X_\H\times_{\H}\overline{\X_\H})$, we see that
$$|\epsilon_{\XGX}(z)| = |\epsilon_{\X_\H\times_{\H}\overline{\X_\H}}(z)| \le \|z\|_{L^\infty(\X_\H)\ovot L^\infty(\bar{\X}_\H)} = \|z\|_r,$$
and the implication is proven. The equivalence $(3)\iff (5)$ is trivial.
\end{proof}

\begin{Def}
    We call $\XGX$ coamenable if the equivalent conditions in Theorem \ref{counitbounded} are satisfied.
\end{Def}

The following result is surely well-known, but we are unaware of an explicit reference for it in the literature.

\begin{Cor}\label{coro}
The quantum dimension function
$\mcO(\operatorname{Fus}(\G))\to \C: \chi(\pi)\mapsto \dim_q(\pi)$
is bounded (for the norm coming from the inclusion $\mcO(\operatorname{Fus}(\G))\subseteq L^\infty(\G)$) if and only if $\G$ is Kac and coamenable.
\end{Cor}

\begin{proof}
    By Theorem \ref{counitbounded} and the discussion in Subsection \ref{fusion algebra}, the quantum dimension function is bounded on $\mcO(\operatorname{Fus}(\G))$ if and only if $L^\infty(\G)$ is $\G^{\op}\times \G$-injective, which is equivalent with $\G$ being Kac and coamenable by \cite{DR25a}*{Proposition 4.7}.
\end{proof}
\begin{Rem} Coamenability of a (reduced) compact quantum group $\G$ is also characterized through the existence of a character on the $C^*$-algebra $C(\G)$. However, the existence of a character on the $C^*$-algebra $C_\mcO(\XGX)$ is not sufficient to guarantee the coamenability of $\XGX$. To see this, let $\H$ be a non-Kac coamenable compact quantum group (e.g. $\H = SU_q(2)$), $\X= \H$, $\G= \H^{\op}\times \H$ and consider the natural action $L^\infty(\X)\curvearrowleft \G$ studied in Subsection \ref{fusion algebra}. Since $\H$ is coamenable, the counit $\epsilon_\H$ is bounded on $\mcO(\operatorname{Fus}(\H))\cong \mcO(\XGX)$, and hence induces a character on $C_\mcO(\XGX)$. However, from Corollary \ref{coro}, the counit $\epsilon_{\XGX}$ (corresponding to the quantum dimension function) is not bounded. 
\end{Rem}

The following result was established in the papers \cites{AK24, DR25a}. We can now give a new proof of the hard part:
\begin{Cor} Let $\G$ be a compact quantum group. The following are equivalent for a coideal von Neumann algebra $L^\infty(\H\backslash \G)$:
\begin{enumerate}
    \item $(L^\infty(\H\backslash \G), \Delta_\G)$ is $\G$-injective.
    \item \vspace{-1.5mm} $\epsilon_\G: \mcO(\H\backslash \G/\H)\to \C$ is bounded (for the reduced norm) and $\H$ is a compact quasi-subgroup of $\G$.
     \item\vspace{-1.5mm} $\epsilon_\G: \mcO(\H\backslash \G)\to \C$ is bounded (for the reduced norm) and $\H$ is a compact quasi-subgroup of $\G$.
\end{enumerate}
\end{Cor}
\begin{proof} If $L^\infty(\H\backslash \G)$ is $\G$-injective,  \cite{AK24}*{Proposition 5.8} implies that $\H$  is a compact quasi-subgroup of $\G$. In that case, under the isomorphism \eqref{identification2}, we have seen at the end of Subsection \ref{coideals} that $\epsilon_{\H\backslash \G/\H}$ is simply the restriction of the counit $\epsilon_\G$. Thus, the equivalence $(1) \iff (2)$ follows from Theorem \ref{counitbounded}. The implication $(3)\implies (2)$ is trivial. To prove the converse, assume that $(2)$ holds and recall from the proof of Proposition \ref{compact quasi-subgroup} that the projection \eqref{expectation} extends uniquely to a normal conditional expectation $E: L^\infty(\G)\to L^\infty(\H\backslash \G)$. Similarly, the
projection $F: \mcO(\G)\to R_\G(\mcO(\H\backslash \G)): a \mapsto (\id \odot \varphi_\H)\Delta_\G(a)$ extends to a normal conditional expectation $F: L^\infty(\G)\to R_\G(L^\infty(\H\backslash \G))$. We have $E\circ F = F \circ E$, so $E\circ F$ defines a normal conditional expectation $L^\infty(\G)\to L^\infty(\H\backslash \G/\H)$. If $x\in \mcO(\H\backslash \G)$, we then have
$|\epsilon_\G(x)|=|\epsilon_\G(EF(x))| \le \|EF(x)\| \le \|x\|$, and $(3)$ is proven.
\end{proof}

\section{Outlook} The work in this paper suggests the following lines of research:
\begin{itemize}
    \item As discussed in Remark \ref{completions}, the $C^*$-algebraic theory of compact quantum hypergroups developed in \cite{CV99} has some deficits. Is it possible to develop a satisfactory theory that resolves these? The $C^*$-algebra $C_\mcO(\XGX)$ arising from an ergodic compact quantum group action $L^\infty(\X)\curvearrowleft \G$ should then fit perfectly in such a theory.
    \item As is apparent from Section \ref{coamenable}, the theory of equivariant correspondences as developed in \cites{DCDR24, DCDR25} provides a conceptual bridge between structural properties of the compact quantum hypergroup $\XGX$ and dynamical properties of the $\G$-$W^*$-algebra $L^\infty(\X)$. A detailed technical development of this connection, together with its implementation in concrete examples, will appear in forthcoming work.
\end{itemize}

\textbf{Acknowledgements.} The author thanks B. Anderson-Sackaney, K. De Commer and A. Skalski for useful discussions and feedback. The author also acknowledges a useful conversation with M. Daws.

\end{document}